 \RequirePackage{fix-cm}

 \documentclass[smallextended]{svjour3} %% onecolumn (second format) %% use this one for FCAA ! %%
 \smartqed  % flush right qed marks, e.g. at end of proof %

\usepackage{amssymb,amsmath,amsfonts,latexsym} %%

\usepackage{mathtools}
\usepackage{multirow,array}
\usepackage{graphicx}
\usepackage{subfig}
\usepackage{epstopdf}
\usepackage{float} %%% to use [H], [ht] etc. to fix places of figures
\usepackage{color, xcolor}
\allowdisplaybreaks
\usepackage{enumitem}
\usepackage{algorithm}
\usepackage{algorithmic}
\usepackage{tabularx}
\usepackage{makecell}       
\usepackage{booktabs}
\usepackage{diagbox}
\usepackage{caption}
\usepackage[all]{xy}
\usepackage{hyperref}
   
\makeatletter
\def\subclassname{{\bfseries MSC (2020)}\enspace}
\makeatother

\newcommand{\mred}[1]{\textcolor{red}{#1}}

\journalname{Fract. Calc. Appl. Anal.} %%%%% and FCAA-logo will appear on the right %%

\begin{document}

 %%% Pls. specify the kind of the article:
 %%% This is: %% ORIGINAL PAPER %%%% 
 %%% or: REVIEW PAPER / SHORT PAPER, etc. %%%%%%%%

%%%% Title of article for FCAA %%%%%%%%%%%%%%%%%%%%%%%%
\title{Exploring the Scaling Crossover of Generalized Jeffreys-Type Law: Stochastic Modeling, Analysis, and High-Accuracy Numerical Computation}

\titlerunning{Exploring the Scaling Crossover of Generalized Jeffreys-Type Law  \dots}
%% if too long for running head - use the text from 1st line !%

%%%%% authors:
\author{
        Fugui Ma$^1$ %%% 1 %%% F. G. Ma  %% ORCID ID: 0009-0005-7070-5860 
 }

\authorrunning{F.~G. Ma} %%  Short form of author list % if too long for running head

%%% affiliations:
\institute{Fugui Ma$^{1,*}$ 
\at 
School of Mathematical Sciences, Peking University, No. 5 Yiheyuan Road, Haidian District, Beijing 100871, China. \\
\email{Mafugui@math.pku.edu.cn} $^*$ corresponding author.          
}

%%% Dates:
\date{Received: 17 December 2025 / Revised: 23 November 2026 / Accepted: ......}
% These dates will be entered by the Editor (EiC, V. Kiryakova) or authors/ system %

%%%%% For Production Dept.: variants of COPYRIGHT notice to appear: %%%%%%
% if Open Access option chosen, it is as:
   %% "\copyright The Author(s) 2022" % (or next year..) %
% if No Open Access, the TCA form signed to FCAA journal is used, and it appears:
   %% "\copyright Diogenes Co. Ltd. 2022" %% or 2021, or next year %%%%%
%%%%%%%%%%%%%%%%%%%%%%%%%%%%%%%%%%%%%%%%%%%%%%%%%%%%%%%%%%%%%%%%%%%%%%%%%%

\maketitle

%%%%%%%%%%%%%%%%%%%%%%%%%%%%%%%%%%%%%%%%%%%%%%%%%%%%%%
\begin{abstract}
The generalized Jeffreys-type law is formulated as a multi-term time-fractional Jeffreys-type equation (MT-TFJE), whose dynamics exhibit rich scaling crossover phenomena associated with different diffusion mechanisms. In this work, we provide a microscopic stochastic interpretation of the MT-TFJE through a continuous-time random walk (CTRW) framework with a generalized waiting-time distribution, elucidating the physical mechanisms underlying the scaling crossover. Based on the Laplace-Volterra formulation and spectral analysis, we establish the well-posedness and stability of the MT-TFJE and derive Sobolev regularity estimates under appropriate assumptions on the initial data and source term. We further develop a high-accuracy numerical scheme, termed CIM-CLG, which combines a contour integral method for temporal approximation with the  Chebyshev-Legendre Galerkin method for spatial discretization. The proposed method achieves exponential temporal convergence under the stated regularity of the governing equation, together with high-order spatial accuracy. The temporal approximation involves $\mathcal{O}(N)$ independent contour evaluations, where $N$ is the number of contour points, while the FFT-based transforms have a computational cost of $\mathcal{O}(M^n\log M)$ for spatial dimensions $n=1,2$, with $M$ denoting the number of modes per spatial dimension. The resulting algorithm is naturally parallelizable. Detailed implementation procedures and rigorous error estimates are presented. Extensive numerical experiments in one and two dimensions demonstrate the efficiency, accuracy, and robustness of the proposed method, providing a computationally effective framework for generalized nonlocal diffusion problems.

%%%%%%%%%% Enter suitable key words and phrases %%% examples:

%%% These are examples only. Pls. use MSC 2020 for suitable topic numbers %%%%%%
\end{abstract} %%%%%%%%%%%%%

%%%%%%%% begin papers' body %%%%%%%%%%%%%%%%%%%%%%%%%%%%%

%%%%%%%%%%%% Section 1 %%%%%%%%%%%%%%%%%%%%%%%%%%
\section{Introduction}
Heat and particle transport in media with memory can involve several relaxation scales. In some complex media, the mean-squared displacement (MSD) of diffusing particles follows different power laws at short and long times; this change in the scaling exponent is called scaling crossover. Experimental observations include the transition from subdiffusion to Fickian yet non-Gaussian diffusion in colloidal suspensions (see \cite{Pastore2021}) and from subdiffusion to normal diffusion in water on protein surfaces (see \cite{Tan2018}). A ballistic-to-diffusive crossover has been studied in a weakly interacting Fermi gas (see \cite{Lloyd2024}), and Lévy-walk superdiffusion in active suspensions can masquerade as a ballistic-to-diffusive crossover (see \cite{Mukherjee2021}). In glassformers, displacement distributions can also cross from non-Gaussian to Gaussian within the subdiffusive regime (see \cite{Srinivasan2024}). These experiments show that scaling crossover is a measurable feature of transport in complex media. Understanding this phenomenon requires a connection between microscopic particle dynamics and macroscopic diffusion laws that captures both the limiting MSD exponents and the transition between them.

Jeffreys-type laws provide such a connection at the constitutive level. The classical Jeffreys law extends Fick's instantaneous relation by adding flux relaxation and a time derivative of the concentration gradient; the Cattaneo law retains only flux relaxation. Coupling the Jeffreys law with mass conservation gives the Jeffreys-type diffusion equation (JE), and removing the gradient-rate term reduces it to the Cattaneo law (see, e.g.,~\cite{Joseph89,Rukolaine2013}). This limiting relation connects the Jeffreys framework to fractional Cattaneo and telegrapher studies of stochastic representations, diffusion scaling, and first-passage statistics (see, e.g.,~\cite{Awad20b,Masoliver16,Sandev2025,Tomovski2025}). On the constitutive side, Awad~\cite{Awad2020} gave sufficient conditions for nonnegative solutions of the parabolic dual-phase-lag equation and proposed a fractional Jeffreys-type extension. Awad et al.~\cite{Awad2021} connected the time-fractional Jeffreys equation (TFJE) to a continuous-time random walk (CTRW) with a generalized waiting-time law. Furthermore, Awad et al.~\cite{Awad25} examined a complementary extension of the Jeffreys-type law that incorporates a Riemann-Liouville fractional integral into the constitutive law, investigating propagator positivity, MSD asymptotics, and underdamped oscillations of the temporal amplitude. These studies form a foundation for Jeffreys-type diffusion laws and have largely focused on microscopic stochastic modeling and on the derivation and interpretation of macroscopic Jeffreys-type diffusion laws. However, research on the bounded-domain Dirichlet problem for the TFJE itself and its low-regularity numerical approximation remains limited.

A recent extension \cite{Bazhlekova24} introduced the multi-term time-fractional Jeffreys equation (MT-TFJE), which allows several nonlocal relaxation scales in both the flux and the gradient memory, 
\begin{equation}\label{eq:problem02}
\Bigg(1+a{}_{0}^{RL}\mathfrak{D}_{t}^{\alpha}+\sum_{k=1}^{K}a_k{}_{0}^{RL}\mathfrak{D}_{t}^{\alpha_k}\Bigg)
\partial_t p=\mathcal{D}{}_{0}^{RL}\mathfrak{D}_{t}^{1-\gamma}\Bigg(1+b{}_{0}^{RL}\mathfrak{D}_{t}^{\beta}
+\sum_{j=1}^{J}b_j{}_{0}^{RL}\mathfrak{D}_{t}^{\beta_j}\Bigg)\Delta p,
\end{equation}
where $_{0}^{RL}\mathfrak{D}_{t}^{s}$ denotes the Riemann-Liouville fractional derivative, defined as $_{0}^{RL}\mathfrak{D}_{t}^{s}p(t):=\frac{1}{\Gamma(1-s)}\frac{\mathrm{d}}{\mathrm{d}t}\int_{0}^{t}(t-\tau)^{-s}p(\tau)\mathrm{d}\tau$ with $s\in(0,1)$;  $\mathcal{D}>0$ and 
\begin{equation}\label{eq:parameters}
\left\{\begin{aligned}
&0<\alpha_k<\alpha\leq1,~ 0<\beta_j<\beta\leq1,~ 0<\gamma\leq1,
~ a>0,~ b>0,~ a_k\geq0,~ b_j\geq0,\\
&k=1,2,...,K,~ j=1,2,...,J,~ K,J\in\mathbb{N}_{+}.
\end{aligned}\right.
\end{equation}
Under suitable parameter restrictions, Bazhlekova \cite{Bazhlekova24} studied the corresponding one-dimensional Cauchy problem through a subordination approach, proving that the fundamental solution is a spatial probability density function evolving in time. This shows that the MT-TFJE admits a probabilistic representation and that its memory structure can be made precise in terms of multinomial Mittag-Leffler kernels. However, this subordination result is primarily a representation of the fundamental solution on the whole line; it does not by itself give a solution theory for the Dirichlet problem on a bounded domain, Sobolev estimates for smooth and nonsmooth initial data, or a numerical approximation that remains accurate near the initial temporal singularity.

In this paper, we consider the generalized MT-TFJE by incorporating appropriate initial and boundary conditions along with a suitably defined source term, which leads to the following initial-boundary value problem (IBVP):
\begin{equation}\label{eq:problem}
\left\{\begin{aligned}
&_{0}^{RL}\mathfrak{I}_{t}^{1-\gamma}\Big(1+a~^{RL}_{0}\mathfrak{D}_{t}^{\alpha}
+\sum\nolimits_{k=1}^{K}a_k~_{0}^{RL}\mathfrak{D}_{t}^{\alpha_k}\Big)
   \partial_tp(t,r)\\
&{}=-\mathcal{D}\Big(1+b~_{0}^{RL}\mathfrak{D}_{t}^{\beta}
+\sum\nolimits_{j=1}^{J}b_j~_{0}^{RL}\mathfrak{D}_{t}^{\beta_j}\Big)\mathcal{A} p(t,r)+f(t,r),
                                    && \hspace{-0.5em}{\rm in}~(0,T]\times\Omega, \\
&p(0,r)=p_0(r),                        && \hspace{-0.5em}{\rm in}~\Omega, \\
&p(t,r)=0,                          && \hspace{-0.5em}{\rm on}~(0,T]\times\partial\Omega,
\end{aligned}\right.
\end{equation}
where $\Omega\subset\mathbb{R}^{n}$ ($n=1,2,3$) is a bounded convex domain with Lipschitz boundary, and $T>0$. The parameters satisfy the constraints given in \eqref{eq:parameters}. For $0<s<1$, the Riemann-Liouville fractional integral is defined by ${}_{0}^{RL}\mathfrak{I}_t^sp(t)=\frac{1}{\Gamma(s)}\int_0^t(t-\tau)^{s-1}p(\tau)\mathrm{d}\tau$ (see \cite{Kilbas2006}). The fractional initial-trace term vanishes when the corresponding trace is zero. $\mathcal{A}$ is a second-order,  self-adjoint, closed, positive definite spatial operator. The source term satisfies $f\in L^1(0,T;L^2(\Omega))$ in the Bochner sense. Additional assumptions are stated explicitly in the well-posedness, regularity, and convergence results.

\begin{figure}[thbp]
\centering
\includegraphics[width=0.88\textwidth]{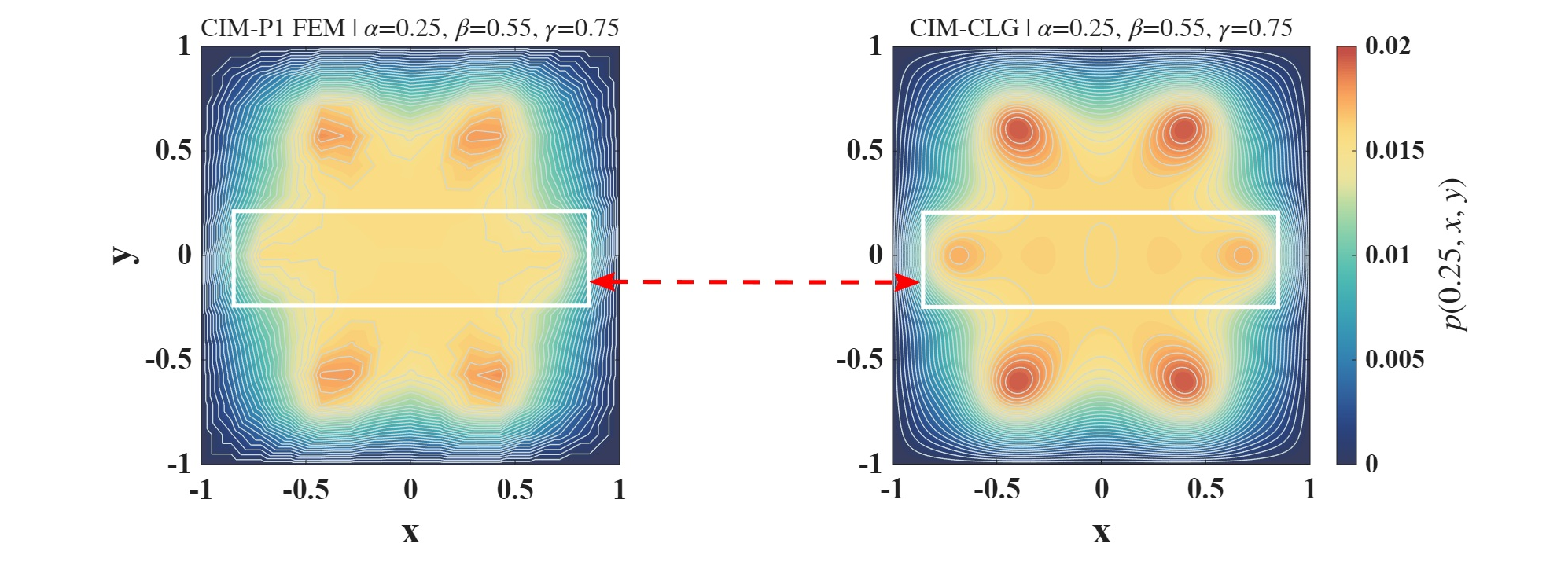}
\caption{Comparison of numerical solution contours for Example~\ref{Examp:2D} at $t=0.25$, using the same CIM temporal discretization and 169 spatial degrees of freedom for each spatial discretization method: (left) the CIM-FEM (P1 FEM) from~\cite{Ma2023a}; (right) the high-order CIM-CLG method proposed in this work.}
\label{fig:fem_clg_comparation}
\end{figure}

Problem \eqref{eq:problem} is motivated by the scaling crossover observed in experiments and by several gaps in the existing studies. $\mathbf{(i)}$ In the stochastic modeling, the additional lower-order terms in \eqref{eq:problem02} allow more flexible memory laws, but their influence on the limiting MSD exponents must be separated from their influence on the crossover between those limits. An explicit CTRW construction with the multi-term waiting-time law is needed; the framework in \cite{Sandev2018} and the subordination conditions in \cite{ChechkinSokolov2021} provide the basis for this construction and its admissibility analysis. $\mathbf{(ii)}$ In the mathematical analysis, bounded-domain theories for multi-term time-fractional diffusion equations (MT-TFDEs) $(\partial^{\alpha}_t+\sum_{j=1}^{m}q_j\partial_t^{\alpha_j})u+\mathcal{A}u=f$ (with $0<\alpha_m<\cdots<\alpha_1<\alpha<1$), heat conduction, and wave equations are available in \cite{Jin2015,Luchko2011,Maes23,Sin22} and references therein. For these equations, the Laplace symbol is a sum $\tilde\eta(z)=z^{\alpha}+\sum_{j}q_jz^{\alpha_j}$, which is itself a complete Bernstein function, and the resolvent $(\tilde\eta(z)+\mathcal{A})^{-1}$ separates the temporal symbol from the spatial operator. They do not cover \eqref{eq:problem}, where fractional operators act on both $\partial_t p(t,r)$ and $\mathcal{A} p(t,r)$. Therefore, the corresponding resolvent estimates do not apply directly without additional analysis of the quotient symbol. With $\mathcal{N}(z)$ and $\mathcal{D}(z)$ the fractional-polynomial symbols of the two parenthesized operators, the relevant symbol is the quotient $\eta(z)=z^{\gamma}\mathcal{N}(z)/\mathcal{D}(z)$ and both its sectorial behavior and its short- and long-time asymptotic orders must be controlled. These features are absent from the existing theory and directly affect the Sobolev estimates. $\mathbf{(iii)}$ In the numerical approximation, the initial temporal singularity limits standard high-order time discretizations. $L1$-type finite differences (see, e.g., \cite{Feng18,Jin2015,Lin07,Mustapha20}) and convolution quadrature (see \cite{Cuesta06,Lubich96} and references therein), as well as global polynomial approximation (see, e.g.,~\cite{Li2009}), require high-order regularity of the solution (e.g., $C^{2}(0,T)$) to attain their optimal convergence order, and may require $\mathcal{O}(\tilde{N}^2)$ arithmetic operations and $\mathcal{O}(\tilde{N})$ storage for time-stepping, where $\tilde{N}$ is the number of time discrete nodes. For \eqref{eq:problem}, however, the initial temporal singularity prevents the solution from reaching such regularity (see Theorem \ref{thm:theom1}). This makes methods with low regularity requirements essential.  Contour integral methods (CIM, see, e.g., \cite{Colbrook2022b,Li2023,Fernandez2006,Ma2023b,McLean2010,Talbot1979}) can achieve high temporal accuracy under limited regularity with $\mathcal{O}(N)$ memory cost (see Table~\ref{Tab:03} and Section \ref{subsec:cpu-time}). However, extending them to \eqref{eq:problem} requires estimates for the multi-term Jeffreys quotient and for its coupling with a high-accuracy spectral spatial discretization. As illustrated in Fig.~\ref{fig:fem_clg_comparation}, the spectral method resolves more pronounced spatial variations than the low-order discretization in this comparison under the same degrees of freedom.

The three unresolved aspects noted above naturally lead to the contributions presented below. We summarize their main novel features as follows.

$\bullet$ \textbf{Stochastic modeling and scaling crossover.}
At the microscopic level, we employ the CTRW model to give a new physical interpretation of the MT-TFJE \eqref{eq:problem02}. Combining asymptotic analysis with stochastic simulations, we compute the MSD and examine the scaling crossover both qualitatively and quantitatively. The results show that, although the multi-term equation offers greater flexibility in representing memory effects, its short- and long-time crossover is still governed by the leading memory orders and thus resembles that of the TFJE. This part also provides a theoretical basis for choosing admissible fractional-order parameter ranges in the subsequent analysis and numerical computation.

$\bullet$ \textbf{Bounded-domain solution theory and Sobolev estimates.}
To address the absence of a bounded-domain solution theory for the MT-TFJE, we formulate the problem in a generalized form, as given in Problem \eqref{eq:problem}. Using the Laplace transform method, we derive the mild-solution representation and obtain a priori estimates for the associated solution operators. These estimates yield well-posedness and sharp stability for $p_0\in L^2(\Omega)$ and $f\in L^\infty(0,T;L^2(\Omega))$. We further establish time-explicit Sobolev regularity estimates for smooth and nonsmooth initial data, together with source-term estimates. These estimates describe the temporal singularity and provide the a priori estimates needed for numerical stability and convergence. For $a_k=b_j=0$, the MT-TFJE reduces to the TFJE; thus the established solution theory also applies to the TFJE.

$\bullet$ \textbf{High-accuracy numerical approximation and analysis.} 
Building on our contour analysis in \cite{Ma2023a,Ma2023b}, we develop the fully discrete CIM-CLG method for Problem \eqref{eq:problem} by coupling a hyperbolic contour for the CIM in time with a Chebyshev-Legendre Galerkin (CLG) discretization in space. We establish the stability and convergence of CIM-CLG through discrete resolvent and spatial projection estimates for the multi-term operator. Under the stated source-transform assumptions, the temporal discrete error decays exponentially on fixed intervals $[t_0,\Lambda t_0]$, with $t_0>0$ and $\Lambda>1$, without requiring high-order temporal smoothness at the origin. Spatial convergence is governed by the available Sobolev regularity and becomes spectral for sufficiently smooth solutions. The method requires $N$ independent elliptic resolvent solves, which can be computed in parallel. At fixed spatial resolution, the computational work scales as $\mathcal{O}(N)$, and sequential accumulation of the contour sum requires $\mathcal{O}(1)$ auxiliary storage. FFT-based Chebyshev transforms cost $\mathcal{O}(M^n\log M)$ for $n=1,2$, where $M$ is the polynomial resolution per coordinate. Detailed implementation procedures and 1D and 2D experiments assess the accuracy, robustness, and efficiency of the method.

The rest of the paper is organized as follows. In Section \ref{sec:Modeling}, we develop the stochastic model and analyze the MSD. Section \ref{sec:SolutionT} establishes solution theory and Sobolev regularity. Section \ref{sec:CIMCLGs} presents and analyzes the CIM-CLG algorithm. Section \ref{sec:Numerical} reports numerical experiments, and Section \ref{sec:conclusions} gives concluding remarks.

%%%%%%%%%%%%%%% Section 2 %%%%%%%%%%%%%%%%%%%%%%%
\section{Mathematical Modeling: From Micro to Macro}
\label{sec:Modeling}
In this section, following the approaches in \cite{ChechkinSokolov2021,Ma2025,Sandev2015,Sandev2018,SokolovKlafter2005}, we derive the MT-TFJE \eqref{eq:problem02} from the CTRW framework in the diffusion limit. This stochastic formulation provides a microscopic physical picture of the macroscopic equation and its scaling crossover, and it shows the parameter ranges allowed by the probabilistic construction.

\subsection{CTRW model with Generalized Waiting-Time Law}
\label{subsec:modelI}
Consider a decoupled CTRW on $\mathbb{R}^n$ with independent, identically distributed (i.i.d.) waiting times and jumps, having densities $\psi(t)$ and $\lambda(r)$, respectively. Denote their Laplace and Fourier transforms by $\widehat\psi(z)$ and $\widetilde\lambda(\mathbf{k})$. If $p_0$ is the initial particle-position density, the Montroll-Weiss formula (see, e.g., \cite{Metzler00,Montroll65}) reads
\begin{equation}\label{Eq:M-W}
\widehat{\widetilde p}_{\mathrm{CTRW}}(z,\mathbf{k})
=\frac{1-\widehat\psi(z)}{z}
\frac{\widetilde p_0(\mathbf{k})}{1-\widehat\psi(z)\widetilde\lambda(\mathbf{k})},
\quad z>0,\quad \mathbf{k}\in\mathbb{R}^n.
\end{equation}
We consider the multi-term Laplace symbol
\begin{equation}\label{eq:denoteeta}
\eta(z):=z^{\gamma}\Big(1+az^{\alpha}+\sum\nolimits_{k=1}^{K}a_kz^{\alpha_k}\Big)
         \Big(1+bz^{\beta}+\sum\nolimits_{j=1}^{J}b_jz^{\beta_j}\Big)^{-1},
\end{equation}
with the parameters defined in \eqref{eq:parameters}. Following the rational waiting-time parametrization used in generalized CTRW theory (see \cite{Sandev2015,Sandev2018}), we set
\begin{equation}\label{eq:PDFTime}
\widehat\psi(z)=\widehat\psi_{\mathrm{rat}}(z)=(1+\eta(z))^{-1}.
\end{equation}
In contrast to the approximation $\widehat\psi(z)=\widehat{\psi}_{\mathrm{exp}}(z)=e^{-\eta(z)}\simeq1-\eta(z)$, which requires $\eta(z)\ll1$, the expression \eqref{eq:PDFTime} is used without a series expansion. While both approaches yield identical first-order expansions $1-\eta(z)$ as $\eta(z)\to0$, they define distinct waiting-time distributions beyond this limit.

Substituting \eqref{eq:PDFTime} into \eqref{Eq:M-W} yields the exact relation
\begin{equation}\label{eq:PCTRW}%\label{eq:CTRW-exact}
\eta(z)\left(\widehat{\widetilde p}_{\mathrm{CTRW}}(z,\mathbf{k})
-\frac{\widetilde p_0(\mathbf{k})}{z}\right)
=-\big(1-\widetilde\lambda(\mathbf{k})\big)
\widehat{\widetilde p}_{\mathrm{CTRW}}(z,\mathbf{k}).
\end{equation}
To distinguish the jump process from its continuum approximation, introduce a spatial scale $\varepsilon>0$ and the compatible waiting-time scale
\begin{equation*}
\widetilde\lambda_\varepsilon(\mathbf{k})
=\exp(-\varepsilon^2\mathcal{D}|\mathbf{k}|^2),
\qquad
\widehat\psi_{\varepsilon,\mathrm{rat}}(z)
=\frac{1}{1+\varepsilon^2\eta(z)},
\end{equation*}
where $\mathcal{D}>0$ is the diffusion coefficient. The resulting Montroll-Weiss Equation \eqref{Eq:M-W} is
\begin{equation*}
\widehat{\widetilde p}_\varepsilon(z,\mathbf{k})
=\frac{\eta(z)\widetilde p_0(\mathbf{k})}
{z\left[\eta(z)+\varepsilon^{-2}
\bigl(1-e^{-\varepsilon^2\mathcal{D}|\mathbf{k}|^2}\bigr)\right]}.
\end{equation*}
Since $\varepsilon^{-2}(1-e^{-\varepsilon^2\mathcal{D}|\mathbf{k}|^2})
\to\mathcal{D}|\mathbf{k}|^2$ as $\varepsilon\to0$, the limiting Fourier-Laplace density satisfies 
\begin{equation}\label{eq:CTRW}
\eta(z)\left(\widehat{\widetilde p}(z,\mathbf{k})
-\frac{\widetilde p_0(\mathbf{k})}{z}\right)
=-\mathcal{D}|\mathbf{k}|^2\widehat{\widetilde p}(z,\mathbf{k}).
\end{equation}
Equation~\eqref{eq:CTRW} retains the initial datum explicitly and defines the MT-TFJE~\eqref{eq:problem02} through its Laplace-Volterra formulation. Its classical Riemann-Liouville representation additionally requires sufficient regularity and compatible fractional initial traces; see \cite{Bazhlekova24,Kilbas2006}. This derivation identifies the MT-TFJE at the Fourier- Laplace transform level. At fixed $\varepsilon>0$, the exact evolution retains the spatial jump operator; the Laplacian emerges as $\varepsilon\to0$. The bounded-domain problem \eqref{eq:problem} extends this macroscopic equation by imposing the prescribed initial and boundary conditions.

For a stochastic interpretation, \eqref{eq:PDFTime} must be the Laplace transform of a probability density. The following proposition gives sufficient parameter conditions for this admissibility.

\begin{proposition}\label{prop:PDF}
Assume the parameters in \eqref{eq:problem02} satisfy \eqref{eq:parameters} and 
\begin{equation}\label{eq:probabilistic-parameters}
0<\beta\leq\gamma,\quad \alpha+\gamma\leq1.
\end{equation}
Then the symbol $\eta(z)$ defined in \eqref{eq:denoteeta} is a complete Bernstein function satisfying $\eta(0^+)=0$ and $\eta(z)\to\infty$ as $z\to\infty$. Consequently, $\widehat\psi_{\mathrm{rat}}(z)$ is a Stieltjes function and the Laplace transform of a completely monotone probability density $\psi_{\mathrm{rat}}(t)$ on $(0,\infty)$, with $\psi_{\mathrm{rat}}(t)\geq0$ and $\int_0^\infty\psi_{\mathrm{rat}}(t)\mathrm{d}t=1$.
\end{proposition}

\begin{proof}
We prove that a nonnegative, normalized waiting-time density exists under the stated parameter constraints. Let $\mathcal{CMF}$, $\mathcal{SF}$, $\mathcal{BF}$ and $\mathcal{CBF}$ denote the sets of completely monotone functions (CMF), Stieltjes functions (SF), Bernstein functions (BF), and complete Bernstein functions (CBF) respectively (see, e.g., \cite{Awad2020,Awad2021,Bazhlekova24,Sandev2018}). The proof has two steps. We first show that $\eta$ is a CBF and hence that $\widehat\psi_{\mathrm{rat}}$ is a Stieltjes function. We then use its Stieltjes representation to construct a completely monotone waiting-time density and prove its normalization.

We use $z$ for a complex variable and $x>0$ for a real variable. All fractional powers are taken on the principal branch. Set $\mathcal{N}(z):=1+az^\alpha+\sum_{k=1}^{K}a_kz^{\alpha_k}$ and $\mathcal{D}(z):=1+bz^\beta+\sum_{j=1}^{J}b_jz^{\beta_j}$, so that $\eta(z)=z^\gamma\mathcal{N}(z)/\mathcal{D}(z)$. We first establish that $\eta\in\mathcal{CBF}$. Under \eqref{eq:parameters} and \eqref{eq:probabilistic-parameters}, we have $0<\alpha_k<\alpha<1$ and $0<\beta_j<\beta<1$. For $z=re^{i\theta}$ with $r>0$ and $0<\theta<\pi$, the nonnegative coefficients and the ordering of the fractional powers give
\begin{equation}
\left\{\begin{aligned}
&\operatorname{Im}\big(\mathcal{N}(z)\big)
=ar^\alpha\sin(\alpha\theta)
+\sum\nolimits_{k=1}^{K}a_kr^{\alpha_k}
\sin(\alpha_k\theta)\geq0,\\[-0.8ex]
&\operatorname{Im}\big(e^{-i\alpha\theta}\mathcal{N}(z)\big)
=-\sin(\alpha\theta)
+\sum\nolimits_{k=1}^{K}a_kr^{\alpha_k}
\sin((\alpha_k-\alpha)\theta)\leq0.
\end{aligned}\right.
\end{equation}
Since the terms of $\mathcal N(z)$ lie in a sector of opening $\alpha\theta<\pi$ and include the positive constant term, $\mathcal N(z)\neq0$ and $0\leq\arg(\mathcal{N}(z))\leq\alpha\theta$. Similarly, 
\begin{equation}
\left\{\begin{aligned}
&\operatorname{Im}\big(\mathcal{D}(z)\big)
=br^\beta\sin(\beta\theta)
+\sum\nolimits_{j=1}^{J}b_jr^{\beta_j}
\sin(\beta_j\theta)\ge0,\\[-0.8ex]
&\operatorname{Im}\big(e^{-i\beta\theta}\mathcal{D}(z)\big)
=-\sin(\beta\theta)
+\sum\nolimits_{j=1}^{J}b_jr^{\beta_j}
\sin((\beta_j-\beta)\theta)\le0.
\end{aligned}\right.
\end{equation}
Hence $\mathcal D(z)\neq0$ and $0\le\arg(\mathcal{D}(z))\le\beta\theta$. It follows from \eqref{eq:probabilistic-parameters} that $0\leq(\gamma-\beta)\theta
\leq\arg\eta(z)=\gamma\theta+\arg(\mathcal{N}(z))-\arg(\mathcal{D}(z))\leq(\gamma+\alpha)\theta\leq\theta$. The corresponding properties in the lower half-plane follow by complex conjugation. Since $\eta$ is analytic in $\mathbb C\setminus(-\infty,0]$ and $\eta(x)>0$ for $x>0$, the analytic characterization of CBFs (see \cite{Schilling2009}) yields $\eta\in\mathcal{CBF}$. Consequently, $1+\eta\in\mathcal{CBF}$ and $\widehat\psi_{\mathrm{rat}}=(1+\eta)^{-1}\in\mathcal{SF}\subset\mathcal{CMF}$. This argument applies to all admissible nonnegative coefficients, including the cases where some or all additional coefficients vanish.

We now construct the waiting-time density and verify its normalization. Along the positive real axis,  $\eta(x)\sim x^\gamma$ as $x\to0^+$ and $\eta(x)\sim\frac{a}{b}x^{\alpha+\gamma-\beta}$ as $x\to+\infty$. Since $\alpha+\gamma-\beta>0$, we obtain $\lim_{x\downarrow0}\widehat\psi_{\mathrm{rat}}(x)=1$ and $\lim_{x\to+\infty}\widehat\psi_{\mathrm{rat}}(x)=0$. We use these limits to identify the constant terms in the Stieltjes representation $\widehat\psi_{\mathrm{rat}}(z)=\frac{c_0}{z}+c_\infty+\int_0^\infty\frac{1}{z+s}\mathrm{d}\mu(s)$ with $c_0$, $c_\infty\geq0$, and $\mu$ is a positive measure satisfying $\int_{0}^{\infty}\frac{\mu(\mathrm{d}s)}{1+s}<\infty$ (see \cite[Definition 2.1]{Schilling2009}). The limit at infinity gives $c_\infty=0$, while the finite limit at zero forces $c_0=0$. Thus, by the monotone convergence theorem, $\widehat{\psi}_{\mathrm{rat}}(z)=\int_{0}^{\infty}\frac{\mu(\mathrm{d}s)}{z+s}$ and $\int_{0}^{\infty}\frac{\mu(\mathrm{d}s)}{s}=1$. Define $\psi(t)=\psi_{\mathrm{rat}}(t):=\int_{0}^{\infty}e^{-st}\mu(\mathrm{d}s)\geq0$. The preceding integrability condition ensures that this function and all its derivatives are finite for $t>0$. Moreover, $(-1)^m\psi_{\mathrm{rat}}^{(m)}(t)=\int_{0}^{\infty}s^m e^{-st}\mu(\mathrm ds)\geq0$, $m=0,1,\cdots$. Thus $\psi_{\mathrm{rat}}$ is completely monotone. For $\operatorname{Re}z>0$, the following double integral converges absolutely because $\int_{0}^{\infty}\int_0^\infty|e^{-(z+s)t}|\mathrm{d}t\mu(\mathrm{d}s) =\int_{0}^{\infty}\frac{\mu(\mathrm ds)}{\operatorname{Re}z+s}<\infty$. Fubini's theorem therefore gives
\begin{equation}
\begin{aligned}
\int_0^\infty e^{-zt}\psi_{\mathrm{rat}}(t)\,\mathrm{d}t
=\int_0^\infty
\int_0^\infty e^{-(z+s)t}\,\mathrm{d}t
\mathrm{d}\mu(s)
=\int_0^\infty\frac{1}{z+s}\,\mathrm{d}\mu(s)
=\widehat\psi_{\mathrm{rat}}(z).
\end{aligned}
\end{equation}
Hence $\psi(t)$ has the prescribed Laplace transform. By Tonelli's theorem (see \cite[Theorem 2.37(a)]{Folland1999}), $\int_0^\infty\psi_{\mathrm{rat}}(t)\mathrm{d}t
=\int_0^\infty\frac{1}{s}\mathrm{d}\mu(s)=1$. Thus $\psi_{\mathrm{rat}}(t)$ is a completely monotone probability density. \qed
\end{proof}

\subsection{The MSD}\label{subsec:MSD}
Let $\mathbf{X}(t)$ denote the $n$-dimensional stochastic process generated by the generalized CTRW in Sec. \ref{subsec:modelI} starting from the origin. For the Gaussian jump law $\widetilde\lambda(\mathbf{k})=\exp(-\mathcal{D}|\mathbf{k}|^2)$, the mean-squared jump length is $2n\mathcal{D}$. A standard way to characterize particle spreading is through the MSD, i.e., $M(t):=\langle|\mathbf{X}(t)|^2\rangle$ (see \cite{LeVot2017}). Below, we determine its Laplace transform and investigate the short- and long-time asymptotic behavior.

For the MSD calculation, we assume that the particle starts at the origin, so that $p_0=\delta_0$ and $\tilde{p}_{0}(\mathbf{k})=1$. Recall that, provided the relevant moments are finite, they can be obtained by differentiating the Fourier transform of the propagator (see \cite{Klafter15}). In particular,
$\widehat{M}(z):=\mathcal{L}_{t\to z}\{M(t)\}(z)=-\Delta_{\mathbf{k}}\widehat{\widetilde{p}}_{\mathrm{CTRW}}(z,\mathbf{k})|_{\mathbf{k}=0}$. Moreover, since both the initial distribution $p_0(r)$ and the jump distribution $\lambda(r)$ are normalized, we have
$\widetilde{p}_0(0)=\widetilde\lambda(0)=1$. It follows from \eqref{eq:PCTRW} that $\widehat{\widetilde{p}}_{\mathrm{CTRW}}(z,0)=\frac{\eta(z)\widetilde{p}_0(0)}{z[\eta(z)+1-\widetilde\lambda(0)]}=\frac{1}{z}$. Taking the inverse Laplace transform confirms that $\langle \mathbf{X}^0(t)\rangle=1$, and hence the CTRW propagator is normalized.

Next, we calculate the MSD of $\mathbf{X}(t)$. Differentiating the exact expression \eqref{eq:PCTRW} at $\mathbf{k}=0$, we obtain
\begin{equation}
\widehat{M}(z)
=-\left.\Delta_{\mathbf{k}}\frac{\eta(z)}{z\,[\eta(z)+1-e^{-\mathcal{D}|\mathbf{k}|^2}]}\right|_{\mathbf{k}=0}
=\frac{2n\mathcal{D}}{z\eta(z)}.
\end{equation}
Writing $\sigma:=\alpha+\gamma-\beta$, we obtain from \eqref{eq:denoteeta}
\begin{equation}\label{MMsD}
\begin{aligned}
\widehat{M}_{\mathrm{rat}}(z)
=\frac{2n\mathcal{D}
\big(1+bz^\beta+\sum_{j=1}^{J}b_jz^{\beta_j}\big)}
{z^{\gamma+1}\big(1+az^\alpha+\sum_{k=1}^{K}a_kz^{\alpha_k}\big)}
\sim
\begin{cases}
\displaystyle\frac{2n\mathcal{D} b}{a}\,z^{-1-\sigma},
&z\to+\infty,\\[0.4em]
\displaystyle 2n\mathcal{D}\,z^{-1-\gamma},
&z\to0^+.
\end{cases}
\end{aligned}
\end{equation}
For the rational waiting-time law, since the jumps are independent and centered, $M(t)=2n\mathcal{D}\mathbb{E}[N(t)]$, where $N(t)$ is the nondecreasing renewal counting process. Thus $M(t)$ is nondecreasing. Applying Karamata's Tauberian theorem for Laplace-Stieltjes transforms and its counterpart at the origin (see \cite[Chapter XIII, Section 5]{Feller71} or \cite[Theorem 1.5.7]{Applebaum2009}), we obtain
\begin{equation}\label{asp:MSD}
M_{\mathrm{rat}}(t)\sim
\begin{cases}
\displaystyle\frac{2n\mathcal{D}b}
{a\,\Gamma(1+\sigma)}\,t^{\alpha+\gamma-\beta},
&t\to0^+,\\[0.7em]
\displaystyle\frac{2n\mathcal{D}}
{\Gamma(1+\gamma)}\,t^\gamma,
&t\to+\infty.
\end{cases}
\end{equation}

Based on \eqref{asp:MSD}, the short- and long-time scaling exponents are $\alpha+\gamma-\beta$
and $\gamma$, respectively. Under the sufficient admissibility assumptions of Prop.~\ref{prop:PDF}, we have
$0<\alpha+\gamma-\beta<1$ and $0<\gamma<1$. The following cases illustrate the possible behaviors:
\begin{itemize}[label=$\bullet$]
\item If $\alpha=\beta$, then $\alpha+\gamma-\beta=\gamma$: the short- and long-time subdiffusive exponents coincide, although the prefactors and intermediate behavior may differ.
\item If $\alpha<\beta$, then $\alpha+\gamma-\beta<\gamma$, giving accelerating subdiffusion in the sense of an increase between the two limiting exponents.
\item If $\alpha>\beta$, then $\alpha+\gamma-\beta>\gamma$, giving retarded subdiffusion.
\end{itemize}
These classifications concern the limiting powers and do not assert a monotone change in the effective MSD exponent throughout the transition. The ratio $b/a$ affects the short-time prefactor. The ordering of $\alpha$ and $\beta$ determines the relative sizes of the limiting exponents, while the additional lower-order terms can modify the crossover profile without changing either exponent.

Regimes with $\gamma=1$ or $\alpha+\gamma-\beta>1$ fall outside Prop.~\ref{prop:PDF} and require a separate admissibility analysis. For example, the classical choice $\alpha=\beta=\gamma=1$, with all additional multi-term coefficients set to zero, gives normal MSD scaling in both limits. Its rational waiting-time law is admissible when $b\geq a>0$.
\begin{remark}
Under the assumptions of Prop.~\ref{prop:PDF}, consider the alternative exponential-form waiting-time law $\widehat\psi_{\mathrm{exp}}(z)=e^{-\eta(z)}$. For particles starting at the origin with the same Gaussian jump law, the exact Montroll-Weiss formula gives $\widehat M_{\mathrm{exp}}(z)=\frac{2n\mathcal D}{z(e^{\eta(z)}-1)}$ and $\widehat M_{\mathrm{rat}}(z)=\frac{2n\mathcal D}{z\eta(z)}$. As $z\to0^+$, $\eta(z)\sim z^\gamma$, so $\frac{\widehat M_{\mathrm{exp}}(z)}{\widehat M_{\mathrm{rat}}(z)}=\frac{\eta(z)}{e^{\eta(z)}-1}\to1$. Both MSDs are nondecreasing, and the same Tauberian argument gives $M_{\mathrm{exp}}(t)\sim M_{\mathrm{rat}}(t)
\sim\frac{2n\mathcal D}{\Gamma(1+\gamma)}t^\gamma$ as $t\to+\infty$. Their short-time behavior differs. Since $\eta(z)\sim(a/b)z^{\alpha+\gamma-\beta}\to+\infty$, $\widehat M_{\mathrm{exp}}(z)\sim\frac{2n\mathcal D}{z}e^{-\eta(z)}$ and $\widehat M_{\mathrm{rat}}(z)\sim\frac{2n\mathcal D b}{a}\,z^{-1-(\alpha+\gamma-\beta)}$ as $z\to+\infty$. This comparison concerns the fixed-scale CTRWs. Under compatible diffusive scaling, both waiting-time families lead to the continuum equation \eqref{eq:problem02}.
\end{remark}

To numerically illustrate and compare the two waiting-time laws, we consider four parameter sets with $n=\mathcal D=1$:
\begin{equation}\label{eq:PPar}
\begin{aligned}
\text{(a)}\quad &(\alpha,\beta,\gamma)=(0.25,0.55,0.75), &
\text{(b)}\quad &(\alpha,\beta,\gamma)=(0.50,0.35,0.45),\\
\text{(c)}\quad &(\alpha,\beta,\gamma)=(0.75,0.15,0.15), &
\text{(d)}\quad &(\alpha,\beta,\gamma)=(1.00,1.00,1.00).
\end{aligned}
\end{equation}
For (a)-(c), we take $K=J=2$, $(a,b)=(1,0.5)$, $(\alpha_1,\alpha_2)=(0.10,0.18)$, $(\beta_1,\beta_2)=(0.05,0.10)$, $(a_1,a_2)=(0.35,0.20)$, and $(b_1,b_2)=(0.25,0.15)$. These cases satisfy the sufficient admissibility conditions of Prop.~\ref{prop:PDF}. For (d), we set $a_k=b_j=0$ and $(a,b)=(0.0125,0.025)$, yielding $\eta(z)=z/2+z/[2(1+0.025z)]$, which is a Bernstein function. Thus both waiting-time laws are admissible. The theoretical MSDs are computed by numerical inverse Laplace transformation using the Talbot method (see \cite{Talbot1979}). We also estimate the MSDs from $2\times10^4$ independent CTRW trajectories for each waiting-time law over $10^{-6}\le t\le10^6$. Their relative difference is defined by $E(t):=|M_{\mathrm{rat}}(t)-M_{\mathrm{exp}}(t)|/M_{\mathrm{rat}}(t)$.

Fig.~\ref{fig:msd-four-cases} illustrates the different scaling behaviors and the discrepancy between the two waiting-time laws. As predicted by \eqref{asp:MSD}, the rational MSD exhibits accelerating subdiffusion in (a), with the exponent increasing from $0.45$ to $0.75$, and retarded subdiffusion in (b) and (c), with the exponents decreasing from $0.60$ to $0.45$ and from $0.75$ to $0.15$, respectively. In (d), the rational MSD has normal scaling in both limits, with $M_{\mathrm{rat}}(t)\sim4t$ as $t\to0^+$ and $M_{\mathrm{rat}}(t)\sim2t$ as $t\to\infty$. The exponential-form waiting time in (d) is bounded below by $1/2$, so $M_{\mathrm{exp}}(t)=0$ and $E(t)=1$ for $t<1/2$.
In all four cases, the two MSDs differ markedly at short times and exhibit distinct intermediate-time behaviors, while their relative discrepancy tends to zero at long times. The Monte Carlo estimates show good agreement with the theoretical MSD curves.

\begin{figure}[htbp]
\centering
\begin{minipage}{0.48\linewidth}
\centering
\includegraphics[width=0.8\linewidth]{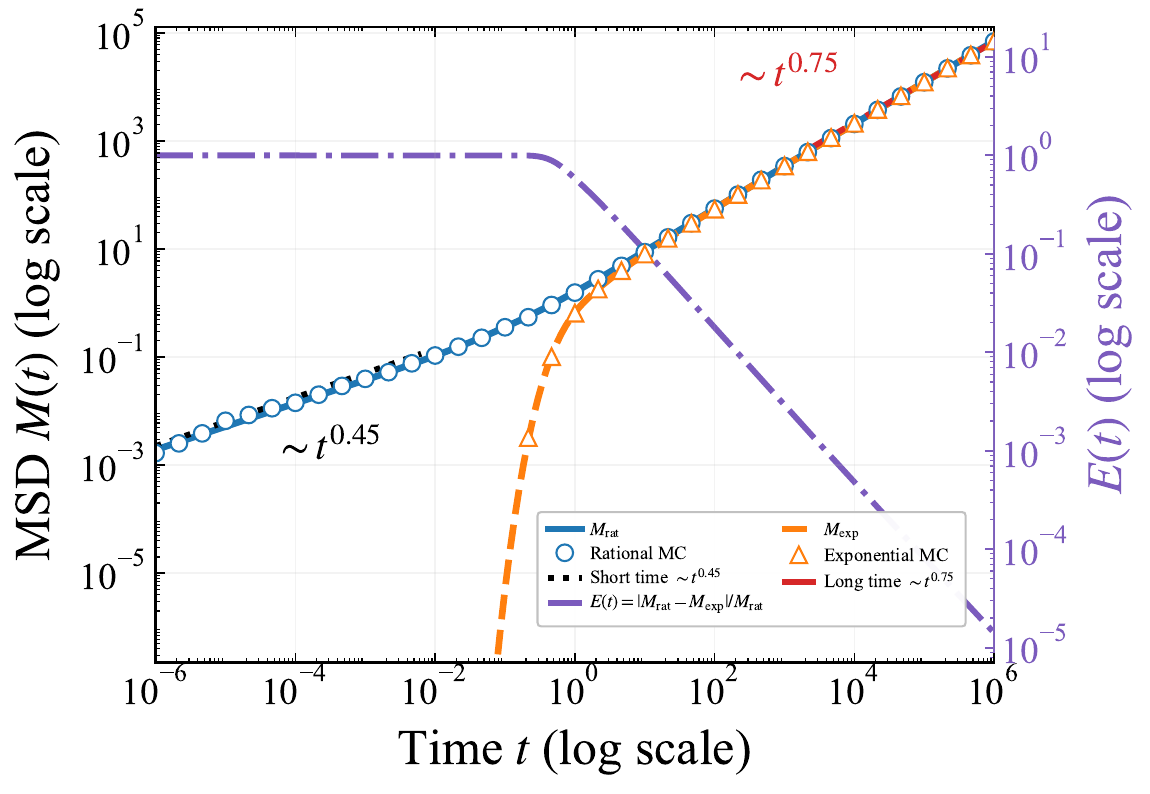}\\[-0.3ex]
\small (a) $(\alpha=0.25,\beta=0.55,\gamma=0.75)$
\end{minipage}%\hfill
\begin{minipage}{0.48\linewidth}
\centering
\includegraphics[width=0.8\linewidth]{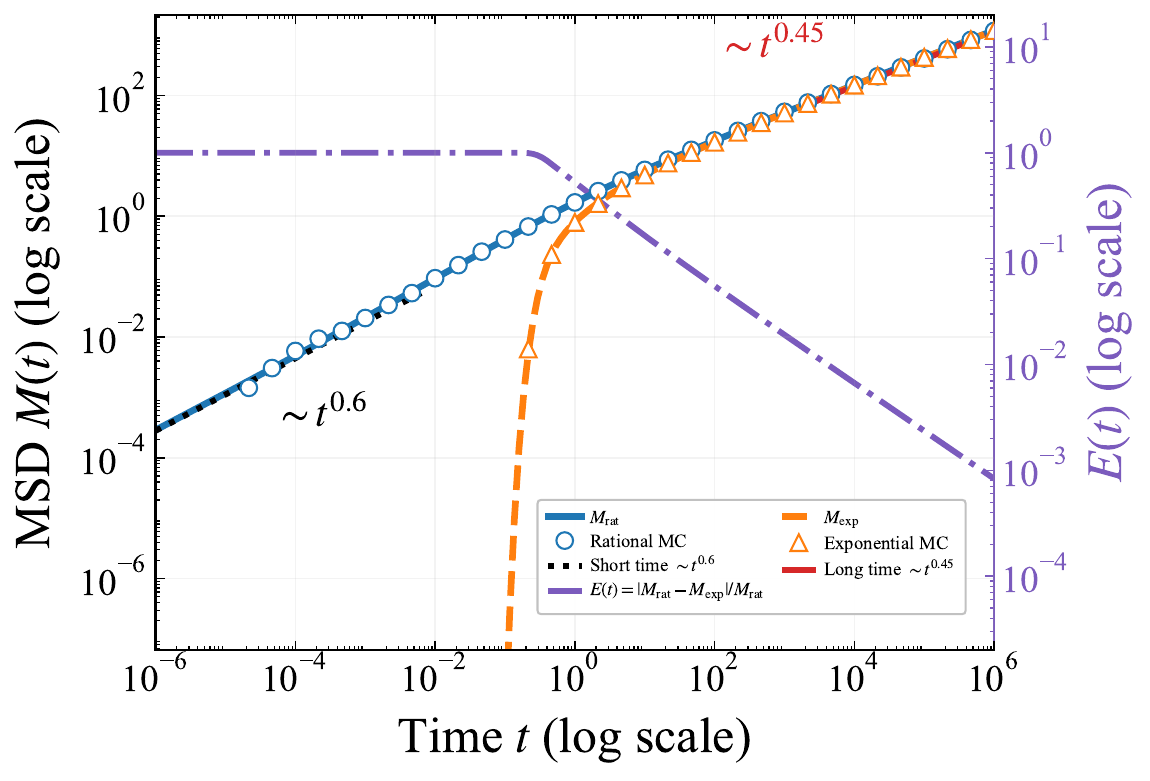}\\[-0.3ex]
\small (b) $(\alpha=0.50,\beta=0.35,\gamma=0.45)$
\end{minipage}
\par\smallskip
\begin{minipage}{0.48\linewidth}
\centering
\includegraphics[width=0.8\linewidth]{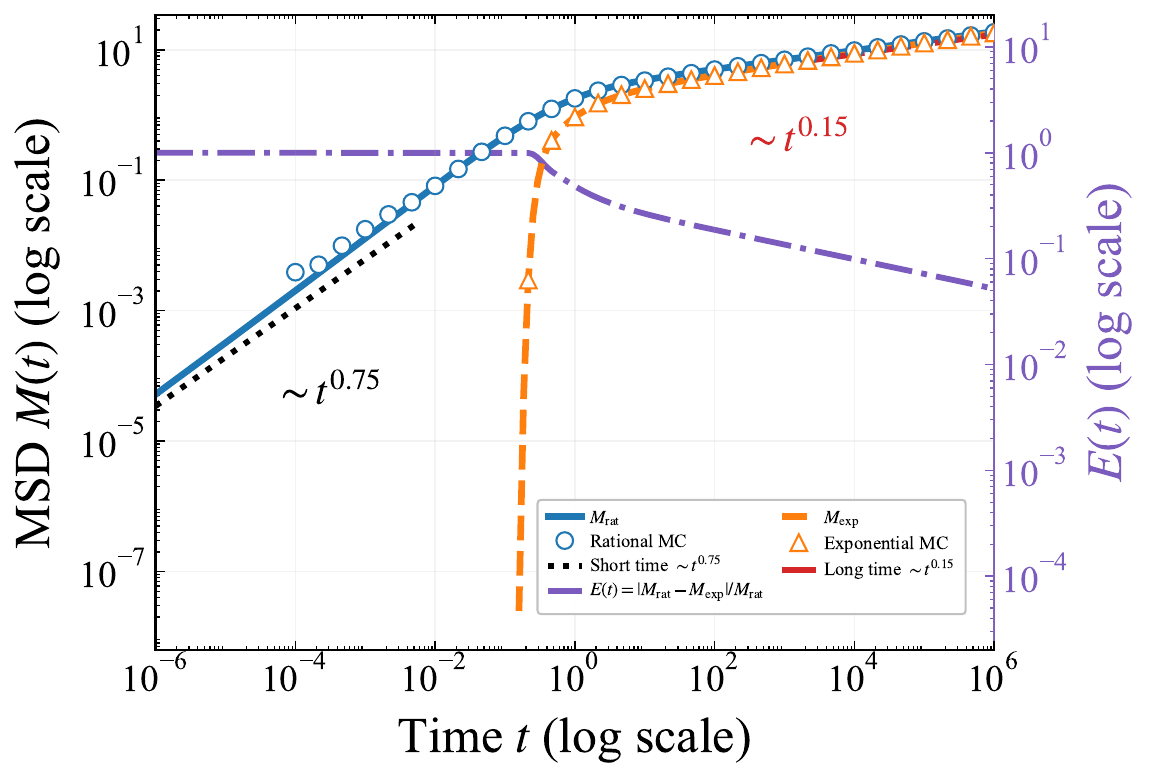}\\[-0.3ex]
\small (c) $(\alpha=0.75,\beta=0.15,\gamma=0.15)$
\end{minipage}%\hfill
\begin{minipage}{0.48\linewidth}
\centering
\includegraphics[width=0.8\linewidth]{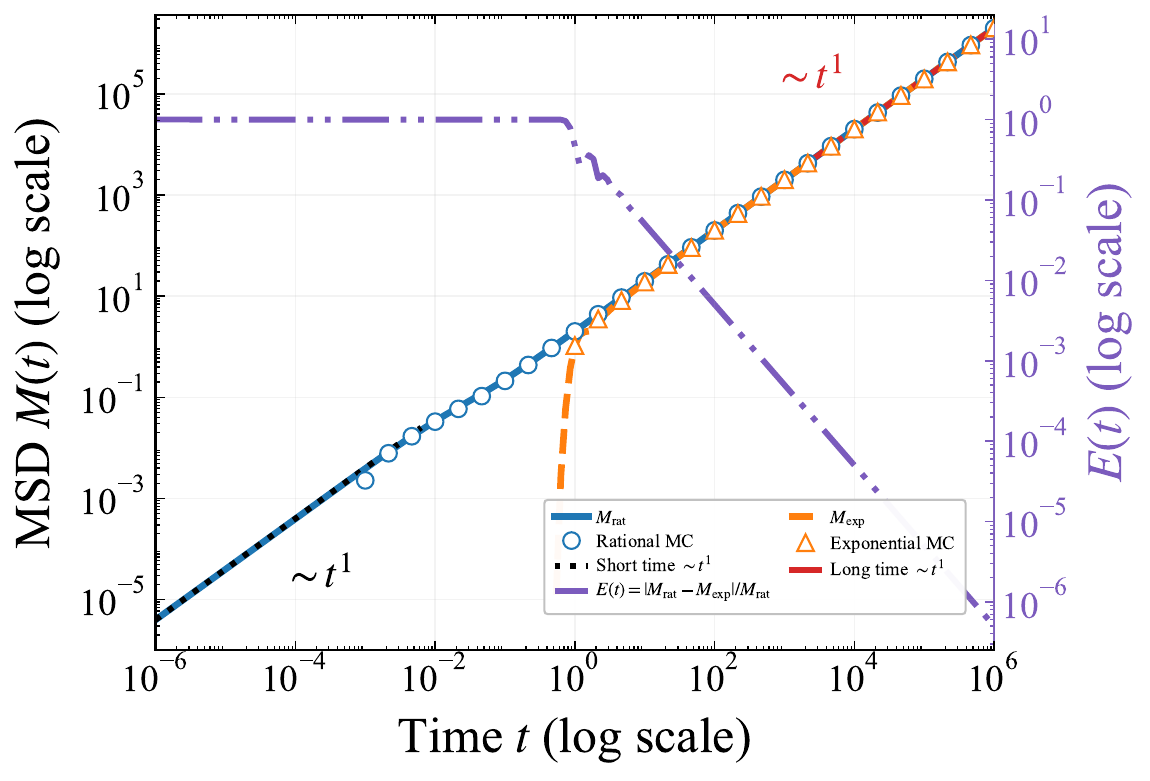}\\[-0.3ex]
\small (d) $(\alpha=1.00,\beta=1.00,\gamma=1.00)$
\end{minipage}
\caption{Comparison of the rational and exponential MSDs for four parameter sets, illustrating subdiffusive scaling crossovers and normal diffusion. Blue solid and orange dashed curves show $M_{\mathrm{rat}}(t)$ and $M_{\mathrm{exp}}(t)$ obtained by Talbot inversion, while open circles and triangles denote Monte Carlo estimates from $2\times10^4$ trajectories per waiting-time law. Dotted and dash-dotted segments indicate the short- and long-time asymptotics of the rational MSD. The purple dash-dotted curve shows the relative discrepancy $E(t)$ on the right-hand axis. The two MSDs exhibit different transient behaviors but share the same leading long-time scaling.}
\label{fig:msd-four-cases}
\end{figure}

%%%%%%%%  Section 3 %%%%%%%%%%%%%%%%%%%%%%%%%%%%%%%%%%%%%%%%%%%%%%%%%%%
\section{Solution theory}
\label{sec:SolutionT}
Having established the microscopic foundations, we now analyze the macroscopic IBVP \eqref{eq:problem} rigorously. We focus on regularity and well-posedness, which provide the foundation for the error estimates of the numerical schemes below.
\subsection{Additional notation and function spaces}
\label{sec2}
We use standard notation and results for Lebesgue spaces $L^{q}(\Omega)$ and Sobolev spaces $W^{m,q}(\Omega)$, where $1\leq q\leq \infty$ and $m\in\mathbb{N}$. After a standard rescaling, we set the diffusion coefficient $\mathcal{D}=1$. Let $\mathcal{A}:=-\Delta: H^{2}(\Omega)\cap H_{0}^{1}(\Omega)\to L^{2}(\Omega)$ be the negative Laplacian with homogeneous Dirichlet boundary conditions. Let $(\lambda_{j},\varphi_{j})_{j\in\mathbb{N}}$ denote its eigenvalues in non-decreasing order, together with the corresponding eigenfunctions normalized in $L^{2}(\Omega)$. For $s\geq0$,  we define the fractional Sobolev space $\dot{H}^{s}(\Omega):=\{v\in L^{2}(\Omega): \sum_{j=1}^{\infty}\lambda^{s}_{j}|(v,\varphi_{j})|^{2}<\infty\}$, equipped with the inner product $(u,v)_{\dot{H}^{s}(\Omega)}:=\sum_{j=1}^{\infty}\lambda^{s}_{j}(u,\varphi_{j})(v,\varphi_{j})$ and norm $\|v\|_{\dot{H}^{s}(\Omega)}:=(v,v)_{\dot{H}^{s}(\Omega)}^{1/2}$, for all $u$, $v\in\dot{H}^{s}(\Omega)$ (see  \cite{Adams2003}). These spaces form a Hilbert scale. In particular, $\dot{H}^{0}(\Omega)=L^{2}(\Omega)$, $\dot{H}^{1}(\Omega)=H^{1}_{0}(\Omega)=\{v\in H^{1}:v|_{\partial\Omega}=0\}$, and $\dot{H}^{2}(\Omega)=H^{2}(\Omega)\cap H^{1}_{0}(\Omega)$.

For $\theta\in(\pi/2,\pi)$ and $\delta>0$, we define the sector and the truncated sector
\begin{itemize}[label=$\bullet$]
  \item $\Sigma_{\theta}:=\{z\in\mathbb{C}\setminus\{0\}, |\arg(z)|<\theta\}$,\quad  $\Sigma_{\theta,\delta}:=\{z\in\mathbb{C}: |z|\geq\delta, |\arg(z)|<\theta\}$.
\end{itemize}
We also use the contour
\begin{itemize}[label=$\bullet$]
  \item $\Gamma_{\theta,\delta}:=\big\{z\in\mathbb{C}:|z|=\delta,|\arg(z)|\leq \theta\big\}\cup\big\{z=\bar{r}e^{\pm i\theta},0<\delta\leq \bar{r}<\infty\big\},$
\end{itemize}
oriented with an increasing imaginary part, where $i^2=-1$.

With the above preparations in place, we now state a useful lemma that will be used repeatedly in constructing the solution theory. Throughout, $\|\cdot\|_{\Psi}$ denotes the operator norm from $L^2(\Omega)\rightarrow L^2(\Omega)$.

\begin{lemma}[Resolvent estimate \cite{Arendt2011}] \label{Lem:Resolvent}
Let $\mathcal{A}$ be the negative Laplacian with homogeneous Dirichlet boundary conditions. For every $\vartheta\in\big(0,\pi\big)$, there is a constant $C>0$ such that the resolvent $(zI+\mathcal{A})^{-1}: L^{2}(\Omega)\rightarrow L^{2}(\Omega)$ satisfies
\begin{equation}\label{eq:resolvent}
\left\|(zI+\mathcal{A})^{-1}\right\|_{\Psi}
\leq C(1+|z|)^{-1}, \qquad \forall~ z\in\Sigma_{\vartheta}.
\end{equation}
\end{lemma}

Throughout this work, the symbols $c$ and $C$ denote generic positive constants. They may depend on the model parameters but not on the functions under consideration, and their values may change from line to line.

\subsection{Sobolev regularity of the solution}
This section establishes the well-posedness and Sobolev regularity of the solution to Problem~\eqref{eq:problem}.

For $0<s<1$, the Laplace transform of the Riemann-Liouville derivative satisfies $\mathcal{L}\{{}_{0}^{RL}\mathfrak D_t^s\phi\}(z)=z^s\widehat\phi(z)-\big[{}_{0}^{RL}\mathfrak{I}_t^{1-s}\phi\big](0^+)$ (see \cite{Kilbas2006}). Under the standard regularity assumptions adopted in \cite{Kilbas2006} (e.g., $\phi$ bounded near $t=0$), the fractional initial trace vanishes, since $[{}_{0}^{RL}\mathfrak{I}_t^{1-s}\phi](t)=O(t^{1-s})\to0$ $(t\to0^+)$, it reduces to $\mathcal{L}\{{}_{0}^{RL}\mathfrak{D}_t^s\phi\}(z)=z^s\widehat{\phi}(z)$. For nonsmooth solutions, however, this trace may be nontrivial or not well defined; hence $p(0)=p_0$ alone does not determine the fractional initial traces. We therefore adopt the Laplace-Volterra formulation to define the mild solution (see \cite{Bazhlekova24}), incorporating $p_0$ explicitly without additional fractional initial-trace conditions, namely
\begin{displaymath}\label{eq:Lapace-problem}
z^{\gamma-1}\Bigg(1+az^{\alpha}+\sum_{k=1}^{K}a_kz^{\alpha_k}\Bigg)\big(z\widehat{p}(z)-p_0\big)
+\Bigg(1+bz^{\beta}+\sum_{j=1}^{J}b_jz^{\beta_j}\Bigg)\mathcal{A}\widehat{p}(z)=\widehat{f}(z).
\end{displaymath}

Recall the definition of $\eta(z)$ from \eqref{eq:denoteeta} and introduce
\begin{equation}\label{eq:denoteH}
   \mathcal{H}(z):=z^{-1}\eta(z)\big(\eta(z)I+\mathcal{A}\big)^{-1}.
\end{equation}
Then the solution in the Laplace domain is
\begin{equation}\label{eq:integrand}
\widehat{p}(z)
=\mathcal{H}(z)p_0+\Big(z^{\gamma-1}+az^{\alpha+\gamma-1}+\sum\nolimits_{k=1}^{K}a_kz^{\alpha_k+\gamma-1}\Big)^{-1}
           \mathcal{H}(z)\widehat{f}(z).
\end{equation}
Applying the inverse Laplace transform and employing Cauchy's theorem together with Remark \ref{thm:remark1}, we obtain
\begin{equation}\label{eq:integral}
\begin{aligned}
p(t)
%=&~\frac{1}{2\pi i}\int_{\Gamma_{\theta,\delta}}e^{zt}\mathcal{H}(z)u_0
%   +\Big(z^{\gamma-1}+az^{\alpha+\gamma-1}+\sum_{k=1}^{K}a_kz^{\alpha_k+\gamma-1}\Big)^{-1}
%     \mathcal{H}(z)\widehat{f}(z)dz   \\
=\mathcal{E}(t)p_0+\int_{0}^{t}\mathcal{F}(s)f(t-s)\mathrm{d}s, \quad t>0,
\end{aligned}
\end{equation}
where the solution operators $\mathcal{E}(t)$ and $\mathcal{F}(t): L^{2}(\Omega)\rightarrow L^{2}(\Omega)$ are defined by
\begin{equation}\label{eq:solvers}
\mathcal{E}(t):=\frac{1}{2\pi i}\int_{\Gamma_{\theta,\delta}}e^{zt}\mathcal{H}(z)\mathrm{d}z,\quad
\mathcal{F}(t):=\frac{1}{2\pi i}\int_{\Gamma_{\theta,\delta}}e^{zt}g(z)\mathcal{H}(z)\mathrm{d}z,
\end{equation}
with
\begin{equation}\label{eq:dg}
g(z):=\Big(z^{\gamma-1}+az^{\alpha+\gamma-1}+\sum\nolimits_{k=1}^{K}a_kz^{\alpha_k+\gamma-1}\Big)^{-1}. 
\end{equation}
The representation \eqref{eq:integral} is referred to as the mild solution of Problem \eqref{eq:problem}.

Having derived the solution representation, we now study the analytic properties of $\eta(z)$, $\mathcal{H}(z)$, and the solution operators $\mathcal{E}(t)$ and $\mathcal{F}(t)$. These properties are key to a rigorous theory for the problem.

%++++++++++++++++++++++++++++++++++++++++++++++++++++++++++++++++++++++++++++++++++++++++++++++++++++++++
%                            Proposition 1  \eta_(z), W(z)
%++++++++++++++++++++++++++++++++++++++++++++++++++++++++++++++++++++++++++++++++++++++++++++++++++++++++
\begin{lemma}[Analyticity]\label{Lem:function}
Let $\eta(z)$ and $\mathcal H(z)$ be defined by \eqref{eq:denoteeta} and \eqref{eq:denoteH}, respectively. Assume that the parameters satisfy \eqref{eq:parameters} and the sufficient conditions in \eqref{eq:probabilistic-parameters} of  Prop.~\ref{prop:PDF}. For $\theta\in(\pi/2,\pi)$ and $\delta\ge1$, the following assertions hold:
\begin{description}
\item[$(i)$] The function $\eta$ is analytic in $\Sigma_\theta$ and satisfies $\eta(\Sigma_\theta)\subset\Sigma_\theta$. Moreover, there exist constants $c,C>0$ such that
\begin{equation}\label{eq:esteta}
  c|z|^{\alpha-\beta+\gamma}\le|\eta(z)|\le C|z|^{\alpha-\beta+\gamma},\qquad \forall~z\in\Sigma_{\theta,\delta}.
\end{equation}
\item[$(ii)$] Let $g$ be defined by \eqref{eq:dg}. Then the operator-valued function $\mathcal H$ is analytic in $\Sigma_\theta$.  For $z\in\Sigma_{\theta,\delta}$, 
\begin{equation}\label{eqs:oppp}
\begin{aligned}
&\left\|\mathcal H(z)\right\|_\Psi \le C|z|^{-1},
&&\left\|g(z)\mathcal H(z)\right\|_\Psi \le C|z|^{-\alpha-\gamma},\\
&\left\|\mathcal A\mathcal H(z)\right\|_\Psi \le C|z|^{\alpha-\beta+\gamma-1},
&&\left\|g(z)\mathcal A\mathcal H(z)\right\|_\Psi \le C|z|^{-\beta},
\end{aligned}
\end{equation}
\end{description}
The constants $c$ and $C$ may depend on $\theta$ and the model parameters, but are independent of $z$ and $\delta$.
\end{lemma}

\begin{proof}
Recall that $\mathcal{N}(z):=1+az^\alpha+\sum_{k=1}^{K}a_kz^{\alpha_k}$ and $\mathcal{D}(z):=1+bz^\beta+\sum_{j=1}^{J}b_jz^{\beta_j}$, so that $\eta(z)=z^\gamma\mathcal{N}(z)/\mathcal{D}(z)$.

For part $(i)$, we first establish the analyticity and sectorial property of $\eta$. As shown in the proof of Prop.~\ref{prop:PDF}, for $z=re^{i\varphi}$ with $r>0$ and $0<\varphi<\pi$, we have $0\le\arg\mathcal N(z)\le\alpha\varphi$ and $0\le\arg\mathcal D(z)\le\beta\varphi$. In particular, $\mathcal{N}(z)$ and $\mathcal{D}(z)$ are nonzero in $\Sigma_\theta$.  Using the conditions $\beta\le\gamma$ and $\alpha+\gamma\le1$ from \eqref{eq:probabilistic-parameters}, we obtain
\begin{displaymath}
0\le(\gamma-\beta)\varphi
\le\arg\eta(z)
=\gamma\varphi+\arg\big(\mathcal{N}(z)\big)-\arg\big(\mathcal{D}(z)\big)
\le(\gamma+\alpha)\varphi\le\varphi.
\end{displaymath}
The corresponding bounds for $\varphi<0$ follow by complex conjugation.  Hence $\eta$ is analytic in $\Sigma_\theta$ and satisfies $\eta(\Sigma_\theta)\subset\Sigma_\theta$. In particular, no additional condition involving $\sum_{j=1}^{J}\beta_j$ and $\sum_{k=1}^{K}\alpha_k$ is required.

We next establish the two-sided estimate in \eqref{eq:esteta}. Since $0<\alpha_k<\alpha$ and $0<\beta_j<\beta$, there exists a sufficiently large $R\ge1$ such that, for $|z|\ge R$, $1+\sum_{k=1}^{K}a_k|z|^{\alpha_k}
\le\frac{a}{2}|z|^\alpha$ and $1+\sum_{j=1}^{J}b_j|z|^{\beta_j}\le\frac b2|z|^\beta$. Therefore, by the triangle and reverse triangle inequalities, $\frac{a}{2}|z|^\alpha\le|\mathcal{N}(z)|\le\frac{3a}{2}|z|^\alpha$ and $\frac{b}{2}|z|^\beta\le|\mathcal{D}(z)|\le\frac{3b}{2}|z|^\beta$, as $|z|\ge R$. In particular, the coefficient in the upper bound for $|\mathcal{D}(z)|$ is $b$, and the corresponding power is $\beta$, rather than $a$ and $\alpha$. Since $\eta(z)=z^\gamma\mathcal{N}(z)/\mathcal{D}(z)$, the preceding inequalities yield
\begin{equation}
\frac{a}{3b}|z|^{\alpha-\beta+\gamma}
\le|\eta(z)|
\le\frac{3a}{b}|z|^{\alpha-\beta+\gamma},
\quad |z|\ge R,\quad |\arg z|\le\theta.
\end{equation}
For $1\le|z|\le R$ and $|\arg z|\le\theta$, the function $\frac{|\eta(z)|}{|z|^{\alpha-\beta+\gamma}}$ is continuous and strictly positive on this compact set because $\mathcal{N}(z)$ and $\mathcal{D}(z)$ have no zeros there. It therefore has a positive minimum and a finite maximum. Combining these estimates proves \eqref{eq:esteta} for $z\in\Sigma_{\theta,\delta}$, with constants independent of $\delta\ge1$.

For $(ii)$, the sector inclusion established in $(i)$ and Lemma~\ref{Lem:Resolvent} imply
\begin{equation}\label{eq:resolvent2}
\big\|(\eta(z)I+\mathcal{A})^{-1}\big\|_\Psi
\le C(1+|\eta(z)|)^{-1},\qquad z\in\Sigma_{\theta,\delta}.
\end{equation}
Since the resolvent is analytic on $\Sigma_\theta$, it follows from \eqref{eq:denoteH} that $\mathcal{H}(z)$ is analytic in the same sector. Furthermore,
\begin{equation}\label{est:Hz}
\left\|\mathcal H(z)\right\|_\Psi
\le|z|^{-1}|\eta(z)|\left\|(\eta(z)I+\mathcal A)^{-1}\right\|_\Psi
\le C|z|^{-1}.
\end{equation}
Using the resolvent identity $\mathcal A(\eta(z)I+\mathcal A)^{-1}=I-\eta(z)(\eta(z)I+\mathcal A)^{-1}$, we obtain 
\begin{equation}\label{est:AHz}
\begin{aligned}
\left\|\mathcal{A}\mathcal{H}(z)\right\|_\Psi
&\le|z|^{-1}|\eta(z)|
\left(1+|\eta(z)|\left\|(\eta(z)I+\mathcal{A})^{-1}\right\|_\Psi\right)\\
&\le C|z|^{-1}|\eta(z)|
\le C|z|^{\alpha-\beta+\gamma-1}.
\end{aligned}
\end{equation}
Finally, since $g(z)=z^{1-\gamma}/\mathcal{N}(z)$ and $|\mathcal N(z)|\ge C|z|^\alpha$ for $z\in\Sigma_{\theta,\delta}$, we have $|g(z)|\le C|z|^{1-\alpha-\gamma}$. Combining this estimate with the preceding bounds obtained in \eqref{est:Hz} and \eqref{est:AHz} yields
\begin{displaymath}
\left\|g(z)\mathcal{H}(z)\right\|_\Psi\le C|z|^{-\alpha-\gamma},\qquad
\left\|g(z)\mathcal{A}\mathcal H(z)\right\|_\Psi\le C|z|^{-\beta}.
\end{displaymath}
Thus, all the estimates in \eqref{eqs:oppp} follow, completing the proof. \qed 
\end{proof}

\begin{remark}\label{thm:remark1}
By Lemma~\ref{Lem:function}, $\eta(z)\in\Sigma_\theta$ for $z\in\Sigma_\theta$, where $\theta\in(\pi/2,\pi)$. Since $\mathcal A$ is positive definite, $(\eta(z)I+\mathcal A)^{-1}$ and $\mathcal H(z)$ are analytic in $\Sigma_\theta$. Consequently, the Laplace-domain solution $\widehat p(z)$ in \eqref{eq:integrand} is analytic in this sector whenever $\widehat f(z)$ admits an analytic continuation there. Furthermore, the behavior at the origin requires particular attention. Since $\eta(z)\sim z^\gamma$ as $z\to0$ and $\lambda_1>0$, we have $\mathcal H(z)=z^{\gamma-1}\big(\mathcal A^{-1}+o(1)\big)$, as $z\to0$, in the operator norm. Under \eqref{eq:probabilistic-parameters}, $0<\gamma<1$; thus the origin is generally a branch point rather than a simple pole. The inverse Laplace transform is therefore taken on the principal branch of the fractional powers, with the branch cut along the negative real axis.
\end{remark}

Next, we employ the Laplace transform method to establish the Sobolev regularity of the solution to Problem \eqref{eq:problem}.

\begin{lemma}[A priori estimates]\label{Lem:Lemma1}
Let $T>0$ be fixed. Under the assumptions of Lemma~\ref{Lem:function}, the following estimates hold for the solution operators $\mathcal{E}(t)$ and $\mathcal{F}(t)$ for every $0<t\le T$ and $m\in\mathbb{N}^+$, where the constant $C_T>0$ may depend on $T$ and the model parameters but is independent of $t$. 
\begin{description}
  \item[$(i)$] The operator $\mathcal{E}(t)$ satisfies
     \begin{equation*}
      \left\|\mathcal{E}(t)\right\|_{\Psi}
      +t^{m}\left\|\partial_{t}^{m}\mathcal{E}(t)\right\|_{\Psi}
      +t^{m+\alpha-\beta+\gamma}\left\|\mathcal{A}\partial_{t}^{m}\mathcal{E}(t)\right\|_{\Psi}
      \leq C_T,\quad 0<t\le T.
     \end{equation*}
  \item[$(ii)$] The operator $\mathcal{F}(t)$ satisfies
      \begin{equation*}\begin{aligned}
      t^{1-\alpha-\gamma}\left\|\mathcal{F}(t)\right\|_{\Psi}
        +t^{m+1-\alpha-\gamma}\left\|\partial_{t}^{m}\mathcal{F}(t)\right\|_{\Psi}
      +t^{m+1-\beta}\left\|\mathcal{A}\partial_{t}^{m}\mathcal{F}(t)\right\|_{\Psi}
      \leq C.
      \end{aligned}\end{equation*}
\end{description}
\end{lemma}
\begin{proof}
The proof of Lemma~\ref{Lem:Lemma1} is given in Appendix~\ref{proof:prior}.
\end{proof}

Based on the above lemmas, we establish well-posedness, sharp stability, and Sobolev regularity estimates.

\begin{theorem}[Well-posedness and stability]
\label{thm:mild-wellposed}
Assume that the parameters satisfy \eqref{eq:parameters} and \eqref{eq:probabilistic-parameters}. Let $T>0$, $p_0\in L^2(\Omega)$ and $f\in L^\infty(0,T;L^2(\Omega))$. Then \eqref{eq:integral} defines a mild solution $p\in C([0,T];L^2(\Omega))$ with $p(0)=p_0$. This solution is unique among continuous $L^2(\Omega)$-valued functions satisfying the modal Volterra equations
\begin{equation}\label{eq:mild-volterra}
p_j(t)+\lambda_j(k_1*p_j)(t)=p_{0,j}+(k_2*f_j)(t),
\qquad j\ge1,
\end{equation}
where $p_j(t)=(p(t),\varphi_j)$, $p_{0,j}=(p_0,\varphi_j)$,  $f_j(t)=(f(t),\varphi_j)$, $k_1=\mathcal{L}^{-1}\{1/\eta(z)\}$, and $k_2=\mathcal{L}^{-1}\left\{g(z)/z\right\}$. Moreover,
\begin{equation}\label{eq:mild-stability}
\left\|p\right\|_{C([0,T];L^2(\Omega))}\le \left\|p_0\right\|_{L^2(\Omega)}+{\lambda_1}^{-1}\left\|f\right\|_{L^\infty(0,T;L^2(\Omega))},
\end{equation}
where $\lambda_1>0$ is the smallest eigenvalue of $\mathcal A$. For two sets of data $(p_0^{(i)},f^{(i)})$, $i=1,2$, the corresponding solutions satisfy
\begin{equation}\label{eq:mild-cd}
\left\|p^{(1)}-p^{(2)}\right\|_{C([0,T];L^2(\Omega))}
\leq\left\|p_0^{(1)}-p_0^{(2)}\right\|_{L^2(\Omega)}
+\frac{1}{\lambda_1}\left\|f^{(1)}-f^{(2)}\right\|_{L^\infty(0,T;L^2(\Omega))}.
\end{equation}
The constants $1$ and $\lambda_1^{-1}$ are optimal uniformly in $T$.
\end{theorem}
\begin{proof}
A detailed proof is provided in Appendix~\ref{proof:well-posedness}.
\end{proof}

We now turn to analyze the Sobolev regularity of the solution to Problem \eqref{eq:problem}. The main results are stated in the following theorems.

\begin{theorem}[Sobolev regularity]\label{thm:theom1}
Let $m\in\mathbb N_+$, $p_0\in L^2(\Omega)$, and $f\in W^{m,1}(0,T;L^2(\Omega)).$ Assume the parameter conditions of Lemma~\ref{Lem:function}, and let $p$ be the mild solution of \eqref{eq:problem}. For $l=0,1$, set $\kappa_l:=(1-l)(\alpha+\gamma)+l\beta$. Then the following estimates hold for almost every $0<t\le T$:
\begin{description}
\item[$(i)$] There exists a constant $C>0$ such that
\end{description}
\begin{equation}\label{eq:SRSL}
\begin{aligned}
\left\|\mathcal A^l\partial_t^m p(t)\right\|_{L^2(\Omega)}
\le C\Big(&t^{-m-l(\alpha+\gamma-\beta)}\|p_0\|_{L^2(\Omega)}
+\sum_{j=0}^{m-1}t^{\kappa_l+j-m}\|f^{(j)}(0)\|_{L^2(\Omega)}\\
&+\int_0^t(t-s)^{\kappa_l-1}\|f^{(m)}(s)\|_{L^2(\Omega)}\,\mathrm{d}s\Big).
\end{aligned}
\end{equation}
\begin{description}
\item[$(ii)$] If, in addition, $p_0\in D(\mathcal A)=\dot H^2(\Omega)$, then
\end{description}
\begin{equation}\label{eq:SRSH}
\begin{aligned}
\left\|\mathcal A^l\partial_t^m p(t)\right\|_{L^2(\Omega)}
\le C\Big(&t^{(1-l)(\alpha+\gamma-\beta)-m}\|p_0\|_{\dot H^2(\Omega)}
+\sum_{j=0}^{m-1}t^{\kappa_l+j-m}\|f^{(j)}(0)\|_{L^2(\Omega)}\\
&+\int_0^t(t-s)^{\kappa_l-1}\|f^{(m)}(s)\|_{L^2(\Omega)}\,\mathrm{d}s\Big).
\end{aligned}
\end{equation}
The constant $C$ may depend on $T$ and the model parameters, but is independent of $t$, $p_0$, and $f$.
\end{theorem}

\begin{proof}
We refer to Appendix~\ref{proof:Refularity1} for the details.
\end{proof}

For $m=0$, we also obtain estimates under integrability assumptions on $f$.

\begin{theorem}[Sobolev regularity]\label{thm:ffffff}
Let $p_0\equiv0$ and let $p(t)$ be the solution to Problem \eqref{eq:problem}. Assume the parameters in \eqref{eq:parameters} satisfy the conditions stated in Lemma \ref{Lem:function}. Then the following estimates hold with constants $C > 0$ that may vary between inequalities:
\begin{description}
  \item[$(i)$] If $f\in L^{1}(0,T; L^{2}(\Omega))$, then for almost every $t\in(0,T]$ and $l=0,1$,
   \begin{equation}
     \left\|\mathcal{A}^{l}p(t)\right\|_{L^{2}(\Omega)}\leq C\int_{0}^{t}(t-s)^{\alpha+\gamma-1-l(\alpha+\gamma-\beta)}\left\|f(s)\right\|_{L^{2}(\Omega)}\,\mathrm{d}s.
   \end{equation}
  \item[$(ii)$] If $f\in L^{1}(0,T; \mathcal{D}(\mathcal{A}))$, then for almost every $t\in(0,T]$ and $l=0,1$,
     \begin{equation}
       \left\|\mathcal{A}^{l}p(t)\right\|_{L^{2}(\Omega)}\leq C\int_{0}^{t}(t-s)^{\alpha+\gamma-1}\left\|\mathcal{A}^{l}f(s)\right\|_{L^{2}(\Omega)}\,\mathrm{d}s.
     \end{equation}
\end{description}
\end{theorem}

\begin{proof}
For $p_0=0$, the mild representation reduces to $p=\mathcal{F}*f$. By Lemma \ref{Lem:function}, the bounds $\|\mathcal{A}^l\mathcal{F}(t)\|_\Psi\le Ct^{\kappa_l-1}$ with $\kappa_l=(1-l)(\alpha+\gamma)+l\beta>0$, yield
\begin{displaymath}
\left\|\mathcal{A}^l p(t)\right\|_{L^2(\Omega)}\le C\int_0^t(t-s)^{\kappa_l-1}\left\|f(s)\right\|_{L^2(\Omega)}\,\mathrm{d}s.
\end{displaymath}
This proves $(i)$ for almost every $t$. For $(ii)$, if $f\in L^1(0,T;D(\mathcal{A}))$, then commutation with the resolvent and closedness of $\mathcal{A}$ imply $\mathcal{A}^l p=\mathcal{F}*(\mathcal{A}^l f)$ for $l=0,1$. Using $\|\mathcal{F}(t)\|_\Psi\le Ct^{\alpha+\gamma-1}$ gives
\begin{displaymath}
\left\|\mathcal{A}^l p(t)\right\|_{L^2(\Omega)}
\le C\int_0^t(t-s)^{\alpha+\gamma-1}\left\|\mathcal A^l f(s)\right\|_{L^2(\Omega)}\,\mathrm{d}s.
\end{displaymath}
All convolutions are finite for almost every $t$ by the Bochner-valued Young inequality, since the corresponding kernels are locally integrable and the source belongs to $L^1$ in time. \qed
\end{proof}

According to Theorems \ref{thm:theom1} and \ref{thm:ffffff}, the solution $p(t,r)$ of Problem \eqref{eq:problem} may exhibit a temporal singularity near $t=0$. When $p_0=0$, enhancing the temporal regularity of the source term improves the solution's regularity. Moreover, since $\mathcal{A}$ is positive and self-adjoint, we have the norm equivalence $\|\mathcal{A}^{\nu}u\|_{L^{2}(\Omega)}=\|u\|_{\dot{H}^{2\nu}(\Omega)}$ for $\nu>0$. Using this equivalence and applying Sobolev interpolation to the estimates in Theorems \ref{thm:theom1} and \ref{thm:ffffff}, we obtain the following result.

\begin{corollary}\label{rmk:regularity}
Let $p(t)$ be the solution to Problem \eqref{eq:problem}. Assume the parameters in \eqref{eq:parameters} satisfy the conditions of Lemma \ref{Lem:function}. Then the following estimates hold:
\begin{description}
  \item[$(i)$]Let $f\equiv0$ and $p_0\in \dot{H}^{q}(\Omega)$. Then there exists a constant $C>0$ such that
     \begin{equation*}
       \left\|\partial_{t}^{m}p(t,r)\right\|_{\dot{H}^{p}(\Omega)}\leq Ct^{-m-(\alpha+\gamma-\beta)(p-q)/2}\left\|p_0\right\|_{\dot{H}^{q}(\Omega)},\quad 0<t\leq T,
     \end{equation*}
     where either $m=0$ and $0\leq q\leq p\leq2$, or $m>0$ and $0\leq q,~ p\leq2$.
  \item[$(ii)$] Let $p_0\equiv0$ with $0\leq q\leq p\leq2$ and $f\in L^{1}(0,T;\dot{H}^{q}(\Omega))$. Then there exists a constant $C>0$, such that for almost every $0<t\leq T$, 
      \begin{equation*}
      \left\|p(t,r)\right\|_{\dot{H}^{p}(\Omega)}
      \leq C\int_0^t (t-s)^{\alpha+\gamma-1-(\alpha+\gamma-\beta)(p-q)/2}
          \left\|f(s)\right\|_{\dot{H}^{q}(\Omega)}\mathrm{d}s.
      \end{equation*}
\end{description}
\end{corollary}

%%%%%%%%%%%%% Section 4 %%%%%%%%%%%%%%%%%%%%%%%%%%%%%%%%%%%%%%%%%%%%%%%
\section{Spatiotemporal Spectral Accuracy of the CIM-CLG Method}
\label{sec:CIMCLGs}

Having established the well-posedness and regularity results for the MT-TFJE \eqref{eq:problem}, we now turn to its numerical approximation and dynamical properties. We introduce a novel scheme achieving spatiotemporal spectral accuracy, the CIM-CLG algorithm, and present a detailed analysis of its convergence and stability. We begin with the temporal semidiscretization.

\subsection{Semidiscrete CIM}\label{sec:CIM}
We present the temporal semidiscrete approximation of the MT-TFJE via the contour integral method (CIM), accompanied by a comprehensive error study that encompasses the scheme's convergence and the selection of optimal contour parameters.

\subsubsection{Developing the semidiscrete CIM}
The contour integral method (CIM), originally introduced for the numerical approximation of the inverse Laplace transform (see \cite{Talbot1979}), has found extensive applications in various fields (see, e.g., \cite{Colbrook2022a,Colbrook2022b,Li2023,Fernandez2006,Ma2023b,McLean2010}), demonstrating its versatility and effectiveness. The method is based on deforming the Bromwich contour. While the standard contour is a vertical line that leads to challenges such as high-frequency oscillations in the integrand, deforming it into a contour within the left half-plane (as illustrated in Fig.~\ref{fig:FFF} (c) and supported by  \cite{Li2021,Fernandez2004,Ma2023a,Weideman2007}) results in exponential decay due to the factor $e^{zt}$, thereby enhancing the convergence rate of the numerical quadrature.

\begin{figure}[htbp]
  \centering
\begin{minipage}[c]{0.31\textwidth}
 \centering
 \centerline{\includegraphics[width=.88\textwidth]{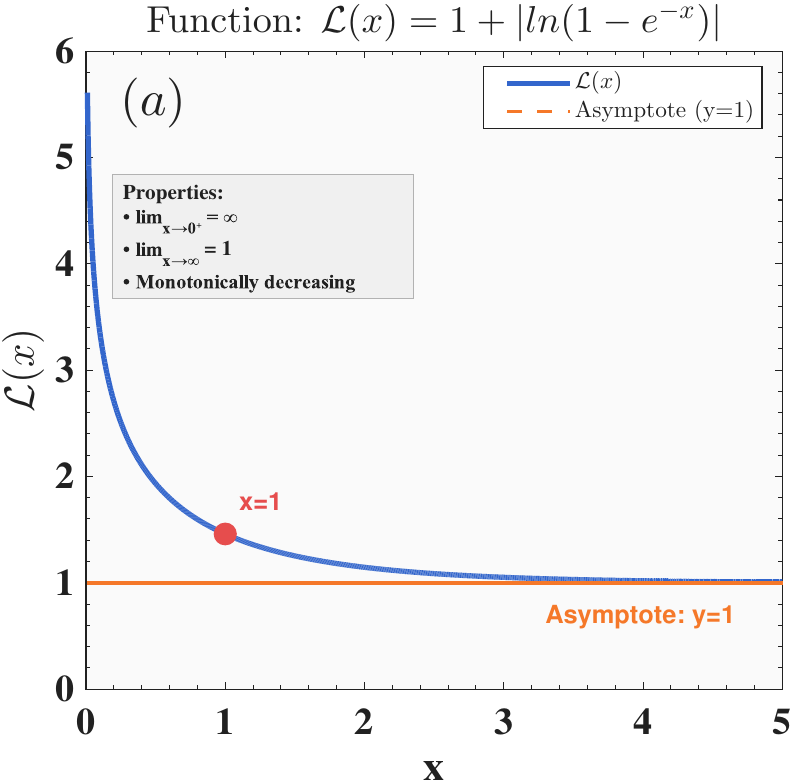}}
% \centerline{$(a)$ }
\end{minipage}
\begin{minipage}[c]{0.31\textwidth}
 \centering
 \centerline{\includegraphics[width=.88\textwidth]{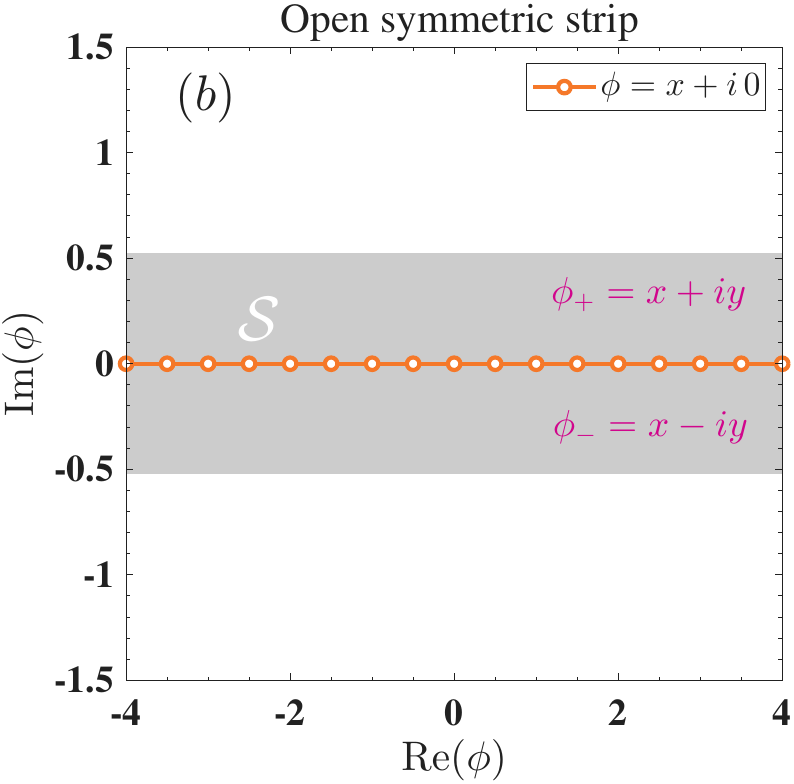}}
% \centerline{$(b)$ }
\end{minipage}
\begin{minipage}[c]{0.31\textwidth}
 \centering
 \centerline{\includegraphics[width=.88\textwidth]{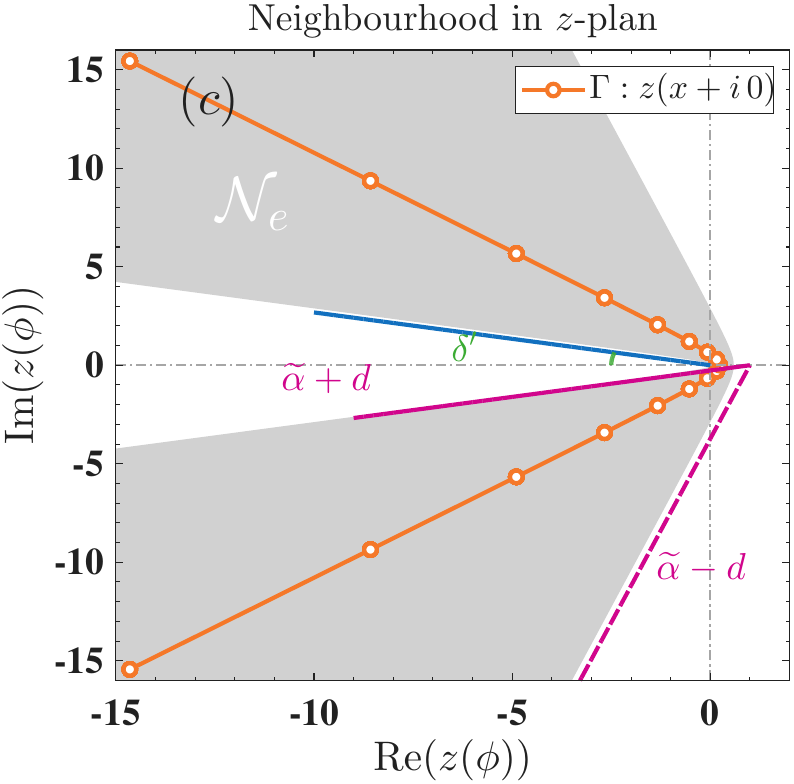}}
% \centerline{$(c)$ }
\end{minipage}
  \caption{Illustration of the auxiliary function $\mathcal{L}(x)$, the open symmetric strip $\mathcal{S}$, and its neighbourhood $\mathcal{N}_e$ in the $z$-plane. (a) Monotonic decrease of $\mathcal{L}(x)$. Panel (b) shows the strip $\mathcal{S}$, while panel (c) shows its image under the mapping $z(\phi)$ defined in \eqref{eq:hyperboliccontour}. The upper half of $\mathcal{S}$ (with $y>0$) maps to the shaded interior region enclosed by $z(x+i\,0)$ within $\mathcal{N}_e$; as $y$ increases, $z(\phi)$ approaches the negative real axis. Conversely, the lower half of $\mathcal{S}$ (with $y<0$) maps to the exterior shaded region outside $z(x+i\,0)$ in $\mathcal{N}_e$; as $y$ decreases, $z(\phi)$ converges toward a vertical line. Common parameters are set as $\widetilde{\alpha}=\pi/4$, $\mu=0.8$ and $\delta'=\pi/8$.}
  \label{fig:FFF}
\end{figure}

In line with the idea of CIM, we choose the left branch of a hyperbola as the integral contour (see, e.g., \cite{Fernandez2004,Fernandez2006,Ma2023a,Ma2023b,Weideman2007}). This contour can be parameterized as
\begin{equation}\label{eq:hyperboliccontour}
\Gamma:\ z(\phi) = \mu\big(1+\sin(i\phi-\widetilde{\alpha})\big),\quad \forall~ \phi\in \mathcal{S},
\end{equation}
where $\mu>0$ and $\widetilde{\alpha}>0$ are parameters to be determined, and $\mathcal{S}$ denotes an open symmetric strip (see Fig. \ref{fig:FFF} (b)), defined by $\mathcal{S}:=\{\phi=x+iy, |y|<d\}$, with $0<d<\min\{\widetilde{\alpha},\frac{\pi}{2}-\widetilde{\alpha}-\delta'\}$, satisfying  $0<\widetilde{\alpha}-d<\widetilde{\alpha}+d<\frac{\pi}{2}-\delta'$ (see Fig. \ref{fig:FFF} (c)). Here, $\delta'$ is a positive angular margin separating the contour neighbourhood from the negative real axis. For further discussion of this strip, we refer to our previous works \cite{Ma2023a,Ma2023b}. The conformal mapping $z(\phi)$ transforms the open strip $\mathcal{S}$ into a neighbourhood in the complex $z$-plane, denoted by $\mathcal{N}_{e}:=\{z(\phi)\in\mathbb{C}: \phi\in \mathcal{S}\}$ (see Fig. \ref{fig:FFF} (c)). According to Remark \ref{thm:remark1}, for a fixed $\delta'>0$, we have $\mathcal{N}_{e}\subseteq\Sigma_{\theta}$, and the integrand $\widehat{p}(z,r)$ is analytic for $z(\phi)\in \mathcal{N}_{e}$. Therefore, the integral \eqref{eq:integral} can be written as
\begin{equation}\label{eq:integralsolve}
p(t,r)=I:=\int_{-\infty}^{+\infty} u(t,r,\phi)\mathrm{d}\phi, \quad
u(t,r,\phi):=\frac{1}{2\pi i}e^{z(\phi)t}\widehat{p}(r,z(\phi))z'(\phi),
\end{equation}
with $\widehat{p}(r,z)$ being defined in \eqref{eq:integrand}.

Assume that $p_0$ and $f$ are real-valued. Since $f\in L^{\infty}(0,T;L^{2}(\Omega))$, the Schwarz reflection principle implies that  $\widehat{f}(\overline{z})=\overline{\widehat{f}(z)}$. Since the integral contour $\Gamma$ is symmetric with respect to the real axis and $\widehat{p}(r,\overline{z(\phi)})=\overline{\widehat{p}(r,z(\phi))}$,
applying the midpoint rule to approximate the integral \eqref{eq:integralsolve} with a uniform step-spacing $\tau>0$ in $\phi$ yields the following temporal semi-discrete scheme for Problem \eqref{eq:problem}, that is
\begin{equation}\label{eq:computeu}
 p^{N}(t,r)=I_{\tau;N}:=\tau\sum\limits_{k=-N}^{N-1}u(t,\phi_{k})=\frac{\tau}{\pi}\mathrm{Im}\Bigg\{\sum\limits_{k = 0}^{N-1}e^{z_{k}t}\widehat{p}(z_{k},r)z'_{k}\Bigg\},
\end{equation}
where $z_k=z(\phi_k)$,  $z'_k=z'(\phi_k)$ and $\phi_k=(k+1/2)\tau$. 

This scheme is straightforward to implement and allows the solution at each time point to be computed independently of previous time steps, facilitating efficient parallel computation.

\subsubsection{Error estimates under lower regularity}
For the well-posedness analysis, we assume $f\in L^\infty(0,T;L^2(\Omega))$,  which implies $f\in L^1(0,T;L^2(\Omega))$ and ensures the existence of its finite-interval Laplace transform. By Lemma~\ref{Lem:function}, the source transfer operator is analytic outside the negative real axis. Hence, the CIM applies provided that $\widehat f$ introduces no singularities obstructing the contour deformation and the required decay holds. For the subsequent estimates, we additionally impose the sufficient condition $\|\widehat{f}\|_{(\mathbb{Z};\mathbb{X})}=\sup_{z\in\mathbb{Z}}\|\widehat{f}(z)\|_{\mathbb{X}}<\infty$, where $\mathbb{Z}$ and $\mathbb{X}$ are specified in the corresponding estimates. The choice of $\mathbb{Z}$ will depend on the contour and the estimate under consideration. 

\begin{lemma}[Auxiliary function \cite{Fernandez2004}]\label{et:lemma}
For $x>0$, define $\mathcal{L}(x)=1+\big|\ln(1-e^{-x})\big|$. Then for any $s>0$ and $\sigma>0$, the following inequalities hold:
\begin{displaymath}
\int_{0}^{+\infty}e^{-s\cosh(x)}\mathrm{d}x\leq \mathcal{L}(s),\quad
\int_{\sigma}^{+\infty}e^{-s\cosh(x)}\mathrm{d}x\leq \big(1+\mathcal{L}(s)\big)e^{-s\cosh(\sigma)}.
\end{displaymath}
\end{lemma}

Note that $\mathcal{L}(x)$ is a decreasing function, with $\mathcal{L}(x)\to 1$ as $x\to \infty$, and $\mathcal{L}(x) \sim |\ln(x)|$ as $x \to 0^{+}$; see Fig. \ref{fig:FFF} ($a$) for an illustration. Based on Lemma \ref{et:lemma}, the integrand $u(t,r,\phi)$ defined in \eqref{eq:integralsolve} satisfies the following property.

\begin{lemma}[Super-exponential decay]\label{thm:prop2}
Let $u(t,r,\phi)$ be defined as in \eqref{eq:integralsolve}, with $z(\phi)\in \mathcal{N}_e\subseteq\Sigma_\theta$. Suppose $t_0\leq t\leq t_1$, where $t_0>0$ and $\Lambda:=t_1/t_0\geq1$. If $p_0\in L^{2}(\Omega)$ and $\|\widehat{f}(z)\|_{(\mathcal{N}_e;L^{2}(\Omega))}<\infty$, then there exists a constant $C>0$ such that for all $\phi=x\pm iy\in \mathcal{S}$,
\begin{equation}\label{ieq:setev}
\begin{aligned}
\|u(t,r,\phi)\|_{L^{2}(\Omega)}
&\leq CQe^{\mu t_0\Lambda}
\big(\|p_0\|_{L^{2}(\Omega)}+t_0^{\alpha+\gamma-1}\|\widehat{f}(z)\|_{(\mathcal{N}_e;L^{2}(\Omega))}\big)\\
&\qquad\times e^{-(1-\nu)\mu t_0\sin(\widetilde{\alpha}-d)\cosh x},
\end{aligned}
\end{equation}
where $Q:=\frac{1}{2\pi}\big(\frac{1-\alpha-\gamma}{\nu\sin(\widetilde{\alpha}-d)e}\big)^{1-\alpha-\gamma}
\sqrt{\frac{1+\sin(\widetilde{\alpha}+d)}{1-\sin(\widetilde{\alpha}+d)}}$ and $\nu\in(0,1)$. 
\end{lemma}

\begin{proof}
For a detailed proof, see Appendix \ref{proofa}.\qed
\end{proof}

Next, we introduce two auxiliary series to characterize errors arising in using the CIM \eqref{eq:computeu}:
\begin{itemize}[label=$\bullet$]
  \item $I_{\tau}:=\tau\sum_{k=-\infty}^{\infty}u(t,r,\phi_k)$,~ $I_{\tau;N;\varepsilon}:=\tau\sum_{k=-N}^{N-1}u(t,r,\phi_{k})(1+\varepsilon_k)\triangleq p_{\varepsilon}^{N}(t,r)$.
\end{itemize}
These allow us to decompose the error of the CIM scheme \eqref{eq:computeu} as follows:
\begin{equation}\label{eq:EN}
\begin{aligned}
E_{N}&=\|p(t,r)-p_{\varepsilon}^{N}(t,r)\|_{L^{2}(\Omega)}                \\
&\leq\|I-I_\tau\big\|_{L^{2}(\Omega)}+\|I_\tau-I_{\tau;N}\|_{L^{2}(\Omega)}
     +\|I_{\tau;N}-I_{\tau;N;\varepsilon}\|_{L^{2}(\Omega)}  \\
&=:\mathcal{D}_e+\mathcal{T}_e+\mathcal{R}_e,
\end{aligned}\end{equation}
where
$\mathcal{D}_e:=\|I-I_\tau\|_{L^{2}(\Omega)}$ denotes the discretization error, $\mathcal{T}_e:=\|I_\tau-I_{\tau;N}\|_{L^{2}(\Omega)}$ represents the truncation error, $\mathcal{R}_e:=\|I_{\tau;N}-I_{\tau;N;\varepsilon}\|_{L^{2}(\Omega)}$ is the round-off error
with $|\varepsilon_k|<\varepsilon$ being relative errors for all $k=-N,\ldots, N-1$. Here, $\varepsilon$ is a specified accuracy used in auxiliary computations, such as solving linear systems. For further details, we refer to \cite{Fernandez2006,Ma2023a}.

Based on the decay property established in Lemma \ref{thm:prop2} and the technical proofs in   \cite{Fernandez2004,Fernandez2006,Ma2023a}, it can be demonstrated that for $t_0\leq t\leq \Lambda t_0$ with $t_0>0$ and $\Lambda=t_1/t_0$, the errors $\mathcal{D}_e$ and $\mathcal{T}_e$ satisfy the following uniform estimate:
\begin{equation}\label{eq:DETE}
\begin{aligned}
\mathcal{D}_e+\mathcal{T}_e
&\leq
CQ\mathcal{L}\big((1-\nu)\mu
    t_0\sin(\widetilde{\alpha}-d)\big)\big(\|p_0\|_{L^{2}(\Omega)}
    +t_0^{\alpha+\gamma-1}\|\widehat{f}(z)\|_{(\mathcal{N}_e;L^{2}(\Omega))}\big) \\
&\qquad\times e^{\mu \Lambda t_0}\left(\frac{1}{e^{2\pi d/\tau}-1}
    +\frac{1}{e^{(1-\nu)\mu t_0\sin(\widetilde{\alpha}-d)\cosh(N\tau)}}\right), \quad \forall~\nu\in(0,1).
\end{aligned}\end{equation}
Additionally, the round-off error $\mathcal{R}_e$ is uniformly bounded by
\begin{equation}\label{eq:RE}
\begin{aligned}
\mathcal{R}_e
&\leq \tau\sum\nolimits_{k=-N}^{N-1}\|u(t,\phi_{k})\|_{L^{2}(\Omega)}|\varepsilon_k|
\leq 2\epsilon\tau\sum\nolimits_{k=0}^{N-1}\|u(t,\phi_{k})\|_{L^{2}(\Omega)} \\
&\leq CQ\mathcal{L}((1-\nu)\mu t_0\sin(\widetilde{\alpha}-d))\big(\|p_0\|_{L^{2}(\Omega)}
    +t_0^{\alpha+\gamma-1}\|\widehat{f}(z)\|_{(\mathcal{N}_e;L^{2}(\Omega))}\big)e^{\mu\Lambda t_0}\varepsilon.
\end{aligned}
\end{equation}

Let $\varrho\in(0,1)$ be a free parameter. Choosing
\begin{equation}\label{para:Fernandz}
\tau=\frac{a(\varrho)}{N},\qquad \mu=\frac{2\pi d N(1-\varrho)}{t_0\Lambda a(\varrho)},
\end{equation}
where $a(\varrho):=\operatorname{arcosh}\big(\frac{\Lambda}{(1-\nu)(1-\varrho)\sin(\widetilde{\alpha}-d)}\big)$ with $0<\nu,\varrho<1$. Denote $\varepsilon_N(\varrho):=\exp\big(-\frac{2\pi dN}{a(\varrho)}\big)$. With this, we can rewrite the estimate \eqref{eq:DETE} as
\begin{equation}\label{eq:DETEEN}
\begin{aligned}
\mathcal{D}_e+\mathcal{T}_e
&\leq CQ\mathcal{L}\big((1-\nu)\mu t_0\sin(\widetilde{\alpha}-d)\big)
  \big(\|p_0\|_{L^{2}(\Omega)}+t_0^{\alpha+\gamma-1}\|\widehat{f}(z)\|_{(\mathcal{N}_e;L^{2}(\Omega))}\big)\\
&\qquad \times e^{\mu\Lambda t_0}\frac{2\varepsilon_N(\varrho)}{1-\varepsilon_N(\varrho)},
\end{aligned}\end{equation}
where $\varepsilon_N(\varrho):=\exp\big(-\frac{2\pi dN}{a(\varrho)}\big)$. Combining \eqref{eq:RE} and \eqref{eq:DETEEN}, we obtain the overall error bound for $E_N$ in the CIM scheme (\ref{eq:computeu}), as stated in the following theorem.

\begin{theorem}[CIM Convergence]\label{thm:fernandz}
Let $p(t)$ be the solution to Problem \eqref{eq:problem}, and let $p^{N}(t)$ be its numerical approximation given by \eqref{eq:computeu} with the parameters satisfying the conditions stated in Lemma~\ref{Lem:function}. Assume $\Lambda>1$, $0<\widetilde{\alpha},\delta'<\pi/2$, and for any $t_0>0$ with $t_0\leq t\leq\Lambda t_0$, $0<\varrho<1$, if $p_0\in L^{2}(\Omega)$, $\|\widehat{f}(z)\|_{(\mathcal{N}_e;L^{2}(\Omega))}<\infty$, and parameters $\tau$, $\mu$ are chosen according to \eqref{para:Fernandz}, then the error $E_N=\|p(t)-p^{N}_{\varepsilon}(t)\|_{L^{2}(\Omega)}$ satisfies
\begin{equation*}
\begin{aligned}
E_N
&\leq CQ\mathcal{L}((1-\nu)\mu t_0\sin(\widetilde{\alpha}-d))\left(\left\|p_0\right\|_{L^{2}(\Omega)}
    +t_0^{\alpha+\gamma-1}\|\widehat{f}(z)\|_{(\mathcal{N}_e;L^2(\Omega))}\right)\\
&\qquad\times e^{\mu\Lambda t_0}\left(\varepsilon+\frac{\varepsilon_N(\varrho)}{1-\varepsilon_N(\varrho)}\right),
\end{aligned}\end{equation*}
where $0<d<\min\{\widetilde{\alpha},\pi/2-\widetilde{\alpha}-\delta'\}$, $Q$ is given in Lemma \ref{thm:prop2}.
\end{theorem}

Fix $\nu\in(0,1)$. By the definitions in \eqref{para:Fernandz}, we have $e^{\mu\Lambda t_0}=\big(\varepsilon_N(\varrho)\big)^{\varrho-1}$.
To minimize the leading exponential error terms in Theorem~\ref{thm:fernandz}, we select the optimal $\varrho^*$ as follows:
\begin{equation}\label{eq:optimal-rho}
\varrho^*\in\operatorname*{arg\,min}_{0<\varrho<1}
\left\{
\varepsilon\big(\varepsilon_N(\varrho)\big)^{\varrho-1}
+\frac{(\varepsilon_N(\varrho))^\varrho}
{1-\varepsilon_N(\varrho)}
\right\}.
\end{equation}
Substituting $\varrho^*$ into \eqref{para:Fernandz} determines $\tau^*$ and $\mu^*$.

\begin{remark}\label{rmk:stability}
For fixed $\Lambda$, $\widetilde{\alpha}$, and $\delta'$ satisfying $0<\widetilde{\alpha}, \delta'<\pi/2$ and $0<\widetilde{d}<\min\{\widetilde{\alpha},\pi/2-\widetilde{\alpha}-\delta'\}$, if the parameters in \eqref{para:Fernandz} are chosen optimally for the hyperbolic contour $\Gamma$, the error satisfies $E_N=\mathcal{O}(\varepsilon+e^{-CN})$ with $C=\mathcal{O}\big(\frac{1}{\ln(\Lambda)}\big)$. This implies that, in exact arithmetic, the CIM exhibits exponential convergence for fixed admissible parameters, provided the source-transform bounds remain uniform with respect to $N$. In finite precision, the contour parameters balance quadrature errors against possible round-off amplification.
\end{remark}

\subsection{Semi-discrete CLG}
In pursuit of spatio-temporal spectral accuracy, we adopt a spatial semi-discretization scheme that offers high precision while maintaining low computational cost. To this end, the Chebyshev-Legendre Galerkin (CLG) method is employed in this section for the spatial semi-discretization of both one- and two-dimensional formulations of Problem~\eqref{eq:problem}. The computational domain is set as $\Omega=(-1,1)^{n}$ for $n=1,2$. 

\subsubsection{The Chebyshev-Legendre Galerkin method}
Let $(u,v):=\int_{\Omega}uv\,\mathrm{d}r$ denote the inner product on $\Omega$, and $\|u\|_{L^{2}(\Omega)}:=(u,u)^{1/2}$ the corresponding norm. The variational formulation of Problem \eqref{eq:problem} is to find $p(t,r)\in H_{0}^{1}(\Omega)$ such that for all $v\in H_{0}^{1}(\Omega)$ and $t \in (0, T]$,
\begin{equation}\label{eq:variation}
\begin{aligned}
&_{0}^{RL}\mathfrak{I}_{t}^{1-\gamma}\Big(1+a~^{RL}_{0}\mathfrak{D}_{t}^{\alpha}
+\sum\nolimits_{k=1}^{K}a_k~_{0}^{RL}\mathfrak{D}_{t}^{\alpha_k}\Big)\big(\partial_tp(t,r),v(r)\big)\\
&=-\Big(1+b~_{0}^{RL}\mathfrak{D}_{t}^{\beta}+\sum\nolimits_{j=1}^{J}b_j~_{0}^{RL}\mathfrak{D}_{t}^{\beta_j}\Big)
\big(\nabla_rp(t,r),\nabla_rv(r)\big)+\bigl(f(t,r),v(r)\bigr).
\end{aligned}\end{equation}

Let $\mathbb{P}_{M}(\Omega)=\mathbb{P}_{M}(-1,1)$ denote the space of all algebraic polynomials of degree at most $M$ and define the approximation space as
\begin{gather}\label{Disspace}
\mathbb{V}^1_M:=\mathbb{P}_M(\Omega)\cap H_{0}^{1}(\Omega),\qquad
\mathbb{V}^2_M:=\mathbb{V}_M^1\otimes\mathbb{V}_M^1,
\end{gather}
for 1D and 2D problems, respectively. Hence, $\dim\mathbb{V}_M^n=(M-1)^n$, $n=1,2$. For convenience, we denote these as $\mathbb{V}^n_M$ with $n=1,2$. Then the semi-discrete scheme for Problem \eqref{eq:problem}, termed the Chebyshev-Legendre Galerkin (CLG) method, is formulated as follows: find $p_M\in \mathbb{V}^n_M$ such that for all $v_M\in \mathbb{V}^{n}_M$,
\begin{equation}\left\{
\begin{aligned}\label{eq:CLG}
&_{0}^{RL}\mathfrak{I}_{t}^{1-\gamma}\Big(1+a~^{RL}_{0}\mathfrak{D}_{t}^{\alpha}
+\sum\nolimits_{k=1}^{K}a_k~_{0}^{RL}\mathfrak{D}_{t}^{\alpha_k}\Big)\big(\partial_tp_M,v_M\big)\\
&{}=-\Big(1+b~_{0}^{RL}\mathfrak{D}_{t}^{\beta}+\sum\nolimits_{j=1}^{J}b_j~_{0}^{RL}\mathfrak{D}_{t}^{\beta_j}\Big)
\big(\nabla_rp_M,\nabla_rv_{M}\big)+\big(I^{c}_Mf,v_{M}\big),~ t>0,\\
%&(~_{0}^{C}D_{t}^{\alpha+\gamma}p_M,v_M)+(~_{0}^{C}D_{t}^{\gamma}p_M,v_M)
%+\big((1+\chi^{\beta}~_{0}^{RL}D_{t}^{\beta})p'_M,v'_M\big)=(I^c_Mf,v_M),~ t\in(0,T],\\
&\big(p^0_{M},v_M\big)=\big(I^{c}_Mp_0,v_M\big),\quad t=0.
\end{aligned}\right.
\end{equation}
Here, $I^c_M:$ $C(\bar{\Omega})\rightarrow\mathbb{P}_M(\Omega)$ denotes the Chebyshev interpolation operator. It is defined by $I^c_M v(\xi_j)=v(\xi_j)$ at the Chebyshev-Gauss-Lobatto (CGL) points $\xi_j:=\cos((\pi j)/M)$ for $0\leq j\leq M$. In two dimensions, $I_M^c$ denotes the tensor-product CGL interpolation operator $I^c_M=I_{M,x}^c\otimes I_{M,y}^c$, defined on the grid $(\xi_\jmath,\xi_\ell)$, $0\leq\jmath,\ell\leq M$. The approximate solution is expressed as $p_M=\sum_{k=0}^{M-2}p_{k}(t)\phi_{k}$ in 1D, and as $p_M=\sum_{\jmath,\ell=0}^{M-2}p_{\jmath\ell}(t)\phi_{\jmath}\phi_{\ell}$ in 2D, where $\phi_{k}$ and $\phi_{\jmath}\phi_{\ell}$ are the basis functions of the approximation space $\mathbb{V}^{n}_M$.

The primary distinction between the CLG method and the Legendre-Galerkin method lies in the interpolation operator employed in the terms $(I^c_Mf,v_M)$ and $(I^{c}_M p_0,v_M)$. The CLG method uses the Chebyshev interpolation operator $I^c_M$ rather than a Legendre interpolation operator. This choice combines two benefits. The Legendre-Galerkin part yields symmetric sparse matrices, while the Chebyshev interpolation allows the discrete Chebyshev transform to be computed by the fast Fourier transform (FFT) in $\mathcal{O}(M^{n}\log_{2}M)$ operations ($n=1,2$). As a result, the CLG method keeps the computational efficiency of Chebyshev interpolation and, at the same time, retains the numerical stability of Legendre discretizations (see, e.g., \cite{Ma98b,Ma98,shen2011,Zhao12}).

The efficiency of the CLG scheme also depends strongly on the choice of basis for $\mathbb{V}^{n}_M$, as this choice determines the structure of the resulting linear systems. We therefore choose the basis functions and the implementation separately for the one- and two-dimensional cases.

\subsubsection{Efficient 1D numerical implementation}
To facilitate the selection of basis functions adapted to the spatial dimension, we adopt the following convention: in one dimension, $x=r\in\Omega=(-1,1)$; in two dimensions, $(x,y)=r\in\Omega=(-1,1)^2$, and similarly for higher dimensions. Following \cite{Ma98,Shen94,shen2011}, we employ modal polynomial bases in the context of modal spectral Galerkin methods.

Accordingly, for the one-dimensional case, we define the approximation space $\mathbb{V}^{1}_M=\mathbb{P}_M(\Omega)\cap H_0^1(\Omega)$ and choose a basis $\{\phi_{0}(x), \phi_{1}(x),...,\phi_{M-2}(x)\}$, where for $\ell=0,1,2,...,M-2$, $\phi_\ell(x)=c_\ell(L_\ell(x)-L_{\ell+2}(x))$. Here, $L_\ell(x)$ denotes the Legendre polynomial of degree $\ell$, and the normalization constant is given by $c_\ell=1/\sqrt{4\ell+6}$. Since $L_\ell(\pm 1)=L_{\ell+2}(\pm 1)$, it holds that $\phi_\ell(\pm 1)=0$, which guarantees that $\phi_\ell \in H_0^1(\Omega)$. 

Define the stiffness and mass matrix entries as $a_{\jmath\ell}:=(\phi'_\jmath(x),\phi'_\ell(x))$,  $m_{\jmath\ell}:=(\phi_\jmath(x),\phi_\ell(x))$, where the prime denotes the first derivative with respect to $x$. Since $L_\ell(\pm1)=(\pm1)^{\ell}$, the functions $\{\phi_\ell(x)\}_{\ell=0}^{M-2}$ are linearly independent. Integration by parts yields
\begin{equation}
a_{\jmath\ell}:=\big(\phi'_\jmath(x),\phi'_\ell(x)\big)=-\big(\phi''_\jmath(x),\phi_\ell(x)\big)
=-\big(\phi_\jmath(x),\phi''_\ell(x)\big).
\end{equation}
Using the orthogonality of Legendre polynomials in $L^{2}(\Omega)$,
\begin{equation}
\big(L_\jmath(x),L_\ell(x)\big)=\int_{-1}^{1}L_\jmath(x)L_\ell(x)dx
=\frac{2}{2\jmath+1}\delta_{\jmath\ell},\quad \forall~ \jmath,~ \ell\geq0,
\end{equation}
and the recurrence relation $(2\ell+1)L_\ell(x)=L'_{\ell+1}(x)-L'_{\ell-1}(x)$, we derive the following explicit expressions
\begin{equation}
a_{\jmath\ell}=\left\{
\begin{aligned}
&1, ~ ~ \ell=\jmath,\\
&0, ~ ~ \ell\neq \jmath,
\end{aligned}\right. \quad
m_{\jmath\ell}=m_{\ell\jmath}=\left\{
\begin{aligned}
&c_\ell^2\Big(\frac{2}{2\ell+1}+\frac{2}{2\ell+5}\Big),&&\ell=\jmath,\\
&-\frac{2 c_\ell c_\jmath}{2\max\{\ell,\jmath\}+1}, &&|\ell-\jmath|=2,\\
&0,  &&{\rm otherwise}.
\end{aligned}\right.
\end{equation}
Now, taking the test functions $v_M=\phi_\ell(x)$ for $\ell=0,1,...,M-2$, the semi-discrete scheme \eqref{eq:CLG} in one dimension becomes
\begin{equation}\label{eq:CGL-basis}
\left\{\begin{aligned}
&{_{0}^{RL}\mathfrak{I}_{t}^{1-\gamma}}\Big(1+a~^{RL}_{0}\mathfrak{D}_{t}^{\alpha}
+\sum\nolimits_{k=1}^{K}a_k~_{0}^{RL}\mathfrak{D}_{t}^{\alpha_k}\Big)\big(\partial_tp_M(t),\phi_\ell\big)\\
&{}=-\Big(1+b{~_{0}^{RL}\mathfrak{D}_{t}^{\beta}}+\sum\nolimits_{j=1}^{J}b_j~_{0}^{RL}\mathfrak{D}_{t}^{\beta_j}\Big)
\big(p'_M(t),\phi'_\ell\big)+\big(I^{c}_Mf(t),\phi_\ell\big),~ t>0,\\
&\big(p^0_{M},\phi_\ell\big)=\big(I^{c}_Mp_0,\phi_\ell\big),\quad t=0.
\end{aligned}\right.
\end{equation}

For brevity, we recall that the numerator and denominator polynomials of the Laplace symbol $\eta(z)$ are $\mathcal{N}(z):=1+az^\alpha+\sum_{k=1}^{K}a_kz^{\alpha_k}$ and $\mathcal{D}(z):=1+bz^\beta+\sum_{j=1}^{J}b_jz^{\beta_j}$, respectively. Then, as discussed in Section \ref{eq:Lapace-problem} and in \eqref{eq:CGL-basis} for all $z\in\Sigma_{\theta,\delta}$, we obtain
\begin{equation*}%\label{eq:CGLL-basis}
z^{\gamma}\mathcal{N}(z)\big(\widehat{p}_{M}(z),\phi_\ell\big)+\mathcal{D}(z)\big(\partial_x\widehat{p}_{M}(z),\partial_x\phi_\ell\big)
=\frac{\mathcal{N}(z)}{z^{1-\gamma}}\big(I_{M}^{c}p_{0},\phi_\ell\big)
  +\big(I_{M}^{c}\widehat{f}(z),\phi_\ell\big),
\end{equation*}
for $\ell=0,1,...,M-2$. Define $p^{0}_{\ell}=(I^{c}_Mp_0,\phi_\ell)$, and let
$\mathbf{p_0}=(p^{0}_{0},p^{0}_{1},...,p^{0}_{M-2})^{T}$. Similarly, define
$\widehat{f}_\ell=(I^{c}_M\widehat{f}(z),\phi_\ell)$ and
$\mathbf{\widehat{f}}=(\widehat{f}_0,\widehat{f}_1,...,\widehat{f}_{M-2})^{T}$. Expanding the solution as
$\widehat{p}_M=\sum_{\ell=0}^{M-2}\widehat{p}_\ell(z)\phi_\ell$, we define the coefficient vector
$\mathbf{\widehat{p}}=(\widehat{p}_0(z),\widehat{p}_1(z),...,\widehat{p}_{M-2}(z))^{T}$. From the previous analysis, the stiffness matrix $\mathbf{A}=(a_{\jmath\ell})_{0 \leq \jmath,\ell\leq M-2}$ is identical to the identity matrix  $\textbf{I}$, while the mass matrix $\textbf{M}=(m_{\jmath\ell})_{0\leq \jmath,\ell\leq M-2}$ is pentadiagonal.
Thus, the scheme \eqref{eq:CGL-basis} is equivalent to the following matrix equation in the Laplace domain
\begin{equation}\label{eq:CGLLP-basis}
\begin{aligned}
\big[z^{\gamma}\mathcal{N}(z)\mathbf{M}+\mathcal{D}(z)\textbf{A}\big]\mathbf{\widehat{p}}
=z^{\gamma-1}\mathcal{N}(z)\mathbf{p_0}+\mathbf{\widehat{f}},
\quad\forall~ z\in\Sigma_{\theta,\delta}.
\end{aligned}
\end{equation}
The resulting system matrix is pentadiagonal, which facilitates the use of efficient algorithms for solving the corresponding linear equations. The banded structure of the one-dimensional matrix facilitates efficient direct solution, while invertibility follows from the discrete resolvent argument rather than diagonal dominance.

\subsubsection{Efficient 2D numerical implementation}
In two dimensions, we construct the approximation space as the tensor product of the one-dimensional bases. Since the one-dimensional bases satisfy $\phi_{\jmath}(\pm 1)=\phi_{\ell}(\pm 1)=0$, the product functions vanish on the boundary $\partial\Omega$. Hence, $\mathbb{V}^2_M \subset H_0^1(\Omega)$. The dimension of this space is $(M-1)^2$.

Choosing test functions of the form $v_M=\phi_{\jmath}(x)\phi_{\ell}(y)$ for $\jmath$, $\ell=0,1,...,M-2$, the semi-discrete scheme \eqref{eq:CLG} for the two-dimensional problem becomes
\begin{equation}\left\{
\begin{aligned}\label{eq:2DCGL-basis}
{_{0}^{RL}\mathfrak{I}_{t}^{1-\gamma}}\Big(1&+a~^{RL}_{0}\mathfrak{D}_{t}^{\alpha}
+\sum\nolimits_{k=1}^{K}a_k{}_{0}^{RL}\mathfrak{D}_{t}^{\alpha_k}\Big)\big(\partial_tp_M(t),\phi_\jmath\phi_\ell\big)\\
+\Big(1+b{}_{0}^{RL}&\mathfrak{D}_{t}^{\beta}+\sum\nolimits_{j=1}^{J}b_j{}_{0}^{RL}\mathfrak{D}_{t}^{\beta_j}\Big)
\Big[\big(\partial_xp_M(t),\phi'_\jmath\phi_l\big)+\big(\partial_yp_M,\phi_\jmath\phi'_\ell\big)\Big]\\
&=\big(I^c_Mf(t),\phi_\jmath\phi_\ell\big),\\
\big(P^0_M,\phi_\jmath\phi_\ell\big)&=\big(I^c_Mp_0,\phi_\jmath\phi_\ell\big).
\end{aligned}\right.
\end{equation}
Applying the Laplace transform, for all $z\in\Sigma_{\theta,\delta}$, we obtain
\begin{equation}\label{eq:2DCGLL-basis}
\begin{aligned}
&z^{\gamma}\mathcal{N}(z)\big(\widehat{p}_{M}(z),\phi_\jmath\phi_\ell\big)
+\mathcal{D}(z)\big[
  \big(\partial_x\widehat p_M(z),
       \phi'_\jmath\phi_\ell\big)
 +\big(\partial_y\widehat p_M(z),
       \phi_\jmath\phi'_\ell\big)
\big]\\
&=z^{\gamma-1}\mathcal{N}(z)\big(I_{M}^{c}p_{0},\phi_\jmath\phi_\ell\big)
  +\big(I_{M}^{c}\widehat{f}(z),\phi_\jmath\phi_\ell\big).
\end{aligned}
\end{equation}
To distinguish the two-dimensional quantities from their one-dimensional counterparts, we use uppercase letters. Define $\mathbf{P_0}=(P^{0}_{\jmath,\ell})_{0\leq \jmath,\ell\leq M-2}$ with $P^{0}_{\jmath,\ell}=(I^{c}_Mp_0,\phi_\jmath\phi_\ell)$ and
$\mathbf{\widehat{F}}=(\widehat{f}_{\jmath,\ell})_{0\leq\jmath, \ell\leq M-2}$ with $\widehat{f}_{\jmath,\ell}=(I^{c}_M\widehat{f},\phi_\jmath\phi_\ell)$. Expanding the solution as
$\widehat{P}_M=\sum_{\jmath,\ell=0}^{M-2}\widehat{P}_{\jmath,\ell}(z)\phi_\jmath\phi_\ell$ and 
$\mathbf{\widehat{P}}=(\widehat{P}_{\jmath,\ell}(z))_{0\leq \jmath,\ell\leq M-2}$ 
the matrix form of the transformed equation \eqref{eq:2DCGLL-basis} is given by
\begin{equation}\label{eq:2DCGLLP-basis}
\begin{aligned}
z^{\gamma}\mathcal{N}(z)\mathbf{M\widehat{P}M}
+\mathcal{D}(z)\big(\mathbf{A}\mathbf{\widehat{P}M}+\mathbf{M\widehat{P}}\mathbf{A}^{T}\big)
=z^{\gamma-1}\mathcal{N}(z)\mathbf{P_0}+\mathbf{\widehat{F}},
\end{aligned}\end{equation}
for all $z\in\Sigma_{\theta,\delta}$. Denote $\mathbf{U}=\operatorname{vec}(\mathbf{\widehat{P}})$ and $\mathbf{b}=\operatorname{vec}(z^{\gamma-1}\mathcal{N}(z)\mathbf{P_0}+\mathbf{\widehat{F}})$. The equation \eqref{eq:2DCGLLP-basis} can be rewritten using Kronecker products as
\begin{equation}\label{eq:2DCGLLP-Tensor}
\begin{aligned}
\left[z^{\gamma}\mathcal{N}(z)(\mathbf{M\otimes M})+\mathcal{D}(z)\big(\mathbf{M\otimes\mathbf{A}}+\mathbf{\mathbf{A}\otimes M}\big)\right]\mathbf{U}=\mathbf{b}.
\end{aligned}\end{equation}

\subsection{Stability and convergence of the CLG method}
We now investigate the stability and convergence of the CLG scheme \eqref{eq:CLG}. While the stability analysis applies to both one- and two-dimensional problems, the convergence proof is presented in one spatial dimension. The extension to two dimensions requires appropriate tensor-product projection and interpolation estimates under corresponding spatial regularity assumptions.

In practical computations, the prescribed initial data and source term are represented and evaluated numerically, which may introduce perturbations due to interpolation and finite-precision arithmetic. We therefore examine the stability of the CLG scheme with respect to such perturbations. For the subsequent convergence analysis, we also introduce the relevant projection operators and recall their approximation properties.

We begin by defining the discrete elliptic operator and its associated resolvent. Recall that $\mathbb{V}^n_M$ (for $n=1,2$) is the approximation space defined in \eqref{Disspace}, and we define $\mathcal{A}_M:\mathbb{V}^n_M\to \mathbb{V}^n_M$ by
\begin{equation}\label{eq:discrete-A}
(\mathcal{A}_M u_M,v_M)=(\nabla u_M,\nabla v_M),\qquad u_M,v_M\in \mathbb{V}^n_M.
\end{equation}
By construction, $\mathcal{A}_M$ is self-adjoint and positive definite on $\mathbb{V}^n_M$ with eigenvalues $0<\lambda_{1,M}\le\lambda_{2,M}\le\cdots\le\lambda_{{\rm dim} \mathbb{V}^n_M,M}$ ($n=1,2$). Furthermore, the Poincare inequality gives $\lambda_{1,M}\geq\lambda_1>0$. Define the discrete resolvent
\begin{equation}\label{eq:discrete-H}
\mathcal{H}_M(z):=z^{-1}\eta(z)\,\big(\eta(z)I_M+\mathcal{A}_M\big)^{-1},\qquad z\in\Sigma_{\theta}.
\end{equation}
By the spectral theorem applied to $\mathcal{A}_M$ (see \cite[Thm.~B.13]{Arendt2011}), together with the sectorial property of $\eta(z)$ established in Lemma~\ref{Lem:function}, we obtain, for $z\in\Sigma_{\theta,\delta}$,
\begin{equation}\label{eq:discrete-resolvent-bound}
\left\|(\eta(z)I_M+\mathcal{A}_M)^{-1}\right\|_{\Psi}\le\frac{C}{|\eta(z)|+\lambda_{1,M}},\qquad
\left\|\mathcal{H}_M(z)\right\|_{\Psi}\le C|z|^{-1}.
\end{equation}
Since $\lambda_{1,M}\geq\lambda_1>0$, the constants are independent of $M$ and $z$.
The same bounds extend to the low-frequency region $0<|z|<\delta$ by the positivity of $\mathcal{A}_M$ and the sectorial behavior of $\eta(z)$. Consequently, the estimates hold uniformly for $z\in\Sigma_\theta$. Using the estimate $|g(z)|\le C|z|^{1-\alpha-\gamma}$ from Lemma~\ref{Lem:function}, we also have
\begin{equation}\label{eq:discrete-gH-bound}
\left\|g(z)\mathcal{H}_M(z)\right\|_{\Psi}\le C|z|^{-(\alpha+\gamma)}.
\end{equation}
For the contour-integral analysis, we use the hyperbolic contour $\Gamma=z(\mathbb R)$ introduced in Section~\ref{sec:CIM}, which satisfies $\Gamma\subset\mathcal{N}_e\subset\Sigma_\theta$. Therefore, the above resolvent estimates remain valid for all $z\in\mathcal{N}_e$.

We shall also use the $L^2$-orthogonal projection operator $P^L_M:L^2(\Omega)\to\mathbb{P}_M(\Omega)$, defined by $(P^L_Mu,v_h)=(u,v_h)$ for all $v_h\in\mathbb{P}_M$ and $u\in L^2(\Omega)$. This projection satisfies the following approximation estimate.
\begin{lemma}[see \cite{shen2011}]\label{lem:interpolation1}
If $u\in H^{\tilde{r}}(\Omega)$ with $\tilde{r}\geq0$, then there exists $C>0$ independent of $M$, such that
\begin{equation}
\left\|P^L_Mu-u\right\|_{L^{2}(\Omega)}
\leq CM^{\mred{-\tilde{r}}}\|u\|_{H^{\tilde{r}}(\Omega)}.
\end{equation}
\end{lemma}

We next introduce the Ritz projection $\Pi_M:H_0^1(\Omega)\to\mathbb{V}_M^n$, defined variationally by $(\nabla(\Pi_Mu-u),\nabla v_M)=0$ for $v_M\in\mathbb V_M^n$. For $n=1$, this projection admits the
explicit representation $\Pi_Mu(x)=\int_{-1}^xP_{M-1}^Lu'(s)\mathrm{d}s$. The corresponding approximation properties are stated below.
\begin{lemma}[see \cite{Wu03}]\label{lem:poluno}
If $u\in H^{\tilde{r}}(\Omega)\cap H^{1}_0(\Omega)$ with $\tilde{r}\geq1$, then there exists a positive constant $C$ such that
\begin{equation}
\left\|\Pi_{M}u-u\right\|_{H^{l}(\Omega)}
\leq CM^{l-\tilde{r}}\|u\|_{\dot{H}^{\tilde{r}}(\Omega)},\quad\text{\rm for}~ -1\leq l\leq1.
\end{equation}
\end{lemma}

Finally, we consider the Chebyshev interpolation operator $I^c_M(\Omega)$: $C(\bar{\Omega})\rightarrow \mathbb{P}_M$ (see \cite{shen2011}), defined by $I^c_Mu(\xi_j)=u(\xi_j)$, where $\xi_j=\cos(\frac{\pi j}{M})$, $0\leq j\leq M$, are the Chebyshev points. This operator exhibits the following approximation property.

\begin{lemma}[see \cite{Li2009,Ma98,Wu03,Zhao12}]\label{lem:interpolation}
Let $u\in H^{1}(\Omega)$. There exists a constant $C>0$ such that
\begin{equation}
M\|I^{c}_Mu-u\|_{L^{2}(\Omega)}+\|I^{c}_Mu\|_{H^{1}(\Omega)}\leq C\|u\|_{H^{1}(\Omega)}.
\end{equation}
Moreover, if $u\in H^{\tilde{r}}(\Omega)$ with $\tilde{r}\geq1$, then
\begin{equation}\label{eq:estima}
\|I^{c}_Mu-u\|_{H^{l}(\Omega)}\leq CM^{l-\tilde{r}}\|u\|_{H^{\tilde{r}}(\Omega)},\quad\text{\rm for }~ 0\leq l\leq1.
\end{equation}
\end{lemma}

In the numerical implementation of the CLG scheme \eqref{eq:CLG}, the prescribed initial data and source term are represented and evaluated in finite precision and may therefore be subject to numerical perturbations. Let $\widetilde{p}_0$ and $\widetilde{f}$ denote the corresponding perturbations in the discrete initial data and source term, respectively, and let $\widetilde{p}$ denote the resulting perturbation in the semi-discrete solution $p_M$. Based on the preceding analysis, for any test function $v_M\in\mathbb{V}_M^1=\mathbb{P}_M(\Omega)\cap H_{0}^{1}(\Omega)$, the following identity holds
\begin{equation}\begin{aligned}
&{_{0}^{RL}\mathfrak{I}_{t}^{1-\gamma}}\Big(1+a~{^{RL}_{0}\mathfrak{D}_{t}^{\alpha}}
+\sum\nolimits_{k=1}^{K}a_k{~_{0}^{RL}\mathfrak{D}_{t}^{\alpha_k}}\Big)\big(\partial_t(p_M(t)+\widetilde{p}),v_M\big)\\
&=-\Big(1+b{~_{0}^{RL}\mathfrak{D}_{t}^{\beta}}+\sum\nolimits_{j=1}^{J}b_j{~_{0}^{RL}\mathfrak{D}_{t}^{\beta_j}}\Big)
\big(p'_M(t)+\widetilde{p}\,',v_M'\big)+\big(I_M^cf(t)+\widetilde{f},v_M\big),
\end{aligned}\end{equation}
for $t>0$. From this, one may directly verify that the errors $\widetilde{p}$, $\widetilde{p}_0$, and $\widetilde{f}$ satisfy the following error equation
\begin{equation}\begin{aligned}
&{_{0}^{RL}\mathfrak{I}_{t}^{1-\gamma}}\Big(1+a~{^{RL}_{0}\mathfrak{D}_{t}^{\alpha}}
+\sum\nolimits_{k=1}^{K}a_k~{_{0}^{RL}\mathfrak{D}_{t}^{\alpha_k}}\Big)\big(\partial_t\widetilde{p},v_M\big)\\
&=-\Big(1+b~{_{0}^{RL}\mathfrak{D}_{t}^{\beta}}+\sum\nolimits_{j=1}^{J}b_j~{_{0}^{RL}\mathfrak{D}_{t}^{\beta_j}}\Big)
\big(\widetilde{p}\,',v_M'\big)+\big(\widetilde{f},v_M\big),\quad t>0.
\end{aligned}\end{equation}
Subtracting the unperturbed semi-discrete equation from the perturbed one, we obtain the variational equation for
$\widetilde p$. Taking its Laplace transform and using the Galerkin interpretation of $\mathcal H_M(z)$ introduced above yields
\begin{equation}\label{eq:Nintegrand}
\widehat{\widetilde{p}}(z)
=\mathcal{H}_{M}(z)\widetilde{p}_0+g(z)\mathcal{H}_{M}(z)\widehat{\widetilde{f}}(z),
\end{equation}
for all $z\in\Sigma_{\theta,\delta}$, where the operator $\mathcal{H}_M(z)$ is defined by \eqref{eq:discrete-H} and $g(z)$ given by \eqref{eq:dg}. This identity provides the basis for the subsequent stability analysis.

\begin{theorem}[$L^{2}$-stability]\label{stability}
Assume the parameter conditions of Lemma~\ref{Lem:function}. Suppose that the CLG approximation $p_M$ is computed from the prescribed initial data and source term via \eqref{eq:CLG}, with corresponding errors $\widetilde{p}$, $\widetilde{p}_0$, and $\widetilde{f}$, respectively. Then there exists a positive constant $C$, independent of $M$ and $t$, such that for all $t\in(0,T]$ the following estimate holds:
\begin{equation}
\|\widetilde{p}(t)\|_{L^{2}(\Omega)}
\leq C\left(\|\widetilde{p}_0\|_{L^{2}(\Omega)}+t^{\alpha+\gamma-1}
\big\|\widehat{\widetilde{f}}(z)\big\|_{(\mathcal{N}_e;L^{2}(\Omega))}\right).
\end{equation}
\end{theorem}

\begin{proof}
For every $z\in\Sigma_\theta$ and $t>0$, $|1+tz|\ge\cos(\theta/2)(1+t|z|)$, and hence
\begin{displaymath}
\Big|e^{tz}-\frac1{1+tz}\Big|\le C
\begin{cases}
t|z|,&t|z|\le1,\\
e^{t\operatorname{Re}z}+(t|z|)^{-1},&t|z|>1.
\end{cases}
\end{displaymath}
On the hyperbolic contour $\Gamma\subset\mathcal{N}_e$, we further have $\operatorname{Re}z
=\mu\cos^2\widetilde\alpha-\sin\widetilde\alpha|z|$ and $e^{t\operatorname{Re}z}\le Ce^{-ct|z|}$, where $C$ is uniform for $0<t\le T$. By the contour-deformation assumptions of Section~\ref{sec:CIM}, Laplace inversion is valid along $\Gamma$. Moreover, $\int_\Gamma\frac{\widehat{\widetilde{p}}(z)}{1+tz}\mathrm{d}z=0$. Indeed, the inverse Laplace transform of $\widehat{\widetilde{p}}(z)/(1+tz)$ is $\mathcal{L}^{-1}\left\{\frac{\widehat{\widetilde{p}}(z)}{1+tz}
\right\}(\xi)=\frac{1}{t}\int_0^\xi e^{-(\xi-\tau)/t}\widetilde{p}(\tau)\mathrm{d}\tau\to0$, as $\xi\to0^+$. Its contour integral is absolutely convergent, so the identity follows by taking the inverse-transform limit at $0^+$. Equation  \eqref{eq:Nintegrand} then gives
\begin{displaymath}
\widetilde p(t)=\frac1{2\pi i}\int_\Gamma
\Big(e^{tz}-\frac1{1+tz}\Big)
\Big[\mathcal{H}_M(z)\widetilde{p}_0+g(z)\mathcal{H}_M(z)\widehat{\widetilde{f}}(z)\Big]\mathrm{d}z.
\end{displaymath}
On the contour
$z(x)=\mu(1+\sin(ix-\widetilde\alpha))$, we have $|z|=\sqrt{\mu^2+\frac{(\operatorname{Im}z)^2}
{\cos^2\widetilde\alpha}}-\mu\sin\widetilde\alpha$, and consequently, $\frac{1-\sin\widetilde\alpha}{\sqrt2}
(\mu+|\operatorname{Im}z|)\le |z|\le\frac{\mu+|\operatorname{Im}z|}{\cos\widetilde\alpha}$, $|dz|\le\frac{|d(\operatorname{Im}z)|}{\cos\widetilde\alpha}$. Splitting the integral into $t|z|\le1$ and $t|z|>1$ yields
\begin{displaymath}
\int_\Gamma\Big|e^{tz}-\frac1{1+tz}\Big||z|^{-(\alpha+\gamma)}|dz|\le Ct^{\alpha+\gamma-1}, \qquad
0<\alpha+\gamma\le1.
\end{displaymath}
In particular,
\begin{displaymath}
\int_\Gamma\Big|e^{tz}-\frac1{1+tz}\Big||z|^{-1}|\mathrm{d}z|\le C.
\end{displaymath}
Using \eqref{eq:discrete-resolvent-bound} and \eqref{eq:discrete-gH-bound}, we finally obtain
\begin{displaymath}
\begin{aligned}
\|\widetilde{p}(t)\|_{L^2(\Omega)}
&\le C\|\widetilde{p}_0\|_{L^2(\Omega)}\int_\Gamma\Big|e^{tz}-\frac1{1+tz}\Big||z|^{-1}|\mathcal{d}z|\\
&\quad+C\|\widehat{\widetilde{f}}\|_{(\mathcal{N}_e;L^2(\Omega))}\int_\Gamma\Big|e^{tz}
-\frac{1}{1+tz}\Big||z|^{-(\alpha+\gamma)}|\mathrm{d}z|\\
&\le C\Big(\|\widetilde{p}_0\|_{L^2(\Omega)}+t^{\alpha+\gamma-1}\|\widehat{\widetilde{f}}\|_{(\mathcal{N}_e;L^2(\Omega))}\Big).
\end{aligned}
\end{displaymath}
This completes the proof. \qed
\end{proof}

Theorem~\ref{stability} establishes the unconditional $L^2$-stability of the CLG method. Let $p(t)$ denote the solution of Problem~\eqref{eq:problem} satisfying the variational formulation~\eqref{eq:variation}, and let $p_M(t)$ be its CLG approximation defined by \eqref{eq:CGL-basis}. We introduce the comparison function $p^*(t)=\Pi_M p(t)$ and decompose the error as $p(t)-p_M(t)=[p(t)-p^*(t)]+\omega(t)$, where $\omega(t)=p^*(t)-p_M(t)$. Based on this decomposition, we establish the following $L^2$-convergence estimates for the one-dimensional CLG scheme.

\begin{theorem}[$L^2$-convergence]\label{eq:homoconverg}
Assume the parameter conditions of Lemma~\ref{Lem:function}. Let $p(t)$ be the solution of Problem~\eqref{eq:problem} and let $p_M(t)$ be its one-dimensional CLG approximation defined by \eqref{eq:CGL-basis}. Then, for $1\le\tilde r\le2$ and $0<t\le T$, the following estimates hold with constants independent of $M$ and $t$:
\begin{description}
\item[$(i)$] If $f\equiv0$ and $p_0\in\dot H^{\tilde r}(\Omega)$, then
\begin{equation}
\|p(t)-p_M(t)\|_{L^2(\Omega)}
\le C_1M^{-\tilde r}\|p_0\|_{\dot H^{\tilde r}(\Omega)}.
\end{equation}
\item[$(ii)$] If $p_0\equiv0$, $\|\widehat f\|_{(\mathcal N_e;\dot H^{\tilde r}(\Omega))}<\infty$, and the contour-deformation conditions of Section~\ref{sec:CIM} hold for $\widehat f$, then
\begin{equation}
\|p(t)-p_M(t)\|_{L^2(\Omega)}
\le C_2M^{-\tilde r}t^{\alpha+\gamma-1}
\|\widehat f\|_{(\mathcal N_e;\dot H^{\tilde r}(\Omega))}.
\end{equation}
\end{description}
\end{theorem}

\begin{proof}
By Lemma~\ref{Lem:function}, the spectral theorem (see \cite{Arendt2011}), and \eqref{eq:discrete-resolvent-bound}-\eqref{eq:discrete-gH-bound}, we have
\begin{equation}\label{eq:global-resolvent-for-error}
\begin{aligned}
&\|\mathcal H(z)v\|_{\dot H^{\tilde r}(\Omega)}\le C|z|^{-1}\|v\|_{\dot H^{\tilde r}(\Omega)},
&&\hspace{-0.55em}\|g(z)\mathcal H(z)v\|_{\dot H^{\tilde r}(\Omega)}\le C|z|^{-(\alpha+\gamma)}\|v\|_{\dot H^{\tilde r}(\Omega)},\\
&\|\mathcal H_M(z)w\|_{L^2(\Omega)}\le C|z|^{-1}\|w\|_{L^2(\Omega)},
&&\hspace{-0.55em}\|g(z)\mathcal H_M(z)w\|_{L^2(\Omega)}\le C|z|^{-(\alpha+\gamma)}\|w\|_{L^2(\Omega)}.
\end{aligned}
\end{equation}
These estimates hold for $z\in\Sigma_\theta$, $v\in\dot H^{\tilde r}(\Omega)$, and $w\in L^2(\Omega)$. Here $\mathcal H_M(z)$ acts on $L^2$-data through its Galerkin variational definition:
\begin{displaymath}
\eta(z)(\mathcal H_M(z)w,v_M)+((\mathcal H_M(z)w)',v_M')
=z^{-1}\eta(z)(w,v_M),\quad v_M\in\mathbb V_M^1.
\end{displaymath}
For $|z|\le1$, the bounds follow from $\eta(z)\to0$, $\mathcal D(z)\to1$, and the uniform spectral gap $\lambda_{1,M}\ge\lambda_1>0$. Moreover, Lemmas~\ref{lem:poluno} and~\ref{lem:interpolation} yield
\begin{equation}\label{eq:projection-estimates-used}
\|(I-\Pi_M)v\|_{L^2(\Omega)}+\|(I-I_M^c)v\|_{L^2(\Omega)}
\le CM^{-\tilde r}\|v\|_{\dot H^{\tilde r}(\Omega)}.
\end{equation}
Set $\widehat\rho=(I-\Pi_M)\widehat{p}$ and $\widehat\omega=\Pi_M\widehat{p}-\widehat{p}_M$. Subtracting the transformed variational equations and using Ritz orthogonality gives
\begin{equation}\label{eq:error-eq}
\begin{aligned}
\eta(z)(\widehat\omega,v_M)+(\widehat\omega',v_M')
={}&-\eta(z)(\widehat\rho,v_M)
+z^{-1}\eta(z)((I-I_M^c)p_0,v_M)\\
&+\mathcal{D}(z)^{-1}((I-I_M^c)\widehat{f}(z),v_M),
\quad v_M\in\mathbb{V}_M^1.
\end{aligned}
\end{equation}
Consequently,
\begin{equation}\label{eq:combined-error-representation}
\widehat{p}-\widehat{p}_M
=(I-z\mathcal{H}_M(z))\widehat\rho+\mathcal{H}_M(z)(I-I_M^c)p_0
+g(z)\mathcal{H}_M(z)(I-I_M^c)\widehat{f}(z).
\end{equation}
Since $\widehat{p}=\mathcal{H}(z)p_0+g(z)\mathcal{H}(z)\widehat{f}(z)$, estimates \eqref{eq:global-resolvent-for-error}-\eqref{eq:projection-estimates-used} imply
\begin{displaymath}
\|\widehat{\rho}(z)\|_{L^2(\Omega)}\le CM^{-\tilde{r}}
\left(|z|^{-1}\|p_0\|_{\dot{H}^{\tilde{r}}(\Omega)}
+|z|^{-(\alpha+\gamma)}\|\widehat{f}(z)\|_{\dot{H}^{\tilde{r}}(\Omega)}\right).
\end{displaymath}
Using $\|I-z\mathcal{H}_M(z)\|_{L^2(\Omega)\to L^2(\Omega)}\le C$ in \eqref{eq:combined-error-representation}, we obtain
\begin{equation}\label{eq:space-error-transform}
\|\widehat{p}(z)-\widehat{p}_M(z)\|_{L^2(\Omega)}
\le CM^{-\tilde{r}(\Omega)}
\left(|z|^{-1}\|p_0\|_{\dot{H}^{\tilde{r}}(\Omega)}
+|z|^{-(\alpha+\gamma)}\|\widehat{f}(z)\|_{\dot{H}^{\tilde{r}}(\Omega)}\right).
\end{equation}

We now invert along the hyperbolic contour $\Gamma\subset\mathcal N_e$. Following the proof of Theorem~\ref{stability}, for each fixed $t>0$, $\mathcal{L}^{-1}\left\{\frac{\widehat{p}(z)-\widehat{p}_M(z)}{1+tz}\right\}(\xi)=\frac{1}{t}\int_0^\xi e^{-(\xi-\tau)/t}[p(\tau)-p_M(\tau)]\mathrm{d}\tau\to0$ as $\xi\to0^+$. By the admissible contour deformation and the absolute convergence following from \eqref{eq:space-error-transform},
$\int_\Gamma\frac{\widehat p(z)-\widehat p_M(z)}{1+tz}\mathrm{d}z=0$. Therefore,
\begin{displaymath}
p(t)-p_M(t)=\frac{1}{2\pi i}\int_\Gamma\Big(e^{zt}-\frac1{1+tz}\Big)
[\widehat{p}(z)-\widehat{p}_M(z)]\,\mathrm{d}z.
\end{displaymath}

For $z\in\Sigma_\theta$, $|1+tz|\ge\cos(\frac{\theta}{2})(1+t|z|)$. On $\Gamma$, the hyperbolic parametrization further gives $\operatorname{Re}z=\mu\cos^2\widetilde\alpha-\sin\widetilde\alpha|z|$, $\frac{1-\sin\widetilde\alpha}{\sqrt2}(\mu+|\operatorname{Im}z|)
\le |z|\le\frac{\mu+|\operatorname{Im}z|}{\cos\widetilde\alpha}$, and $|dz|\le\frac{|d(\operatorname{Im}z)|}{\cos\widetilde\alpha}$.
Hence, for $0<t\le T$,
\begin{displaymath}
\Big|e^{zt}-\frac1{1+tz}\Big|\le C
\begin{cases}
t|z|,&t|z|\le1,\\
e^{-ct|z|}+(t|z|)^{-1},&t|z|>1.
\end{cases}
\end{displaymath}
Splitting the contour accordingly, for $\nu=1$ or $\nu=\alpha+\gamma\in(0,1]$, we obtain
\begin{equation}\label{eq:inverse-power-estimate}
\begin{aligned}
\int_\Gamma\Big|e^{zt}-\frac1{1+tz}\Big||z|^{-\nu}|\mathrm{d}z|
&\le C\left(t\int_0^{C/t}r^{1-\nu}\mathrm{d}r
+\int_{c/t}^\infty\Big(e^{-ctr}+\frac{1}{tr}\Big)r^{-\nu}\mathrm{d}r\right)\\
&\le Ct^{\nu-1}.
\end{aligned}
\end{equation}
Combining \eqref{eq:space-error-transform} and \eqref{eq:inverse-power-estimate} yields
\begin{displaymath}
\|p(t)-p_M(t)\|_{L^2(\Omega)}
\le CM^{-\tilde r}\big(\|p_0\|_{\dot H^{\tilde r}(\Omega)}
+t^{\alpha+\gamma-1}\|\widehat f\|_{(\mathcal N_e;\dot H^{\tilde r}(\Omega))}\big).
\end{displaymath}
Taking $f\equiv0$ and $p_0\equiv0$, respectively, proves $(i)$ and $(ii)$, including $\alpha+\gamma=1$.
\qed
\end{proof}

The preceding stability and convergence analysis relies on the Laplace-Volterra formulation and resolvent estimates for the continuous and discrete operators. This framework accommodates the multi-term fractional memory structure without invoking fractional Grönwall inequalities or requiring high-order temporal regularity. It also provides a basis for analyzing the coupling between the CLG spatial discretization and the CIM temporal approximation.

The spatial error estimates establish algebraic convergence rates determined by the Sobolev regularity of the initial data and source term, while the numerical results further demonstrate spectral convergence for sufficiently smooth data.

\subsection{The CIM-CLG algorithm}
The fully discrete scheme is constructed by applying the CIM to the semidiscrete CLG scheme given in  \eqref{eq:CGLLP-basis} for the one-dimensional case and \eqref{eq:2DCGLLP-Tensor} for the two-dimensional case. Specifically, the numerical solution is computed via
\begin{equation}\label{eq:fullcomputeu}
 p^{N}_{M}(t,r)
 =\frac{\tau}{\pi}\mathrm{Im}\Bigg\{\sum\limits_{k=0}^{N-1}e^{z_{k}t}\mathbf{\widehat{p}}_M(z_{k},r)z'_{k}\Bigg\},
\end{equation}
where $\mathbf{\widehat{p}}_M(z_{k},r)=\sum_{\ell=0}^{M-2}\widehat{p}_{\ell}(z_k)\phi_{\ell}(r)$ denotes the approximate solution in the Laplace domain at the contour quadrature nodes $z_k=z((k+\frac{1}{2})\tau)$ and $k=0,1,2,...,N-1$, for the 1D case; the 2D case is treated analogously using the tensor-product basis. Each vector $\mathbf{\widehat{p}}$ is obtained by solving the corresponding linear system in the Laplace domain, namely \eqref{eq:CGLLP-basis} in 1D and \eqref{eq:2DCGLLP-Tensor} in 2D. Since $\mathcal{A}_M$ is self-adjoint and positive definite and $\eta(z)\in\Sigma_\theta$ for $z\in\Sigma_\theta$, the operator $\eta(z)I_M+\mathcal{A}_M$ is invertible. Consequently, the coefficient matrices in \eqref{eq:CGLLP-basis} and \eqref{eq:2DCGLLP-Tensor} are invertible at every contour quadrature node.

\begin{theorem}[Convergence]
Let $M\geq2$, $t_0>0$, and $\Lambda>1$. Assume the contour parameters satisfy the admissibility conditions of Section \ref{sec:CIM}, and that $\tau$ and $\mu$ are chosen according to \eqref{para:Fernandz}, along with the parameter conditions of Lemma~\ref{Lem:function}, with $\Omega=(-1,1)$. Let $p(t,r)$ be the solution to Problem \eqref{eq:problem}, and let $p_M^{N}(r,t)$ be its approximation defined as in \eqref{eq:fullcomputeu}. For $0<t_0\leq t\leq \Lambda t_0$, the following estimates hold:
\begin{description}
\item[$(i)$] Suppose $f\equiv0$. If $p_0\in\dot{H}^{\tilde{r}}(\Omega)$ with $1\leq\tilde{r}\leq2$, then there exists a constant $C>0$ such that
\begin{align*}
    \big\|p(t,r)-p^{N}_{M}(t)\big\|_{L^{2}(\Omega)}
    \leq C\left(\Xi+M^{-\tilde{r}}\right)\|p_0\|_{\dot{H}^{\tilde{r}}(\Omega)}.
\end{align*}
\item[$(ii)$] Suppose $p_0\equiv0$. If $\|\widehat{f}(z)\|_{(\mathcal{N}_e;\dot{H}^{\tilde{r}}(\Omega))}<\infty$ with $1\leq\tilde{r}\leq2$, then there exists a constant $C>0$ such that
\begin{align*}
\big\|p(t,r)-p^{N}_{M}(t)\big\|_{L^{2}(\Omega)}
\leq Ct_0^{\alpha+\gamma-1}\Big(\Xi+M^{-\tilde{r}}\Big)\|\widehat{f}\|_{(\mathcal{N}_e;\dot{H}^{\tilde{r}}(\Omega))}.
\end{align*}
\end{description}
Here, $\Xi:=QL((1-\nu)\mu t_0\sin(\widetilde{\alpha}-d))e^{\mu\Lambda t_0}\big(\varepsilon+\frac{\varepsilon_N(\varrho)}
{1-\varepsilon_N(\varrho)}\big)$, $Q$ is defined in Lemma~\ref{thm:prop2}, $\varrho$ and $\nu\in(0,1)$.
\end{theorem}

\begin{proof}
By the triangle inequality, $\|p(t)-p_M^N(t)\|_{L^2(\Omega)}\le\|p(t)-p_M(t)\|_{L^2(\Omega)}+\|p_M(t)-p_M^N(t)\|_{L^2(\Omega)}$. The first term is estimated by Theorem~\ref{eq:homoconverg}. For the second term, the discrete operator $\mathcal A_M$ is self-adjoint and positive definite, with $\lambda_{1,M}\ge\lambda_1>0$. Hence the discrete resolvent estimates \eqref{eq:discrete-resolvent-bound}-\eqref{eq:discrete-gH-bound} are uniform in $M$, and the CIM argument of Theorem~\ref{thm:fernandz} applies to the semidiscrete solution. Moreover, the stability of $P^L_M$ and the interpolation estimate in Lemma~8 give $\|P^L_MI_M^cv\|_{L^2(\Omega)} \le C\|v\|_{\dot H^{\widetilde r}(\Omega)}$ under the stated one-dimensional regularity assumptions. Combining the resulting temporal error estimate with the spatial estimate of Theorem~\ref{eq:homoconverg} proves the assertion.
\end{proof}

Algorithm \ref{Algo:CIM-FLG} summarizes the proposed fully discrete CIM-CLG method.

\begin{algorithm}
\caption{CIM-CLG algorithm}
\label{Algo:CIM-FLG}
\begin{algorithmic}[1]
\STATE \textbf{Input:} $a,b,a_k,b_j$,$\alpha,\alpha_k,\beta,\beta_j,\gamma,t_0,\Lambda,T=\Lambda t_0,t,K,J,\widetilde\alpha,\nu,d,\delta',\varepsilon,N,M,N_{\varrho}=201$.
\STATE \textbf{Output:} $\mathbf p_M^N(t)$.
\FOR{$j=1,\ldots,N_{\varrho}-1$}
  \STATE $\varrho_j\gets j/N_{\varrho}$ \hfill $\triangleright$ \eqref{eq:optimal-rho}
  \STATE $a(\varrho_j)\gets \cosh^{-1}\big(\Lambda/[(1-\nu)(1-\varrho_j)\sin(\widetilde\alpha-d)]\big)$ \hfill $\triangleright$ \eqref{para:Fernandz}
  \STATE $\varepsilon_N(\varrho_j)\gets \exp\big(-2\pi dN/a(\varrho_j)\big)$   \hfill $\triangleright$ \eqref{para:Fernandz}
\ENDFOR
\STATE $\varrho^*\gets \operatorname*{arg\,min}_{\varrho_j}\big\{\varepsilon\varepsilon_N(\varrho_j)^{\varrho_j-1}+\varepsilon_N(\varrho_j)^{\varrho_j} / (1-\varepsilon_N(\varrho_j))\big\}$ \hfill $\triangleright$ \eqref{eq:optimal-rho}
\STATE Compute $\tau^*,\mu^*$ from \eqref{para:Fernandz}. \hfill $\triangleright$ \eqref{para:Fernandz}
\STATE Initialize $\mathbf p_M^N\gets0$.
\FOR{$k=0,\ldots,N-1$}
  \STATE $z_k\gets \mu^*(1+\sin(i(k+1/2)\tau^*-\widetilde\alpha))$ \hfill $\triangleright$ \eqref{eq:hyperboliccontour}
  \STATE $z_k'\gets i\mu^*\cos(i(k+1/2)\tau^*-\widetilde\alpha)$ \hfill $\triangleright$ \eqref{eq:hyperboliccontour}
  \STATE Assemble $\mathbf p_0$ and $\widehat{\mathbf f}(z_k)$. \hfill $\triangleright$ \eqref{eq:CGLLP-basis}
  \STATE Solve the linear system for $\widehat{\mathbf p}(z_k)$. \hfill $\triangleright$ \eqref{eq:CGLLP-basis} (1D), \eqref{eq:2DCGLLP-Tensor} (2D)
  \STATE $\mathbf p_M^N\gets \mathbf p_M^N+(\tau^*/\pi)\mathrm{Im}\{e^{z_k t}\widehat{\mathbf p}(z_k)z_k'\}$. \hfill $\triangleright$ \eqref{eq:fullcomputeu}
\ENDFOR
 \STATE Reconstruct $p_M^N(t,\mathbf r)$ from the computed Galerkin coefficients using the 1D or tensor-product basis.
\RETURN $\mathbf p_M^N(t)$.
\end{algorithmic}
\end{algorithm}

%%%%%%%%%%%% Section 5 %%%%%%%%%%%%%%%%%%%%%%%%%%%%%%%%%%%%%%%%%%%%%%%%
\section{Numerical experiments}
\label{sec:Numerical}
This section presents several numerical experiments conducted to validate our theoretical findings and demonstrate the accuracy and computational efficiency of the CIM-CLG algorithm. \footnote{The code used in this study is publicly available at \url{https://github.com/Fugui-Ma/TFDEs-Jeffreys-type-equation-MT-TFJE}.}

To quantitatively evaluate the accuracy of the method, we introduce the following error measures. When the exact solution $p(t)$ is available, we define the temporal error $error_N$ as a function of the time discretization parameter $N$, and the spatial error $error_M$ as a function of the polynomial degree $M$, respectively,
\begin{displaymath}
\left\{\begin{aligned}
  &error_{N}(N):=\max\limits_{t_{0}\leq t\leq \Lambda t_{0}}\left\|p(t)-p_{M}^{N}(t)\right\|_{L^{2}(\Omega)},\\
  &error_{M}(M):=\left\|p(t)-p_{M}^{N}(t)\right\|_{L^{2}(\Omega)},
\end{aligned}\right.
\end{displaymath}
with fixed $M$ or $N$. The corresponding orders of convergence in space and time are defined as
\begin{displaymath}
\mathrm{Order_s}:=\ln[error_{M}(M_{j})/error_{M}(M_{j+1})]/\ln(M_{j+1}/M_{j}),
\end{displaymath}
and
\begin{displaymath}
\mathrm{Order_t}:= \ln[error_{N}(N_{j})/error_{N}(N_{j+1})]/\ln(N_{j+1}/N_{j}),
\end{displaymath}
respectively, where $\{M_j\}$ and $\{N_j\}$ are strictly increasing sequences of polynomial degrees and time discretization parameters, respectively. In cases where the exact solution is unknown, we employ the following error indicators
\begin{displaymath}
\left\{\begin{aligned}
&error_{N}(N)=\max\limits_{t_{0}\leq t\leq \Lambda t_{0}}\left\|p_{M}^{100}(t)-p_{M}^{N}(t)\right\|_{L^{2}(\Omega)},\\
&error_{M}(M)=\left\|p_{2M}^{N}(t)-p_{M}^{N}(t)\right\|_{L^{2}(\Omega)},
\end{aligned}\right.
\end{displaymath}
again with fixed $M$ or $N$. The numerical convergence order is then estimated as
\begin{displaymath}
\mathrm{Order}=\ln(error_{M/2}/error_{M})/\ln(2),
\end{displaymath}
so that the numerical convergence behavior follows $\mathcal{O}(M^{-{\rm Order}})$. 

In the following numerical experiments, the diffusion coefficient is normalized to $\mathcal{D}=1$, and the parameters satisfy  Prop.~\ref{prop:PDF}. Examples~\ref{examp:1} and~\ref{Ne:exapmel2} consider the case $a_k=b_j=0$ as benchmark problems. Specifically, Example~\ref{examp:1} examines the temporal convergence and stability of the CIM for a scalar problem, while Example~\ref{Ne:exapmel2} investigates the temporal and spatial accuracy of the coupled CIM-CLG scheme for a 1D problem with a manufactured exact solution. This simplification is adopted for clarity and is not a consequence of Prop.~\ref{prop:PDF}, Lemma~\ref{Lem:function}, or the MSD asymptotics in \eqref{asp:MSD}. To further validate the applicability of the proposed method to the general MT-TFJE, Example~\ref{Examp:2D} examines the 2D full multi-term model with $K=J=2$ and nonzero additional coefficients. The 2D experiment evaluates temporal and spatial convergence under various fractional-order configurations and illustrates the evolution of the numerical solutions. These results provide direct numerical evidence for the performance of the CIM-CLG method in solving the full multi-term PDE, complementing the multi-term MSD experiments presented in Section~\ref{sec:Modeling}.

In all numerical examples, the domain is taken as $\Omega=(-1,1)^n$ ($n=1, 2$).  The contour parameters are set to $\delta'=\pi/8$ and $\widetilde{\alpha}=\pi/4$.  The experiments are carried out in MATLAB R2025a on a PC with an Intel(R) Core(TM) i7-6700 CPU @3.40GHz and 16.00GB of RAM.

\subsection{Example 1. (A scalar problem)}\label{examp:1}
To illustrate the spectral accuracy and numerical stability of the CIM, we consider the following time-fractional differential equation
\begin{equation}\label{eq:scalarpro}
{~_{0}^{RL}\mathfrak{I}_{t}^{1-\gamma}}\big(1+a{~_{0}^{RL}\mathfrak{D}_{t}^{\alpha}}\big)\partial_tp(t)
+\big(1+b{~_{0}^{RL}\mathfrak{D}_{t}^{\beta}}\big)\lambda p(t)=f(t),
\end{equation}
with exact solution $p(t)=1+t^{4/5}$ (which has limited temporal regularity), where $\lambda\in \mathbb{R}_+$ is a specified positive constant. We set $p_0=1$ accordingly, and the source term $f$ is explicitly derived as
\begin{itemize}[label=$\bullet$]
  \item
  $
  f(t)=\lambda\Big(1+t^{4/5}+\frac{b t^{-\beta}}{\Gamma(1-\beta)}
    +\frac{b\Gamma(9/5)t^{4/5-\beta}}{\Gamma(9/5-\beta)}\Big)
    +\frac{4}{5}\Gamma\Big(\frac{4}{5}\Big)\Big(\frac{t^{4/5-\gamma}}{\Gamma(9/5-\gamma)}
    +\frac{at^{4/5-\alpha-\gamma}}{\Gamma(9/5-\alpha-\gamma)}\Big).
  $
\end{itemize}
We employ the CIM for numerical discretization and quantify the absolute error as a function of the discretization parameter $N$ via
\begin{equation}
  \operatorname{Error}(N)=\max\limits_{t_{0}\leq t\leq \Lambda t_{0}}\big|p(t)-p^{N}(t)\big|.
\end{equation}

Fig.~\ref{Fig:example1} (a) shows the convergence behavior of $Error(N)$ for different values of $\Lambda$, demonstrating the spectral accuracy of the method. Figures \ref{Fig:example1}(b) and (c) present a stability analysis under various fractional orders $\alpha$, $\beta$, and $\gamma$, as well as different ratios $b/a$, further confirming the robustness of the CIM.

\begin{figure}[htbp]
  \centering
\begin{minipage}[c]{0.28\textwidth}
 \centering
 \centerline{\includegraphics[width=\textwidth]{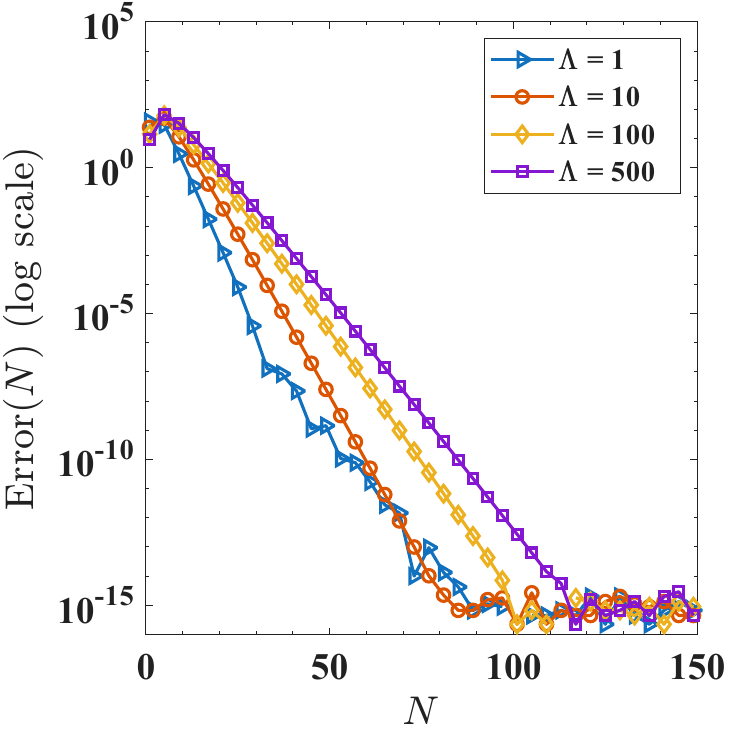}}
% \centerline{$(a)$ }
\end{minipage}
\begin{minipage}[c]{0.28\textwidth}
 \centering
 \centerline{\includegraphics[width=\textwidth]{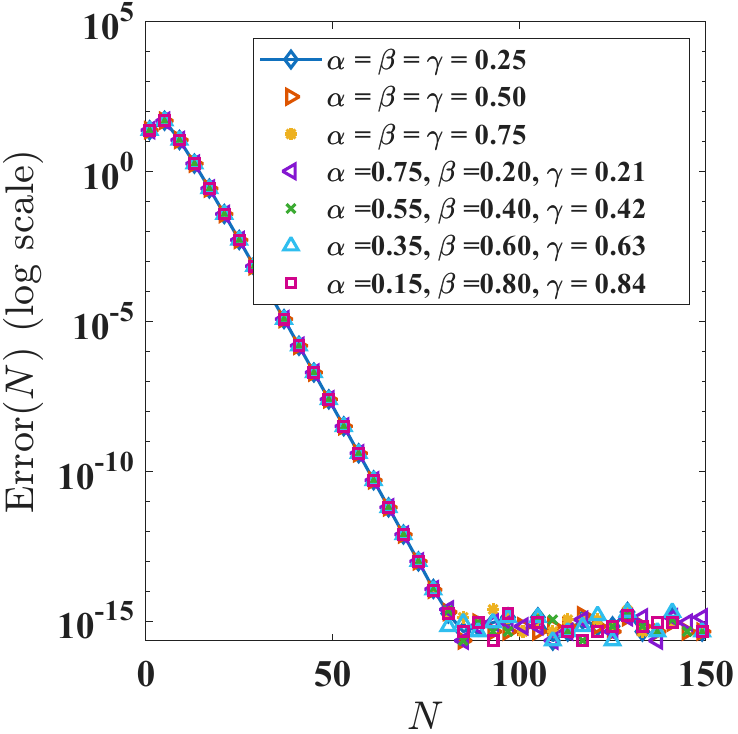}}
% \centerline{$(b)$ }
\end{minipage}
\begin{minipage}[c]{0.28\textwidth}
 \centering
 \centerline{\includegraphics[width=\textwidth]{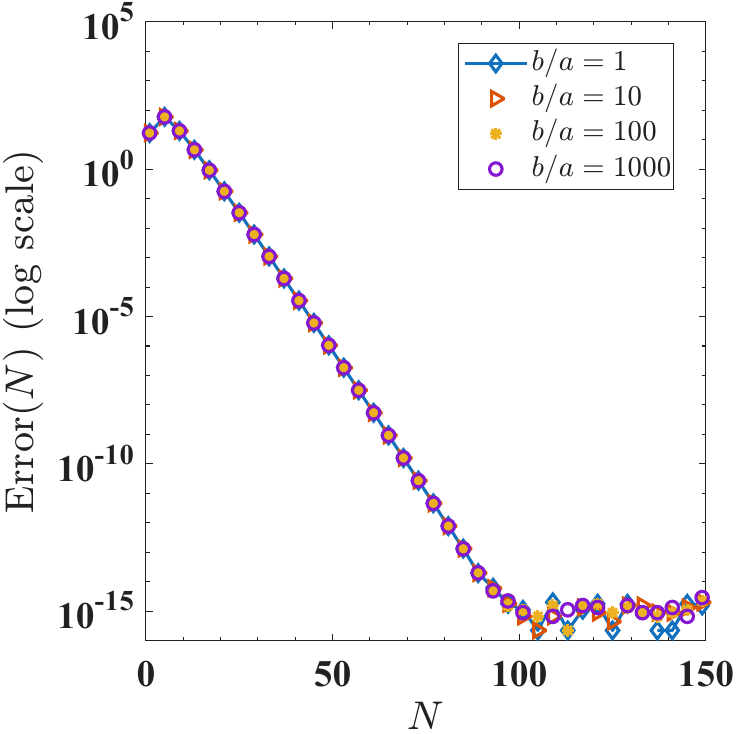}}
% \centerline{$(c)$ }
\end{minipage}
  \caption{Numerical performance of the CIM illustrating its spectral accuracy, stability, and robustness. (left) Error evolution with the number of quadrature points $N$ for various $\Lambda$, with $\alpha=0.50$, $\beta=0.35$, $\gamma=0.45$, $\lambda=1.50$, $a=1.00$, $b=100.00$, and $t=0.50$. (middle) Error under varying fractional orders, with $a=1.00$, $b=10.00$ and $\Lambda = 10.00$. (right) Error across different ratios of $b/a$, with $\alpha=0.50$, $\beta=0.35$, $\gamma=0.45$, $\lambda=1.50$, and $t=0.50$.}
  \label{Fig:example1}
\end{figure}

As demonstrated in Fig. \ref{Fig:example1}, the proposed approach preserves the spectral accuracy and stability of the CIM across a wide range of fractional orders. The results illustrate rapid convergence in the tested parameter regimes. The experiments outside the assumptions of Remark \ref{rmk:stability} provide additional numerical evidence, rather than a direct verification of the theorem. Specifically, the CIM achieves a convergence rate of $\mathcal{O}\left(\varepsilon +e^{-cN}\right)$, where $c=\mathcal{O}\left(1/\ln(\Lambda)\right)$.

\subsection{Example 2. (1-D problem with known source)}\label{Ne:exapmel2}
To examine the accuracy of the CIM-CLG method under limited temporal regularity, we consider Problem~(\ref{eq:problem}) with zero initial data and $a_k=b_j=0$. The exact solution is prescribed as $p(t,x)=t^{\kappa}(\sin(\pi x)+1.1)(1-x^2)$, where $\kappa>0$. The corresponding source term is 
\begin{itemize}[label=$\bullet$]
  \item $f(t,x)=u(x)\Big[\Big(\frac{\Gamma(\kappa+1)t^{\kappa-\gamma}}{\Gamma(\kappa+1-\gamma)}
           +a\frac{\Gamma(\kappa+1)t^{\kappa-\alpha-\gamma}}{\Gamma(\kappa+1-\alpha-\gamma)}\Big)\Big]
          -\bar{u}(x)\Big(t^{\kappa}+b\frac{\Gamma(\kappa+1)t^{\kappa-\beta}}{\Gamma(\kappa+1-\beta)}\Big),$
\end{itemize}
with $u(x)=(\sin(\pi x)+1.1)(1-x^2)$ and $\bar{u}(x)=-\sin(\pi x)[\pi^2(1-x^2)+2]-2.2-4\pi x\cos(\pi x)$. 

For $\kappa=1/4$ and the parameter sets considered here, the source term satisfies
$f\in L^1(0,T;L^2(\Omega))$ but $f\notin L^\infty(0,T;L^2(\Omega))$. Thus, this example tests the numerical performance for weakly singular source terms beyond the source assumptions of Theorem~\ref{thm:mild-wellposed}.

\begin{figure}[htbp]
  \centering
  % ================= 第一行：时间谱精度 (CIM) =================
  \begin{minipage}[c]{0.28\textwidth}
    \centering
    \includegraphics[width=\textwidth]{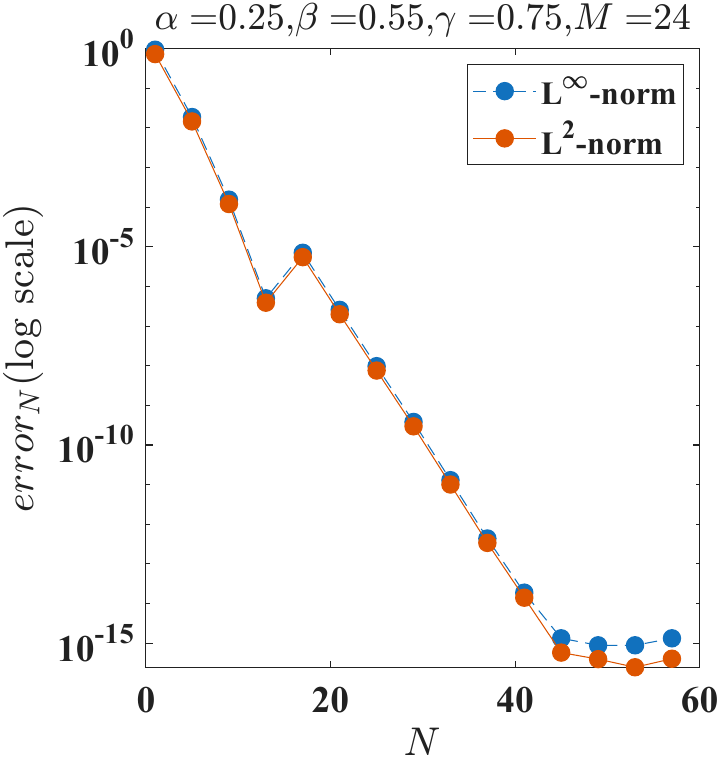}
  \end{minipage}
  \begin{minipage}[c]{0.28\textwidth}
    \centering
    \includegraphics[width=\textwidth]{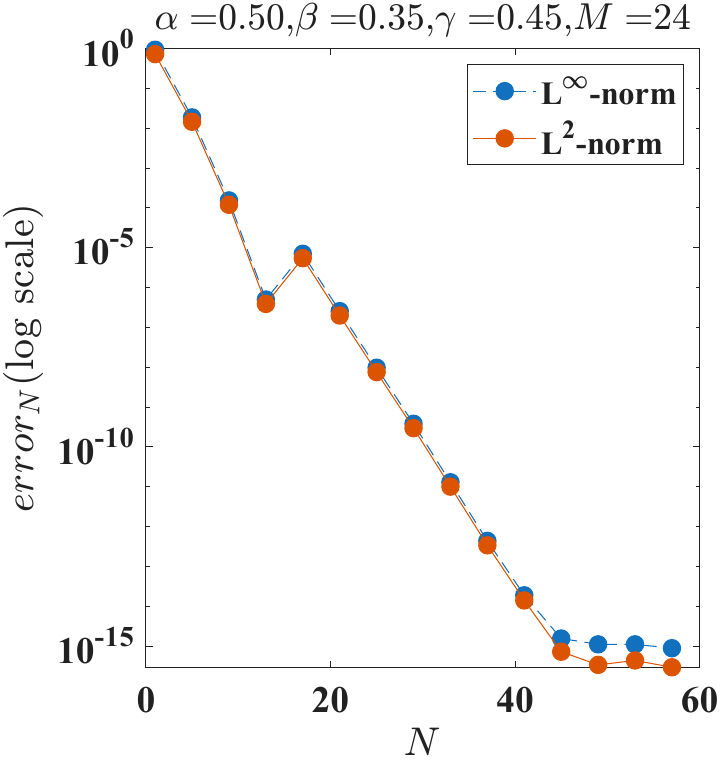}
  \end{minipage}
  \begin{minipage}[c]{0.28\textwidth}
    \centering
    \includegraphics[width=\textwidth]{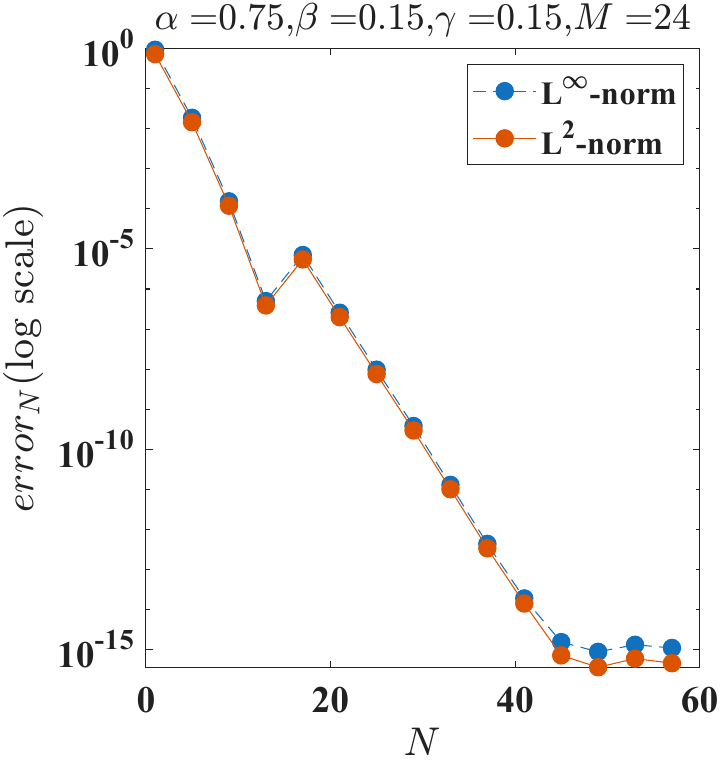}
  \end{minipage}
  \begin{minipage}[c]{0.28\textwidth}
    \centering
    \includegraphics[width=\textwidth]{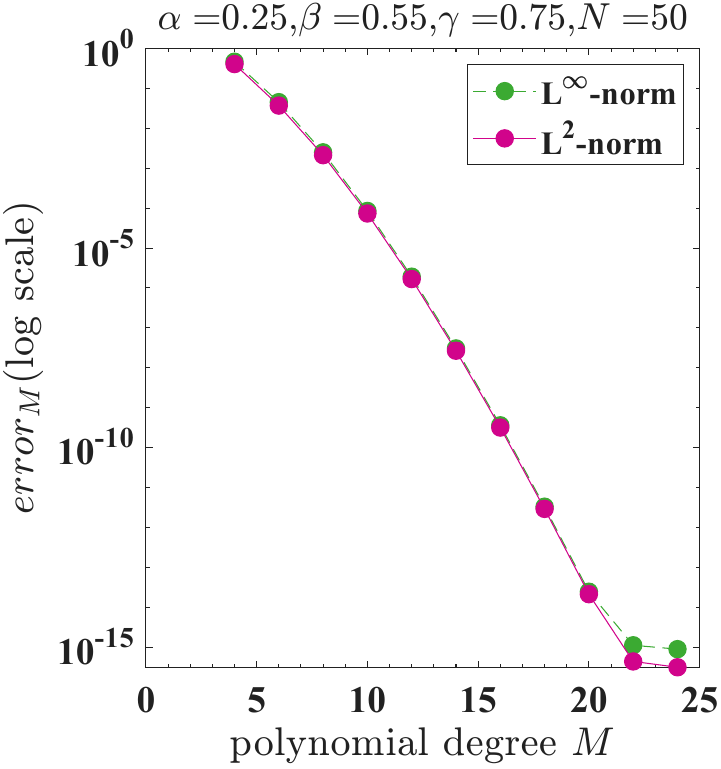}
  \end{minipage}
  \begin{minipage}[c]{0.28\textwidth}
    \centering
    \includegraphics[width=\textwidth]{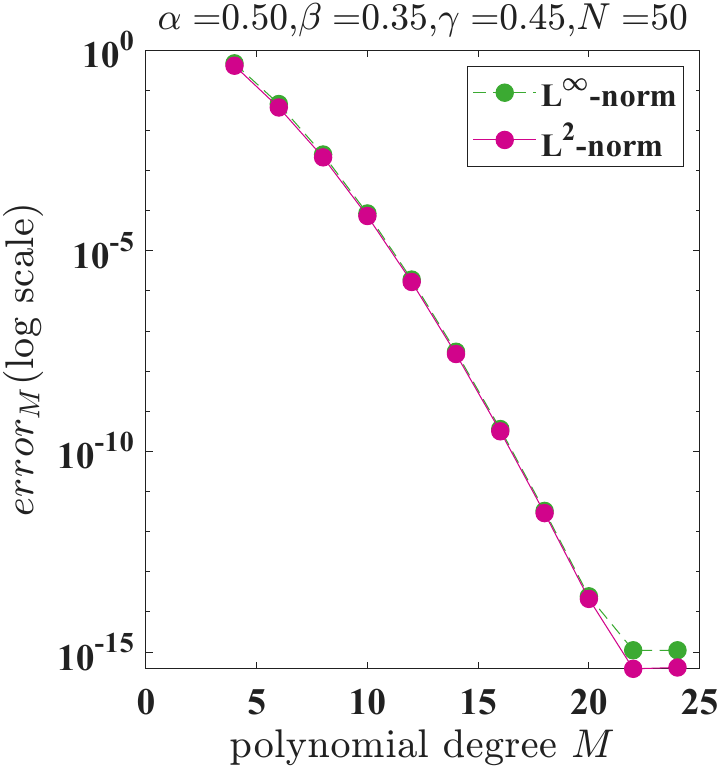}
  \end{minipage}
  \begin{minipage}[c]{0.28\textwidth}
    \centering
    \includegraphics[width=\textwidth]{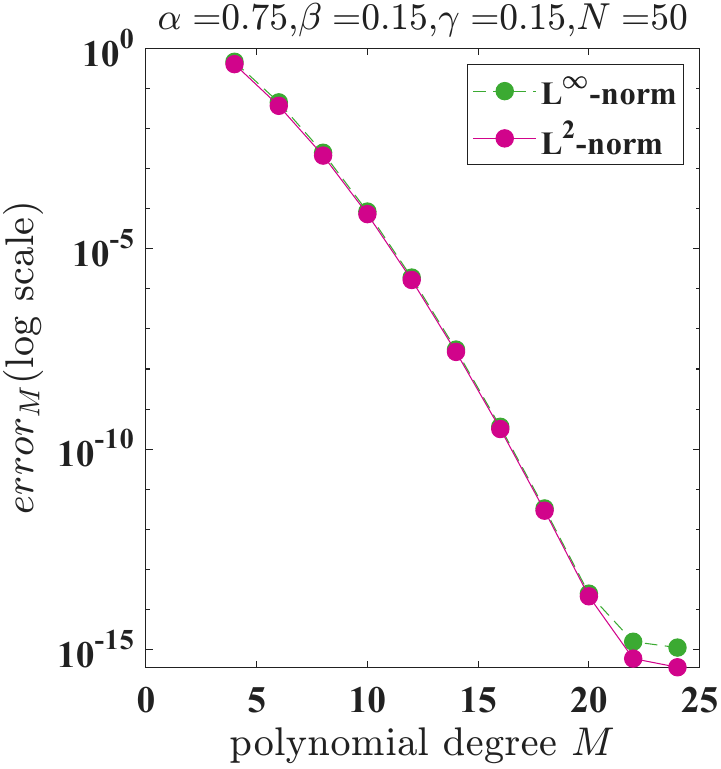}
  \end{minipage}
  \caption{Temporal and spatial spectral convergence of the CIM-CLG method under low temporal regularity ($\kappa=1/4$). The top row displays the errors of the CIM as a function of the number of quadrature points $N$ (with $M=24$). The bottom row presents the spatial errors of the CLG method in relation to the polynomial degree $M$ (with $N=50$). Common parameters are set as $a=1$, $b=1000$, $t_0=0.1$, $\Lambda=10$ and $t=0.5$.}
  \label{fig:Convergence}
\end{figure}

\begin{table}[t!]
\setlength{\abovecaptionskip}{0.15cm}
\centering
\caption{Spatial convergence order and $error_M$ of the CIM-CLG for the one-dimensional inhomogeneous problem  (Example \ref{Ne:exapmel2}) with $\Lambda=10$, $t_0=0.1$, $N=50$, $a=1$, $b=1000$, and $t=0.5$.}
\label{Tab:01}
  \begin{center}
\resizebox{1\textwidth}{!}{% 缩放到页面宽度
    \begin{tabular}{ccllllllll}
\toprule
\multirow{2}{*}{$M$} & \multicolumn{4}{c}{$\alpha=0.25$, $\beta=0.55$, $\gamma=0.75$} & \multicolumn{4}{c}{$\alpha=0.50$, $\beta=0.35$, $\gamma=0.45$}  \\
\cmidrule(r){2-5} \cmidrule(r){6-9}
&  $L^2$-error  &  $M^{-\operatorname{Orders}_s}$   &   $L^{\infty}$-error &  $M^{-\operatorname{Orders}_s}$ & $L^2$-error    &  $M^{-\operatorname{Orders}_s}$   &   $L^{\infty}$-error  &  $M^{-\operatorname{Orders}_s}$      \\
\midrule
4
&4.0726E-01  & ---                   &4.6689E-01  & ---                    &4.0726E-01   & ---            & 4.6690e-01    & ---             \\

6
&3.7110E-02  &$M^{-5.9081}$   &4.5263E-02   &$M^{-5.7554}$   &3.7110E-02  & $M^{-5.9082}$   &4.5263E-02   &$M^{-5.7554}$    \\

8
&2.1238E-03  &$M^{-9.9438}$   &2.5013E-03   &$M^{-10.0654}$  &2.1238E-03  & $M^{-9.9438}$  &2.5013E-03   & $M^{-10.0655}$   \\

10
&7.3416E-05  &$M^{-15.0793}$  &8.4986E-05   &$M^{-15.1566}$  &7.3416E-05  & $M^{-15.0793}$  &8.4986E-05   & $M^{-15.1566}$   \\

12
&1.6682E-06  &$M^{-20.7567}$  &1.9141E-06   &$M^{-20.8051}$  &1.6682E-06  & $M^{-20.7567}$   &1.9141E-06  & $M^{-20.8050}$   \\

14
&2.6709E-08  &$M^{-26.8211}$  &3.0494E-08   &$M^{-26.8536}$  &2.6709E-08  & $M^{-26.8211}$   &3.0494E-08  & $M^{-26.8536}$   \\

16
&3.1717E-10  &$M^{-33.2006}$  &3.6101E-10   &$M^{-33.2236}$  &3.1717E-10  & $M^{- 33.2006}$   &3.6101E-10  & $M^{-33.2236}$   \\
\bottomrule
   \end{tabular}
}
  \end{center}
\end{table}

\begin{table}[htbp]
\setlength{\abovecaptionskip}{0.15cm}
\centering
\caption{Spatial convergence order and $error_M$ of the CIM-CLG for the one-dimensional inhomogeneous problem  (Example \ref{Ne:exapmel2}) with $\Lambda=10$, $t_0=0.1$, $N=50$, $a=1$, $b=1000$, and $t=0.5$.}
\label{Tab:02}
  \begin{center}
\resizebox{0.52\textwidth}{!}{% 缩放到页面宽度
   \begin{tabular}{cclll}
\toprule\scriptsize
\multirow{2}{*}{$M$} & \multicolumn{4}{c}{$\alpha=0.75$, $\beta=0.15$, $\gamma=0.15$}\\ %& \multicolumn{4}{c}{$\alpha=1.0$, $\beta=1.0$, $\gamma=1.0$}  \\
\cmidrule(r){2-5} %\cmidrule(r){6-9}
&  $L^2$-error      &  $M^{-\operatorname{Orders}_s}$   &   $L^{\infty}$-error  &  $M^{-\operatorname{Orders}_s}$\\
%&  $L^2$-error      &  Order   &   $L^{\infty}$-error  &  Order  \\
\midrule
4
&4.0726E-01   & ---             & 4.6690E-01   & ---            \\%&6.3126E-02  & ---            & 4.4637E-02     & ---   \\

6
&3.7110E-02   & $M^{-5.9082}$   & 4.5263E-02  &$M^{-5.7554}$   \\%&1.2349E-03  & $M^{-9.7027}$  & 7.2975E-04   & $M^{-10.1454}$   \\

8
&2.1238E-03   & $M^{-9.9438}$   & 2.5013E-03  &$M^{-10.0655}$   \\%&1.9993E-05  & $M^{-14.3331}$ & 9.9806E-06   & $M^{-14.9194}$  \\

10
&7.3416E-05   & $M^{-15.0793}$  & 8.4986E-05  & $M^{-15.1566}$  \\%&2.5826E-07  & $M^{-19.4903}$ & 1.1835E-07   & $M^{-19.8739}$  \\

12
&1.6682E-06   & $M^{-20.7567}$  & 1.9141E-06  & $M^{-20.8051}$  \\%&2.6230E-09  & $M^{-25.1735}$  & 1.0815E-09   & $M^{-25.7530}$  \\

14
&2.6709E-08   & $M^{-26.8211}$  & 3.0494E-08  & $M^{-26.8536}$  \\%&2.1245E-11  & $M^{-31.2417}$  & 8.2949E-12   & $M^{-31.5952}$  \\

16
&3.1717E-10   & $M^{-33.2006}$  & 3.6101E-10  & $M^{-33.2240}$  \\%&1.3969E-13  & $M^{-37.6274}$  & 5.0293E-14   & $M^{-38.2346}$  \\
\bottomrule
    \end{tabular}
}
  \end{center}
\end{table}

Figure~\ref{fig:Convergence} and Tables~\ref{Tab:01}-\ref{Tab:02} demonstrate rapid temporal error decay and high-order spatial convergence, even under limited temporal regularity.

To further assess computational efficiency, we compare CIM-CLG with the graded-mesh $L1$-CLG method at the same spatial resolution $M=20$ and prescribed $L^2(\Omega)$ error tolerance $2.2025\times10^{-5}$. The comparison uses $a=1$, $b=100$, and $t=0.5$, as reported in Table~\ref{Tab:03}.

\begin{table}[htbp]
\centering
\caption{Fixed-tolerance comparison of CIM-CLG and
graded-mesh $L1$-CLG for Example~\ref{Ne:exapmel2}, with
$M=20$, $\kappa=1/4$, $a=1$, $b=100$, and $t=0.5$.
Both methods satisfy the prescribed $L^2(\Omega)$
error tolerance $2.2025\times10^{-5}$.}
\label{Tab:03}
\resizebox{0.78\textwidth}{!}{
\begin{tabular}{c c c c c c}
\toprule
$(\alpha,\beta,\gamma)$
& $N$ & $\widetilde N$
& CIM-CLG (s)
& $L1$-CLG (s)
& Speed-up \\
\midrule
$(0.25,0.55,0.75)$
& 11 & 1737 & 0.02728 & 5.93173 & $217.4\times$ \\
$(0.50,0.35,0.45)$
& 11 & 580 & 0.02385 & 1.28298 & $53.8\times$ \\
$(0.75,0.15,0.15)$
& 11 & 202 & 0.02331 & 0.43106 & $18.5\times$ \\
\bottomrule
\end{tabular}}
\end{table}

%\mred{As shown in Table~\ref{Tab:03}, CIM-CLG requires only $N=11$ contour nodes for all three parameter sets, whereas the graded-mesh $L1$-CLG method (For the graded-mesh $L1$-CLG method, the temporal grid is defined by $t_n=t(n/\widetilde N)^{r_g}$, $n=0,\ldots,\widetilde N$, where $r_g=6$ denotes the grading parameter; see \cite{Mustapha20} ) uses $\widetilde N=1737$, $580$, and $202$ time steps, respectively. At the same prescribed accuracy, CIM-CLG achieves CPU-time speed-up factors of approximately $217.4$, $53.8$, and $18.5$. Together with the FFT acceleration illustrated in Fig.~\ref{fig:FFT}, these results demonstrate the accuracy and computational efficiency of the proposed method for pointwise-in-time evaluation under limited temporal regularity.}

For the graded-mesh $L1$-CLG method, we adopt the temporal grid $t_n=t(n/\widetilde N)^{r_g}$, $n=0,\ldots,\widetilde N$, with grading parameter $r_g=6$ (see \cite{Mustapha20}). As shown in Table~\ref{Tab:03}, CIM-CLG requires only $N=11$ contour nodes for all three parameter sets, whereas the graded-mesh $L1$-CLG method uses $\widetilde N=1737$, $580$, and $202$ time steps, respectively. Under the same prescribed $L^2(\Omega)$ error tolerance, CIM-CLG achieves CPU-time speed-up factors of approximately $217.4$, $53.8$, and $18.5$. Together with the FFT acceleration illustrated in Fig.~\ref{fig:FFT}, these results demonstrate the computational efficiency of the proposed method for pointwise-in-time evaluation under limited temporal regularity.

\subsection{Example 3. (Homogeneous 2D problem)}\label{Examp:2D}
In this example, we investigate the two-dimensional homogeneous problem~\eqref{eq:problem} on $\Omega=(-1,1)^2$, with $f\equiv0$ and homogeneous Dirichlet boundary conditions. The initial condition is given by
\begin{itemize}[label=$\bullet$]
  \item $p_{0}(x,y)=\pi^2(1-x^2)(1-y^2)\sum_{k=0}^{5}
e^{-8^2\big(\big(x-\frac{4}{5}\cos\frac{k\pi}{3}\big)^2
+\big(y-\frac{4}{5}\sin\frac{k\pi}{3}\big)^2\big)}$.
\end{itemize}
The initial datum is nonnegative and smooth, featuring six localized Gaussian peaks with pronounced spatial
oscillations and steep gradients. The factor $(1-x^2)(1-y^2)$ ensures the homogeneous Dirichlet
boundary condition. Consequently, $p_0\in H^2(\Omega)\cap H_0^1(\Omega)\subset\dot H^{\widetilde r}(\Omega)$ with $1\leq\widetilde r\leq2$ satisfying the required spatial regularity assumptions. Its oscillatory, highly localized, and nonseparable structure provides a challenging test for assessing the spatial accuracy and robustness of the proposed CIM-CLG algorithm.

To examine the applicability of the proposed CIM-CLG algorithm to the full MT-TFJE model, we consider $K=J=2$. The corresponding Laplace-domain coefficient functions are $\mathcal{N}(z)=1+az^\alpha+a_1z^{\alpha_1}+a_2z^{\alpha_2}$ and $\mathcal{D}(z)=1+bz^\beta+b_1z^{\beta_1}+b_2z^{\beta_2}$. The fixed model parameters are chosen as $(a,b)=(1,0.5)$, $(\alpha_1,\alpha_2)=(0.10,0.18)$, $(\beta_1,\beta_2)=(0.05,0.10)$, $(a_1,a_2)=(0.35,0.20)$, and $(b_1,b_2)=(0.25,0.15)$. We investigate three representative fractional-order configurations $(\alpha,\beta,\gamma)=(0.25,0.55,0.75)$, $(\alpha,\beta,\gamma)=(0.50,0.35,0.45)$ and $(\alpha,\beta,\gamma)=(0.75,0.15,0.15)$ given in \eqref{eq:PPar}.

Numerical results under various parameter configurations are summarized in Tables \ref{Tab:07}, \ref{Tab:08}, and \ref{Tab:09}, while the temporal evolution of three representative cases is illustrated in Fig. \ref{fig:2DPP2}. These results not only validate the proposed CIM-CLG algorithm but also provide insight into the behavior of the solutions. Furthermore, they demonstrate the accuracy and robustness of the proposed algorithm in higher spatial dimensions.

\begin{table*}[t!]
 \setlength{\abovecaptionskip}{0.15cm}
\centering
\caption{Temporal $error_N$ of CIM-CLG method for the two-dimensional homogeneous multi-term Jeffreys problem  ($K=J=2$, Example \ref{Examp:2D}) with $\Lambda=500$, $t_0=0.01$, $(a,b)=(1,0.5)$, $(\alpha_1,\alpha_2)=(0.10,0.18)$, $(\beta_1,\beta_2)=(0.05,0.10)$, $(a_1,a_2)=(0.35,0.20)$, $(b_1,b_2)=(0.25,0.15)$, $M=24$ and $t=0.5$.}
\label{Tab:07}
\begin{center}
\resizebox{1\textwidth}{!}{% 缩放到页面宽度
\begin{tabular}{c|cccccc}
\Xhline{0.8pt}
\diagbox{($\alpha,\beta,\gamma$)}{$N$} & $4$     & $6$      &$10$      &$20$       & $30$  & $50$ \\
\hline
($0.25$, $0.55$, $0.75$)
&5.6202E-02    &2.7394E-02   &5.4049E-03  &1.1385E-04   &2.1367E-06    &2.1671E-09 \\

($0.50$, $0.35$, $0.45$)
&5.9376E-02    &3.2843E-02   &6.3433E-03  &1.4449E-04   &2.6968E-06    &1.9867E-09 \\

($0.75$, $0.15$, $0.15$)
&6.6242E-02    &3.9826E-02   &7.4267E-03  &1.6197E-04   &3.2552E-06    &1.8517E-09  \\
%($1.00$, $1.00$, $1.00$)
%& 1.9945E-01   & 6.0303E-02    & 1.3810E-02  & 5.8716E-09    & 2.8600E-12      & 2.7158E-12  \\
\Xhline{0.8pt}
\end{tabular}
}
\end{center}
\end{table*}

\begin{table*}[htbp]
\setlength{\abovecaptionskip}{0.15cm}
\centering
\caption{Spatial convergence orders and $error_M$ of the CIM-CLG method for the two-dimensional homogeneous problem (Example \ref{Examp:2D}) with $\Lambda=500$, $t_0=0.01$, $(a,b)=(1,0.5)$, $(\alpha_1,\alpha_2)=(0.10,0.18)$, $(\beta_1,\beta_2)=(0.05,0.10)$, $(a_1,a_2)=(0.35,0.20)$, $(b_1,b_2)=(0.25,0.15)$, $N=50$ and $t=0.5$.}
\label{Tab:08}
\resizebox{1\textwidth}{!}{% 缩放到页面宽度
\begin{tabular}{cccccccccc}
\toprule
\multirow{2}{*}{$M$} & \multicolumn{4}{c}{$\alpha=0.25$, $\beta=0.55$, $\gamma=0.75$} & \multicolumn{4}{c}{$\alpha=0.50$, $\beta=0.35$, $\gamma=0.45$}  \\
\cmidrule(r){2-5} \cmidrule(r){6-9}
&  $L^2$ -error      &  $M^{-\operatorname{Orders}_s}$   &   $L^{\infty}$ -error &  $M^{-\operatorname{Orders}_s}$
&  $L^2$ -error      &  $M^{-\operatorname{Orders}_s}$   &   $L^{\infty}$ -error &  $M^{-\operatorname{Orders}_s}$         \\
\midrule
6
&2.7575E-02  & ---           &2.9764E-02   & ---           &2.8732E-02   & ---           &3.2506E-02   & ---               \\

12
&6.5267E-03  &$M^{-2.0789}$  &1.0465E-02   &$M^{-1.5079}$  &6.9632E-03   &$M^{-2.0448}$  &1.1773E-02   &$M^{-1.4652}$               \\

24
&7.6650E-05  &$M^{-6.4119}$  &1.5040E-04   &$M^{-6.1207}$  &9.0992E-05   &$M^{-6.2579}$  &1.8142E-04   &$M^{-6.0201}$             \\
\bottomrule
\end{tabular}
}
\end{table*}

\begin{table*}[t!]
\setlength{\abovecaptionskip}{0.15cm}
\centering
\caption{Spatial convergence orders and $error_M$ of the CIM-CLG method for the two-dimensional homogeneous problem (Example \ref{Examp:2D}) with $\Lambda=500$, $t_0=0.01$, $(a,b)=(1,0.5)$, $(\alpha_1,\alpha_2)=(0.10,0.18)$, $(\beta_1,\beta_2)=(0.05,0.10)$, $(a_1,a_2)=(0.35,0.20)$, $(b_1,b_2)=(0.25,0.15)$, $N=50$ and $t=0.5$.}
\label{Tab:09}
\resizebox{0.52\textwidth}{!}{% 缩放到页面宽度
\begin{tabular}{ccccc}
\toprule
\multirow{2}{*}{$M$} & \multicolumn{4}{c}{$\alpha=0.75$, $\beta=0.15$, $\gamma=0.15$} \\%& \multicolumn{4}{c}{$\alpha=1.00$, $\beta=1.00$, $\gamma=1.00$}  \\
\cmidrule(r){2-5} %\cmidrule(r){6-9}
&  $L^2$-error      &  $M^{-\operatorname{Orders}_s}$   &   $L^{\infty}$-error  &  $M^{-\operatorname{Orders}_s}$ \\
%&  $L^2$-error      &  Order   &   $L^{\infty}$-error  &  Order  \\
\midrule
6
&2.8613E-02  & ---          &3.3058E-02   & ---         \\ % &2.0282E-02    & ---            & 1.3753E-02     & ---               \\

12
&7.0186E-03  &$M^{-2.0274}$ &1.2134E-02   &$M^{-1.4459}$ \\  %&1.1474E-03   & $M^{-2.0719}$   & 5.4455E-04    &$M^{-2.3292}$               \\

24
&9.6064E-05  &$M^{-6.1911}$ &1.9273E-04   &$M^{-5.9764}$ \\  %&1.1411E-04   & $M^{-1.6649}$   & 3.3661E-05    & $M^{-2.0080}$              \\
\bottomrule
\end{tabular}
}
\end{table*}

\begin{remark}
The initial datum in Example~\ref{Examp:2D} consists of six smooth, sharply localized Gaussian peaks satisfying the homogeneous Dirichlet boundary conditions. These localized structures are not fully resolved at low polynomial degrees. In Tables~\ref{Tab:08} and \ref{Tab:09}, the observed $L^2$ convergence orders are approximately $2.03$-$2.08$ for $M=6$ to $12$, increasing to $6.19$-$6.41$ for $M=12$ to $24$. The corresponding $L^\infty$ orders also increase markedly. These effective orders are computed over finite polynomial-degree intervals and are consistent with a pre-asymptotic resolution effect rather than an intrinsic second-order limitation. In one dimension, Theorem~\ref{eq:homoconverg} provides the Sobolev-dependent estimate $\mathcal O(M^{-\widetilde r})$, $1\leq\widetilde r\leq2$, under its stated assumptions, without excluding faster convergence for smoother data.
\end{remark}

\begin{figure}[hbpt]
  \centering
\begin{minipage}[c]{0.88\textwidth}
 \centering
 \centerline{\includegraphics[width=\textwidth]{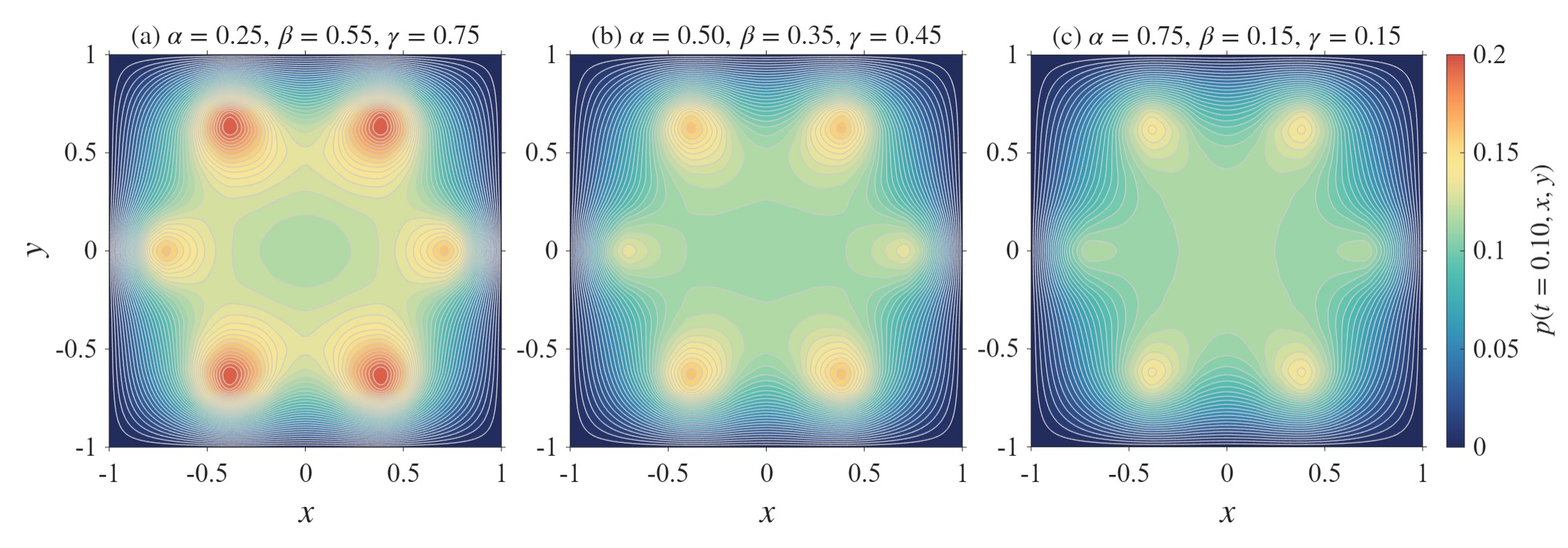}}
% \centerline{$(a)$ }
\end{minipage}
\begin{minipage}[c]{0.88\textwidth}
 \centering
 \centerline{\includegraphics[width=\textwidth]{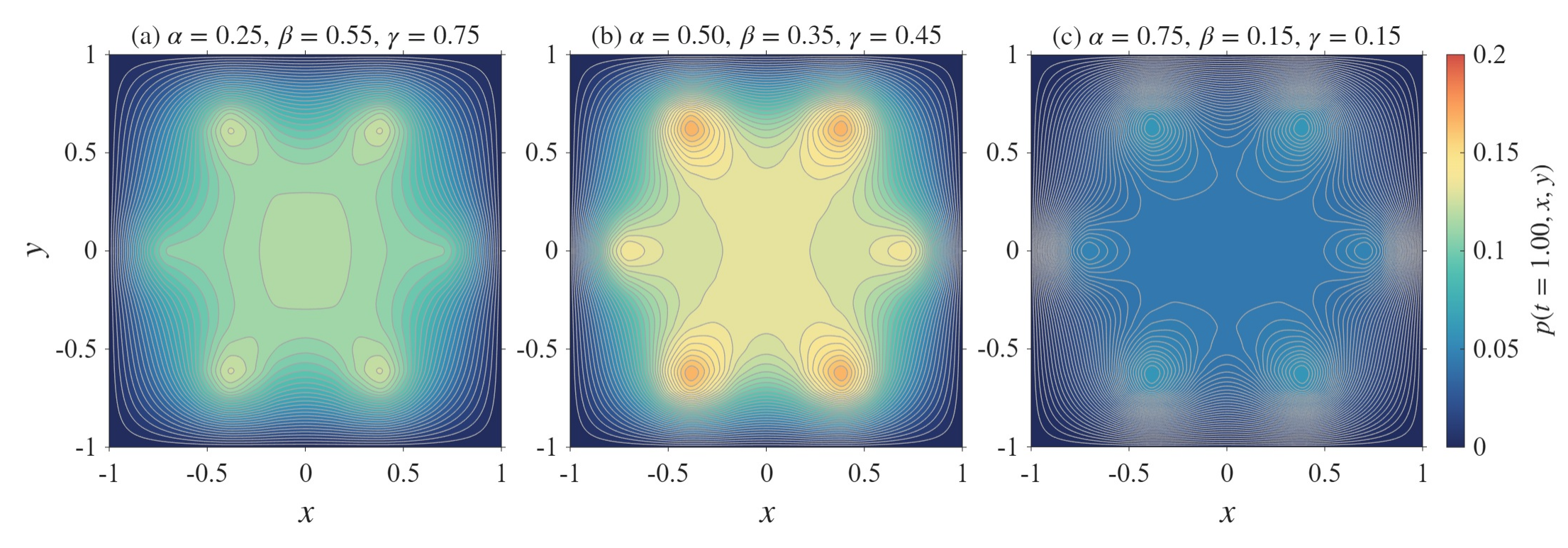}}
% \centerline{$(b)$ }
\end{minipage}
\begin{minipage}[c]{0.88\textwidth}
 \centering
 \centerline{\includegraphics[width=\textwidth]{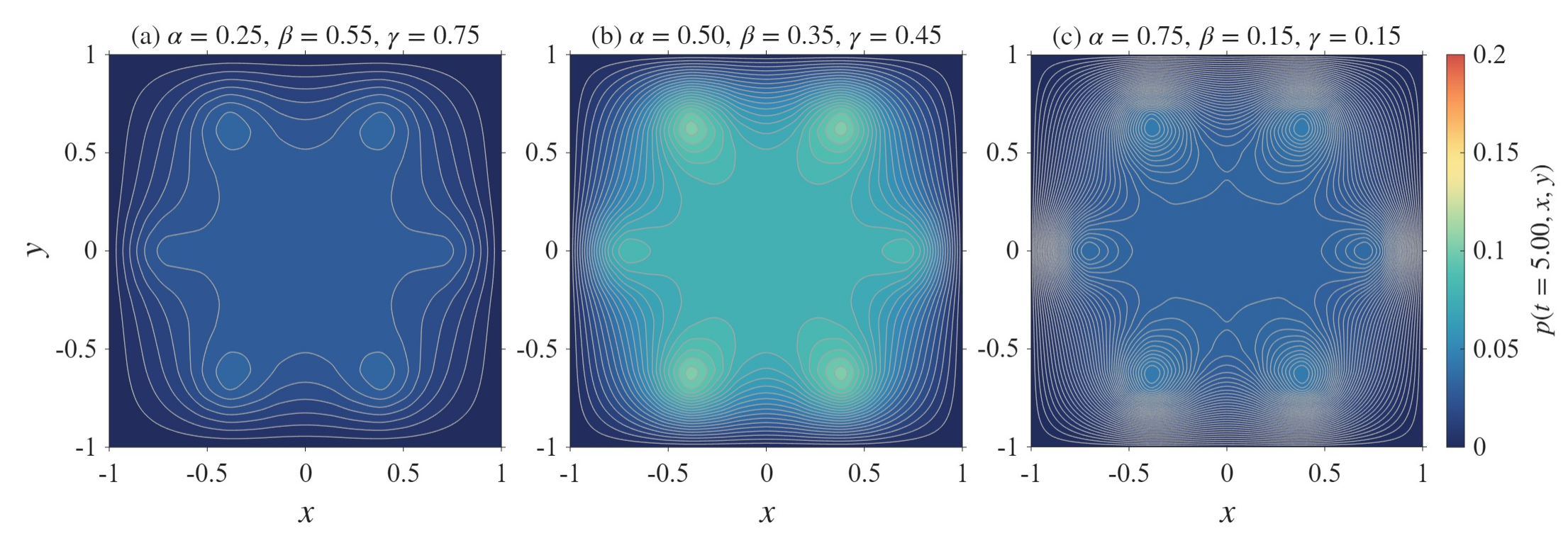}}
% \centerline{$(c)$ }
\end{minipage}
  \caption{Temporal evolution of the numerical solution $p(t,x,y)$ to Problem~\eqref{eq:problem} with the initial condition specified in Example~\ref{Examp:2D} and $f=0$. The left, middle, and right columns correspond to the fractional-order configurations $(\alpha,\beta,\gamma)=(0.25,0.55,0.75)$, $(0.50,0.35,0.45)$, and $(0.75,0.15,0.15)$, respectively. Within each column, the top, middle, and bottom panels depict the numerical solutions at three successive time levels. The remaining parameters are defined in \eqref{eq:PPar} with $M=20$, $N=50$, $\Lambda=500$, and $t_0=0.01$.}
\label{fig:2DPP2}
\end{figure}

\subsection{CPU Time and Computational Cost}
\label{subsec:cpu-time}
To further assess the computational efficiency of the CIM-CLG algorithm, we conduct CPU-time tests using Examples~\ref{Ne:exapmel2} and~\ref{Examp:2D}. The results are presented in Figs.~\ref{fig:FFT} and~\ref{fig:cput}.

Fig.~\ref{fig:FFT} compares the FFT-based and direct CGL implementations with $N=50$. For Example~\ref{Ne:exapmel2}, the FFT-based implementation becomes more efficient as $M$ increases. At $M=256$, the total CPU time decreases from approximately $0.075$ s to $0.043$ s, a reduction of about $43\%$. In contrast, both implementations require approximately $25$ s for Example~\ref{Examp:2D} at $M=256$.

This difference is closely related to how often the CGL transformation is performed. Using the FFT, the transformation requires $\mathcal{O}(M\log M)$ and $\mathcal{O}(M^2\log M)$ operations in one and two dimensions, respectively, compared with $\mathcal{O}(M^2)$ and $\mathcal{O}(M^3)$ for the direct dense-matrix transforms. In the present implementation of Example~\ref{Ne:exapmel2}, the transformed source is evaluated and projected at every contour node. Thus, the savings from FFT accumulate over the $N$ nodes. In Example~\ref{Examp:2D}, however, $f=0$, and the initial data are transformed and projected only once. The subsequent computation mainly consists of solving $N$ complex Galerkin systems with $(M-1)^2$ unknowns.

\begin{figure}[hbpt]
  \centering
\begin{minipage}[c]{0.48\textwidth}
 \centering
 \centerline{\includegraphics[width=0.9\textwidth]{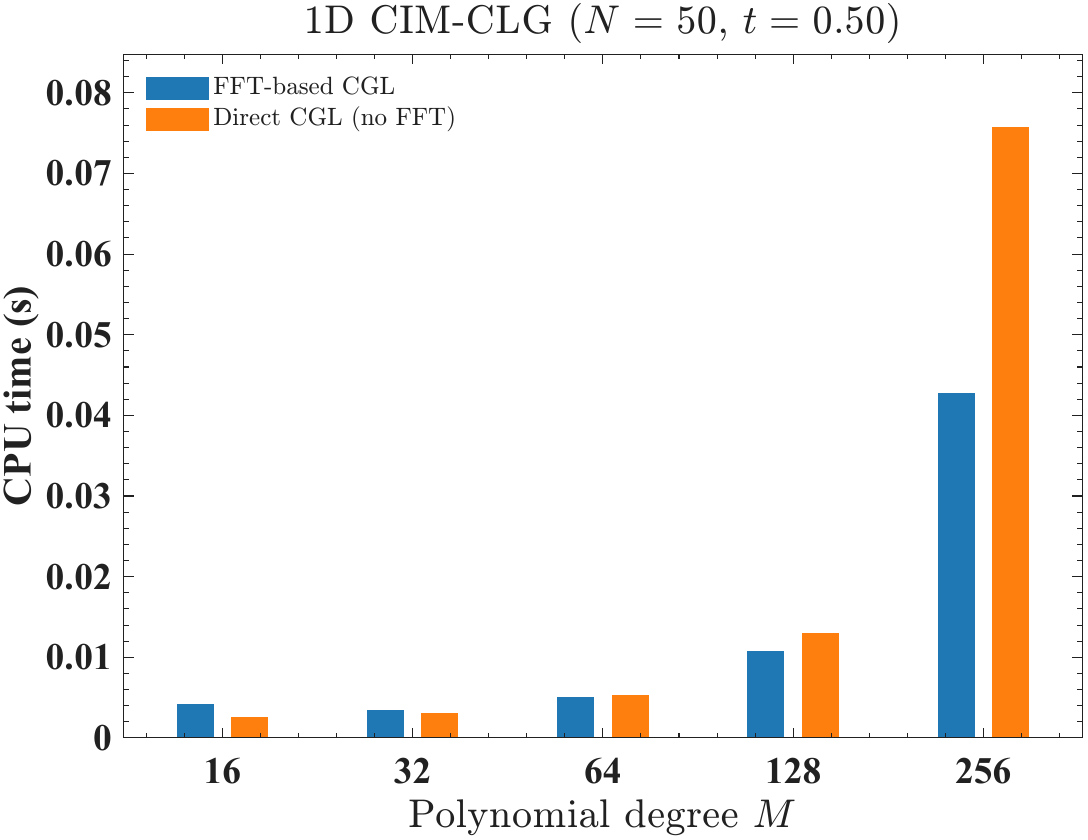}}
% \centerline{$(a)$ }
\end{minipage}
\begin{minipage}[c]{0.48\textwidth}
 \centering
 \centerline{\includegraphics[width=0.9\textwidth]{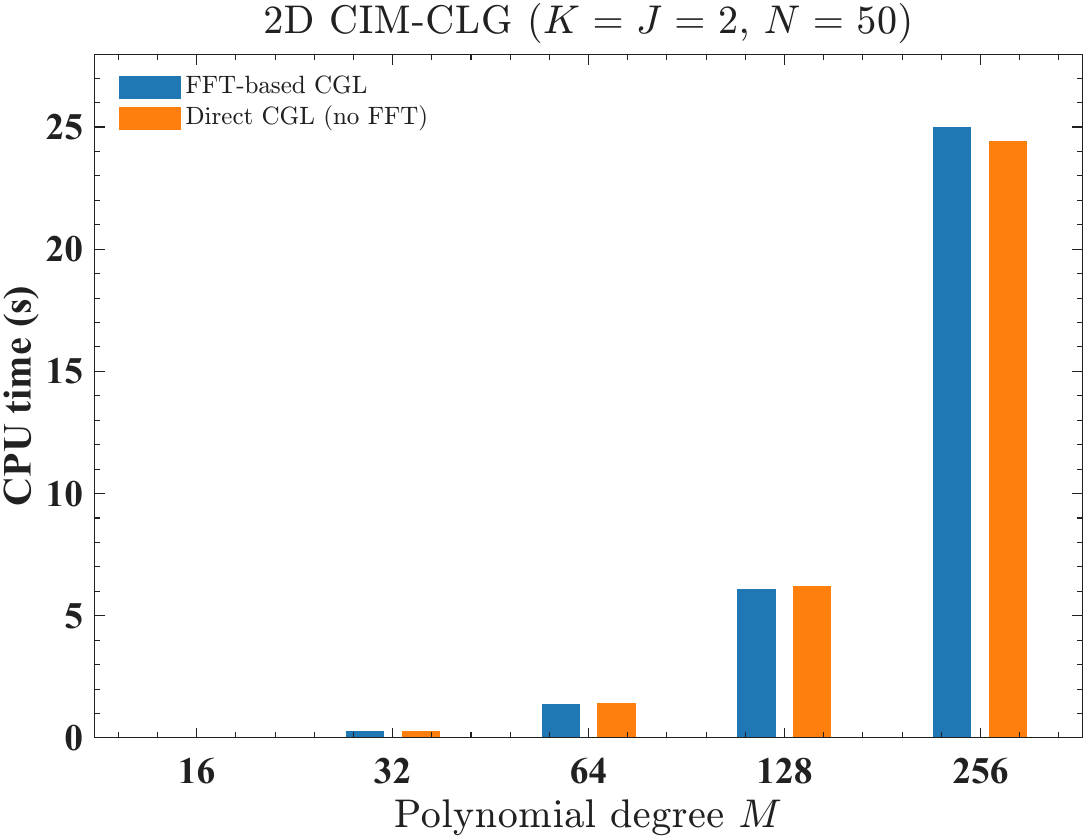}}
% \centerline{$(b)$ }
\end{minipage}
  \caption{Comparison of total CPU times for the FFT-based and direct CGL implementations. (left) 1D Example~\ref{Ne:exapmel2} with $N=50$; (right) 2D Example~\ref{Examp:2D} with $K=J=2$ and $N=50$.}
\label{fig:FFT}
\end{figure}

The CPU-time breakdowns in Fig.~\ref{fig:cput} further illustrate these differences. For $M=24$, the total time increases with $N$ in both examples. The one-dimensional cost is distributed among several computational stages, whereas the two-dimensional computational cost is dominated by the solution of the linear systems. This agrees with the structure of the Galerkin systems: the one-dimensional matrix is pentadiagonal, while the two-dimensional matrix has $\mathcal{O}(M^2)$ nonzero entries and requires a sparse solve at each contour node. The actual solution cost also depends on the factorization fill-in.

We further estimate the computational cost of the complete implementation. The dense Chebyshev-to-Legendre projection requires $\mathcal{O}(M^2)$ and $\mathcal{O}(M^3)$ operations per application in one and two dimensions, respectively. The present implementation also involves a one-time spatial preprocessing cost of $\mathcal{O}(M^3)$. Let $Q_t$ be the number of output times. The principal operation counts are estimated by
\begin{gather*}
W_{\mathrm{1D}}=\mathcal O(M^3+NM^2+NQ_tM)+W_{\widehat f,1}+W_{\mathrm{rec},1},\\
W_{\mathrm{2D}}=\mathcal O(M^3+NM^2+NQ_tM^2)+W_{\widehat f,2}+W_{\mathrm{solve},2}+W_{\mathrm{rec},2},
\end{gather*}
where $W_{\widehat f,d}$ denotes the cost of evaluating the transformed source at the required contour nodes and spatial sampling points. This cost depends on the representation of $\widehat f$ and vanishes for homogeneous problems. $W_{\mathrm{solve},2}$ denotes the total cost of the $N$ two-dimensional sparse linear solves, which depends on matrix ordering and factorization fill-in, and $W_{\mathrm{rec},n}$ denotes the solution reconstruction cost ($n=1,2$). Sequential accumulation avoids storing all contour-node solutions. The output coefficients require $\mathcal{O}(Q_t(M-1)^n)$ storage, in addition to the projection matrices and solver workspace.

\begin{figure}[hbpt]
  \centering
\begin{minipage}[c]{0.48\textwidth}
 \centering
 \centerline{\includegraphics[width=0.9\textwidth]{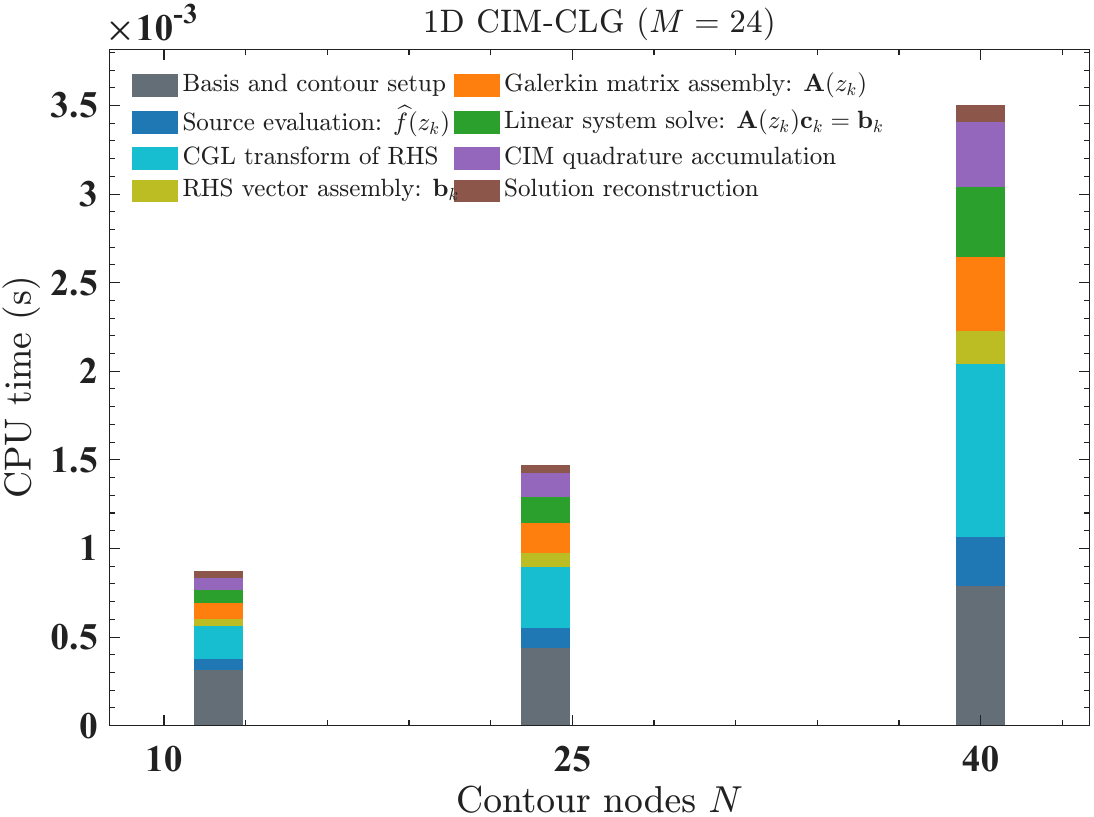}}
% \centerline{$(a)$ }
\end{minipage}
\begin{minipage}[c]{0.48\textwidth}
 \centering
 \centerline{\includegraphics[width=0.9\textwidth]{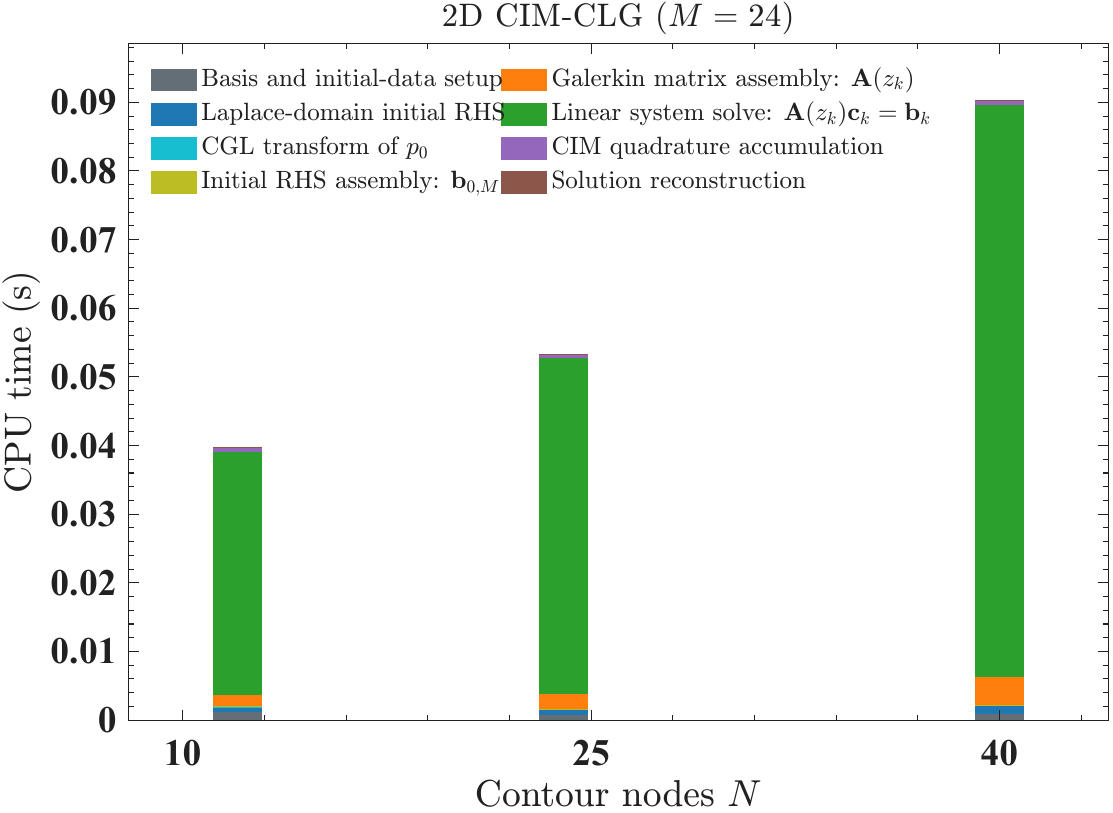}}
% \centerline{$(b)$ }
\end{minipage}
 \caption{CPU-time breakdowns of the CIM-CLG method for $M=24$ and varying numbers of contour nodes $N$. (left) the 1D manufactured-solution problem in Example~\ref{Ne:exapmel2}; (right) the 2D initial-value problem in Example~\ref{Examp:2D} with $K=J=2$.}
\label{fig:cput}
\end{figure}

Compared with conventional time-marching schemes, especially the $L1$ scheme (see \cite{Mustapha20}), whose direct evaluation of fractional history terms typically requires $\mathcal{O}(\tilde{N}^2(M-1)^n)$ operations for $\tilde{N}$ time steps (see e.g.,~\cite{Diethelm22}), the CIM-CLG method requires only $N$ independent contour-node solves. When $N\ll\tilde{N}$ and few output times are needed, it can reduce computational cost and memory usage (see Table~\ref{Tab:03}). Thus, although FFT accelerates the CGL transformation, the overall computational efficiency also depends on projection, assembly, and, particularly, the solution of the Galerkin linear systems. The CPU-time results further indicate that employing more efficient solvers for the underlying elliptic-type systems, such as preconditioned iterative methods, could further reduce the computational cost, especially in two dimensions.

%%%%%%%%%%%% Section 6 %%%%%%%%%%%%%%%%%%%%%%%%%%%%%%%%%%%%%%%%%%%%%%%%
\section{Concluding remarks}
\label{sec:conclusions}
In this work, we present a comprehensive investigation of the generalized multi-term time-fractional Jeffreys-type equation (MT-TFJE), integrating microscopic stochastic modeling, rigorous mathematical analysis, and high-accuracy numerical computation. The governing equation is derived from a CTRW framework, providing a physical interpretation of the multi-term memory structure and elucidating its role in crossover diffusion phenomena characterized by distinct scaling regimes. We establish the well-posedness, stability, and Sobolev regularity of the model, deriving a priori estimates for smooth and nonsmooth data that also underpin the subsequent numerical analysis. Furthermore, we develop a high-accuracy numerical algorithm, called CIM-CLG, which combines a contour integral method (CIM) for temporal approximation with a Chebyshev-Legendre Galerkin (CLG) method for spatial discretization. The theoretical analysis establishes exponential convergence of the temporal quadrature under suitable analyticity assumptions and algebraic spatial convergence rates determined by the regularity of the data. The resulting algorithm is highly parallelizable, avoids stringent temporal smoothness requirements, and offers favorable computational and storage efficiency. Extensive numerical experiments in one and two dimensions demonstrate its accuracy, efficiency, and robustness, confirm the predicted convergence behavior, and illustrate the dynamic characteristics of the underlying diffusion processes.

The spatial CLG discretization is naturally suited to regular computational domains, which limits its direct applicability to irregular geometries. One possible extension is to embed an irregular domain into a larger, geometrically simple computational domain, although appropriate treatment of the original boundary conditions is required and additional computational costs may arise. Nevertheless, the efficiency of the CIM allows high accuracy to be achieved with relatively few quadrature points in the numerical tests (e.g., $N=16$ in selected cases). Further computational improvements may be obtained by incorporating the acceleration strategy proposed in \cite{Ma2023a}. Future research will focus on extending the spatial discretization to irregular domains through suitable coordinate mappings, domain decomposition, or fictitious-domain methods, while retaining the accuracy and parallelizability of the proposed framework.

\appendix

\section{Appendix}
\subsection{\textbf{Proof of Lemma \ref{Lem:Lemma1}}}\label{proof:prior}
\begin{proof}
For part $(i)$, let $T>0$ be fixed and set $\delta=\max\{1,t^{-1}\}$ for $0<t\le T$. We distinguish two cases. If $0<t\le\min\{1,T\}$, then $\delta=t^{-1}\ge1$, and the resolvent estimates in Lemma~\ref{Lem:function} apply directly. If $1<t\le T$, we choose the fixed contour $\Gamma_{\theta,1}$. In both cases, the contributions from the two rays and the circular arc of $\Gamma_{\theta,\delta}$ can be estimated separately.

Indeed, by Lemma~\ref{Lem:function} and the contour representation of $\mathcal{E}(t)$, we obtain
\begin{displaymath}
\begin{aligned}
\left\|\mathcal{E}(t)\right\|_{\Psi}
&\le C\int_{\Gamma_{\theta,\delta}} e^{t\,{\rm Re}(z)}\,
      \left\|\mathcal{H}(z)\right\|_{\Psi}\,|\mathrm{d}z|
\le C\int_{\Gamma_{\theta,\delta}} e^{t\,{\rm Re}(z)}\,|z|^{-1}\,|\mathrm{d}z| \\
&\le C\int_{\delta}^{\infty} e^{t r\cos\theta}\,r^{-1}\,\mathrm{d}r
     +C\int_{-\theta}^{\theta} e^{t\delta\cos\varphi}\,\mathrm{d}\varphi \\
&\quad(\text{set } s=tr,\ r=s/t,\ \cos\theta<0) \\
&\le C\int_{t\delta}^{\infty} e^{s\cos\theta}\,s^{-1}\,\mathrm{d}s
     +C\int_{-\theta}^{\theta} e^{t\delta}\,\mathrm{d}\varphi 
\le C_T,
\end{aligned}
\end{displaymath}
where $C_T>0$ is independent of $t$. Similarly, for the time derivative, we have
\begin{displaymath}\begin{aligned}
\big\|\partial_t\mathcal{E}(t)\big\|_{\Psi}
\leq C \int_{\Gamma_{\theta,\delta}}e^{t {\rm Re}(z)}|z|
       \left\|\mathcal{H}(z)\right\|_{\Psi}|\mathrm{d}z|
%\leq C \int_{\Gamma_{\theta,\delta}}e^{t|z|\cos(\theta)}|dz|\\
%&\leq C\int_{1/t}^{\infty}e^{tr\cos(\theta)}dr
%          + C\int_{-\theta}^{\theta}e^{\delta t\cos\varphi}\delta d\varphi
\leq C_T\,t^{-1}.
\end{aligned}\end{displaymath}
By similar reasoning, we obtain $\big\|\partial_t^{m}\mathcal{E}(t)\big\|_{\Psi}\leq C_Tt^{-m}$. For the operator $\mathcal{A}\partial_{t}^{m}\mathcal{E}(t)$, we have
\begin{displaymath}\begin{aligned}%\label{eq:estmatEt}
\big\|\mathcal{A}\partial_{t}^{m}\mathcal{E}(t)\big\|_{\Psi}
&\leq C \int_{\Gamma_{\theta,\delta}}e^{t {\rm Re}(z)}|z|^{m}
        \left\|\mathcal{A}\mathcal{H}(z)\right\|_{\Psi}|\mathrm{d}z|\\
&\leq C \int_{\Gamma_{\theta,\delta}}e^{t {\rm Re}(z)}|z|^{\alpha-\beta+\gamma-1+m}|\mathrm{d}z|
%&\leq  C\int_{1/t}^{\infty}e^{tr\cos(\theta)}r^{\alpha-\beta+\gamma-1+m}dr\\
%&\quad   + C\int_{-\theta}^{\theta}e^{\delta t\cos(\varphi)}\delta^{\alpha-\beta+\gamma-1+m}\delta d\varphi\\
\leq C_T\,t^{\beta-\alpha-\gamma-m}.
\end{aligned}\end{displaymath}
This completes the proof of part ($i$) of Lemma \ref{Lem:Lemma1}. Part ($ii$) follows by similar arguments.\qed
\end{proof}

\subsection{\textbf{Proof of Theorem \ref{thm:mild-wellposed}}}\label{proof:well-posedness}
\begin{proof}
Let $(\lambda_j,\varphi_j)_{j\ge1}$ be the eigenpairs of $\mathcal A$ introduced in Section~\ref{sec2}, with $0<\lambda_1\le\lambda_2\le\cdots$. By \eqref{eq:denoteH} and \eqref{eq:solvers}, we have $\mathcal{E}(t)\varphi_j=e_j(t)\varphi_j$ and $\mathcal{F}(t)\varphi_j=q_j(t)\varphi_j$, where $\widehat{e}_j(z):=\frac{\eta(z)}{z(\eta(z)+\lambda_j)}$ and $\widehat{q}_j(z):=\frac{g(z)\eta(z)}{z(\eta(z)+\lambda_j)}$.
By Prop.~\ref{prop:PDF}, $\eta$ is a CBF. The parameter conditions imply that $\eta(z)\to0$ as $z\to0^+$ and $\eta(z)\sim\frac{a}{b}z^{\alpha+\gamma-\beta}$ as $z\to+\infty$. Since $\beta\le\gamma$ and $\alpha>0$, we have $\alpha+\gamma-\beta>0$, and hence $\eta(z)\to+\infty$.

For each $\lambda_j>0$, the function $x/(x+\lambda_j)$ is a CBF. Since this class is closed under composition \cite[Corollary~7.6]{Schilling2009}, $\eta(z)/(\eta(z)+\lambda_j)$ is also a CBF. By \cite[Theorem~6.2]{Schilling2009}, $\widehat e_j$ is a Stieltjes function. As $\widehat{e}_j(z)\to0$ when $z\to\infty$, its inverse Laplace transform $e_j$ is completely monotone. The Stieltjes representation further gives $e_j(0^+)=\lim_{z\to\infty}z\widehat e_j(z)=\lim_{z\to\infty}\frac{\eta(z)}{\eta(z)+\lambda_j}=1$. Hence $0\le e_j(t)\le1$ for $t>0$, and $\|\mathcal E(t)\|_\Psi=\sup_{j\ge1}e_j(t)\le1$. Set $e_j(0)=1$. For $v\in L^2(\Omega)$ and $t_0\ge0$, Parseval's identity yields
\begin{equation*}
\|\mathcal E(t)v-\mathcal E(t_0)v\|_{L^2(\Omega)}^2
=\sum\nolimits_{j=1}^\infty|e_j(t)-e_j(t_0)|^2|(v,\varphi_j)|^2\to0, \quad t\to t_0.
\end{equation*}
Indeed, each $e_j$ is continuous and the summands are bounded by $4|(v,\varphi_j)|^2$. Thus $\mathcal E$ is strongly continuous on $[0,\infty)$, with $\mathcal E(0)=I$.

For $\mathcal{F}$, the definition of $g$ gives 
$\widehat q_j(z)=(z^\gamma+az^{\alpha+\gamma}+\sum_{k=1}^{K}a_kz^{\alpha_k+\gamma}+\lambda_j(1+bz^\beta+\sum_{\ell=1}^{J}b_\ell z^{\beta_\ell}))^{-1}$. All powers in the denominator lie in $(0,1]$. Since nonnegative linear combinations of complete Bernstein functions remain in the same class, the denominator is a complete Bernstein function. Its reciprocal is Stieltjes by \cite[Theorem~7.3]{Schilling2009}. Moreover, $\widehat q_j(\infty)=0$ and $\widehat q_j(0^+)={\lambda_j}^{-1}$. Thus the constant and $z^{-1}$ terms in the Stieltjes representation vanish. There exists a positive measure $\mu_j$ such that $\widehat q_j(z)=\int_0^\infty\frac{\mu_j(\mathrm{d}s)}{z+s}$ and $q_j(t)=\int_0^\infty e^{-st}\mu_j(\mathrm{d}s)$. In particular, $q_j$ is nonnegative and completely monotone. Tonelli's theorem gives $\int_0^\infty q_j(t)\,\mathrm{d}t=\lim_{z\to0^+}\widehat q_j(z)={\lambda_j}^{-1}$.

If $\lambda_\ell\ge\lambda_j$, then $\widehat{q}_j(z)-\widehat{q}_\ell(z)=(\lambda_\ell-\lambda_j)\frac{\widehat{q}_\ell(z)}{\eta(z)+\lambda_j}$. Since $\eta+\lambda_j$ is a CBF, its reciprocal is Stieltjes. Let $r_j$ be the nonnegative inverse Laplace transform of $1/(\eta+\lambda_j)$. The convolution theorem gives $q_j(t)-q_\ell(t)=(\lambda_\ell-\lambda_j)(r_j*q_\ell)(t)\ge0$.
The identity holds initially almost everywhere and hence everywhere on $(0,\infty)$ by continuity. Consequently,
\begin{equation}\label{eq:mild-Fbound}
\left\|\mathcal{F}(t)\right\|_\Psi=q_1(t),
\qquad
\int_0^\infty\left\|\mathcal{F}(t)\right\|_\Psi\,\mathrm{d}t
=\frac1{\lambda_1}.
\end{equation}
The operator-valued kernel $\mathcal{F}$ is also measurable. Indeed, since $q_j$ is nonincreasing, $tq_j(t)\le\int_0^tq_j(s)\,\mathrm{d}s\le\frac1{\lambda_j}$. Let $\mathcal{F}_N(t)$ be the spectral truncation of $\mathcal F(t)$ to the first $N$ eigenfunctions. For every $\varepsilon>0$,
\begin{equation}\label{eq:mild-F-truncation}
\sup_{t\ge\varepsilon}\left\|\mathcal{F}(t)-\mathcal{F}_N(t)\right\|_\Psi
\le(\varepsilon\lambda_{N+1})^{-1}\to0.
\end{equation}
Each $\mathcal F_N$ is norm-continuous. Hence $\mathcal{F}$ is norm-continuous on $(0,\infty)$ and therefore Bochner measurable. Together with \eqref{eq:mild-Fbound}, this yields $\mathcal{F}\in L^1(0,\infty;\mathcal{L}(L^2(\Omega)))$.

Consider now the representation \eqref{eq:integral}. The function $\mathcal E(t)p_0$ is continuous on $[0,T]$. The convolution $\mathcal F*f$ is also continuous. To see this, extend $\mathcal F|_{(0,T)}$ and $f$ by zero to the real line. Then $\|(\mathcal{F}*f)(\cdot+h)-(\mathcal{F}*f)(\cdot)\|_{L^\infty(\mathbb{R};L^2(\Omega))}\le
\int_{\mathbb{R}}\|\mathcal{F}(s+h)-\mathcal F(s)\|_{\Psi}\mathrm{d}s\|f\|_{L^\infty(0,T;L^2(\Omega))}
\to0$ as $h\to0$, by translation continuity in $L^1$. Moreover,
\begin{equation}\label{eq:mild-conv-zero}
\|(\mathcal F*f)(t)\|_{L^2(\Omega)}
\le\|f\|_{L^\infty(0,T;L^2(\Omega))}\int_0^t\|\mathcal F(s)\|_\Psi\,\mathrm{d}s\to 0, \qquad t\to0^+.
\end{equation}
Thus \eqref{eq:integral} defines $p\in C([0,T];L^2(\Omega))$ with $p(0)=p_0$.

Set $C_T:=\int_0^T\|\mathcal F(s)\|_\Psi\,\mathrm{d}s$. Using $\|\mathcal E(t)\|_\Psi\le1$, we obtain
\begin{equation}\label{eq:mild-estimate}
\begin{aligned}
\|p(t)\|_{L^2(\Omega)}
&\le\|\mathcal E(t)p_0\|_{L^2(\Omega)}
+\int_0^t\|\mathcal F(t-s)\|_\Psi
\|f(s)\|_{L^2(\Omega)}\,\mathrm{d}s\\
&\le\|p_0\|_{L^2(\Omega)}
+C_T\|f\|_{L^\infty(0,T;L^2(\Omega))}.
\end{aligned}
\end{equation}
Taking the supremum over $t\in[0,T]$ proves the first inequality in \eqref{eq:mild-stability}. Since $C_T\le\int_0^\infty\|\mathcal F(s)\|_\Psi ds=1/\lambda_1$ by \eqref{eq:mild-Fbound}, the inequality in \eqref{eq:mild-stability} follows as well.

It remains to verify \eqref{eq:mild-volterra}. Extend $f$ by zero for $t>T$ and define $p$ on $[0,\infty)$ by the same representation. The estimates above show that $p$ is bounded, so its Laplace transform exists for every real $z>0$. Taking the transform of the $j$-th spectral coefficient gives $\widehat p_j(z)=\widehat e_j(z)p_{0,j}+\widehat q_j(z)\widehat f_j(z)$, which is equivalent to
\begin{equation}\label{eq:mild-phat-2}
\widehat p_j(z)
+\frac{\lambda_j}{\eta(z)}
\widehat p_j(z)
=\frac{p_{0,j}}z
+\frac{g(z)}z\widehat f_j(z).
\end{equation}
Since $\eta$ is a nonzero CBF, $1/\eta$ is Stieltjes \cite[Theorem~7.3]{Schilling2009}. The same holds for
\begin{equation}\label{eq:mild-gz}
\frac{g(z)}z=\Big(z^\gamma+az^{\alpha+\gamma}+\sum\nolimits_{k=1}^{K}a_kz^{\alpha_k+\gamma}\Big)^{-1},
\end{equation}
because its denominator is a CBF. Both transforms vanish as $z\to\infty$. Their inverse Laplace transforms $k_1$ and $k_2$ are therefore nonnegative, completely monotone functions without a Dirac mass at the origin. They are locally integrable. Indeed, for any $R>0$,
\begin{equation}\label{eq:mild-k-local}
\int_0^R k_i(s)\,\mathrm{d}s
\le e^R\int_0^\infty e^{-s}k_i(s)\,\mathrm{d}s
=e^R\widehat{k}_i(1)<\infty,
\qquad i=1,2.
\end{equation}
The convolution theorem and uniqueness of the Laplace transform applied to \eqref{eq:mild-phat-2} now give
\begin{equation}\label{eq:mild-volterra-ae}
p_j(t)+\lambda_j(k_1*p_j)(t)
=p_{0,j}+(k_2*f_j)(t)
\end{equation}
for almost every $t>0$. Since both sides are continuous, the identity holds for every $t\in[0,T]$. Thus the constructed function satisfies the modal Volterra equations.

Suppose that $p^{(1)}$ and $p^{(2)}$ are two continuous solutions of \eqref{eq:mild-volterra} with the same data, and set $u_j=p_j^{(1)}-p_j^{(2)}$. Then
\begin{equation}\label{eq:mild-uj}
u_j+\lambda_j k_1*u_j=0.
\end{equation}
Since $k_1\ge0$ and is locally integrable, for each $j$ there exists $h_j>0$ such that
\begin{equation}\label{eq:mild-hj}
\lambda_j\int_0^{h_j}k_1(s)\,ds<1.
\end{equation}
Consequently,
\begin{equation}\label{eq:mild-uj-est}
\|u_j\|_{C([0,h_j])}
\le\lambda_j\Big(\int_0^{h_j}k_1(s)\,\mathrm{d}s\Big)
\|u_j\|_{C([0,h_j])},
\end{equation}
which implies $u_j=0$ on $[0,h_j]$. Since the convolution is causal and $u_j$ vanishes on the preceding interval, the same argument applies successively on intervals of length $h_j$. Hence $u_j=0$ on $[0,T]$ for every $j$. Completeness of $\{\varphi_j\}_{j\ge1}$ gives $p^{(1)}=p^{(2)}$.

Applying \eqref{eq:mild-stability} to the difference of two solutions yields \eqref{eq:mild-cd}.

Finally, the coefficient of $\|p_0\|_{L^2(\Omega)}$ cannot be smaller than one, since $p(0)=p_0$. To prove the optimality of $C_T$, take $p_0=0$ and $f(t)=\varphi_1$, where $\|\varphi_1\|_{L^2(\Omega)}=1$. Then $p(t)=\left(\int_0^tq_1(s)\,ds\right)\varphi_1$. Since $q_1\ge0$, $\|p\|_{C([0,T];L^2(\Omega))}=\int_0^Tq_1(s)\mathrm{d}s=C_T$.
Thus $C_T$ is the optimal source coefficient for each fixed $T>0$. Furthermore, \eqref{eq:mild-Fbound} gives
\begin{equation}\label{eq:mild-opt-limit}
\lim_{T\to\infty}C_T
=\int_0^\infty q_1(s)\,\mathrm{d}s
=\frac1{\lambda_1}.
\end{equation}
Therefore $1/\lambda_1$ is the optimal source coefficient independent of $T$. \qed
\end{proof}

\subsection{\textbf{Proof of Theorem \ref{thm:theom1}}}\label{proof:Refularity1}

\begin{proof}
For part $(i)$, let $m\in\mathbb N_+$, $p_0\in L^2(\Omega)$, and $f\in W^{m,1}(0,T;L^2(\Omega))$. Starting from the mild solution representation \eqref{eq:integral}, we have $p(t)=\mathcal E(t)p_0+(\mathcal F*f)(t)$. Consequently,
\begin{displaymath}
\begin{aligned}
\|\partial_t^m p(t)\|_{L^2(\Omega)}
&\le\|\partial_t^m\mathcal E(t)p_0\|_{L^2(\Omega)}
+\|\partial_t^m(\mathcal F*f)(t)\|_{L^2(\Omega)}:=\widetilde{I}+\widetilde{II}.
\end{aligned}
\end{displaymath}
By Lemma~\ref{Lem:Lemma1}, the initial-data contribution satisfies $\widetilde{I}\le Ct^{-m}\|p_0\|_{L^2(\Omega)}$. We next estimate $\widetilde{II}$. Since $f\in W^{m,1}(0,T;L^2(\Omega))$, the traces $f^{(j)}(0)\in L^2(\Omega)$ exist for $0\le j\le m-1$. Applying the differentiation formula for causal convolutions, we obtain
\begin{displaymath}
\partial_t^m(\mathcal F*f)(t)
=\sum_{j=0}^{m-1}\partial_t^{m-1-j}\mathcal F(t)f^{(j)}(0)
+(\mathcal F*\partial_t^m f)(t).
\end{displaymath}
Therefore, using the estimates in Lemma~\ref{Lem:Lemma1}, we deduce
\begin{displaymath}
\begin{aligned}
\widetilde{II}
&\le\sum_{j=0}^{m-1}\|\partial_t^{m-1-j}\mathcal F(t)\|_\Psi
\|f^{(j)}(0)\|_{L^2(\Omega)}
+\int_0^t\|\mathcal F(t-s)\|_\Psi\|\partial_s^m f(s)\|_{L^2(\Omega)}\,\mathrm{d}s\\
&\le C\sum_{j=0}^{m-1}t^{\alpha+\gamma+j-m}\|f^{(j)}(0)\|_{L^2(\Omega)}
+C\int_0^t(t-s)^{\alpha+\gamma-1}\|\partial_s^m f(s)\|_{L^2(\Omega)}\,\mathrm{d}s.
\end{aligned}
\end{displaymath}
Combining the estimates for $\widetilde I$ and $\widetilde{II}$ yields
\begin{displaymath}
\begin{aligned}
\|\partial_t^m p(t)\|_{L^2(\Omega)}
\le C\Big(&t^{-m}\|p_0\|_{L^2(\Omega)}
+\sum\nolimits_{j=0}^{m-1}t^{\alpha+\gamma+j-m}\|f^{(j)}(0)\|_{L^2(\Omega)}\\
&+\int_0^t(t-s)^{\alpha+\gamma-1}\|\partial_s^m f(s)\|_{L^2(\Omega)}\,\mathrm{d}s\Big).
\end{aligned}
\end{displaymath}
For the estimate of $\mathcal A\partial_t^m p(t)$, we proceed similarly. Since $\mathcal A$ commutes with the solution operator $\mathcal F(t)$, we have $\mathcal A\partial_t^m p(t)=\mathcal A\partial_t^m\mathcal E(t)p_0+\partial_t^m\big((\mathcal A\mathcal F)*f\big)(t)$. Using the corresponding estimates in Lemmas~\ref{Lem:function} and \ref{Lem:Lemma1}, we obtain
\begin{displaymath}
\begin{aligned}
\|\mathcal A\partial_t^m p(t)\|_{L^2(\Omega)}
&\le Ct^{\beta-\alpha-\gamma-m}\|p_0\|_{L^2(\Omega)}
+\sum\nolimits_{j=0}^{m-1}\|\mathcal A\partial_t^{m-1-j}\mathcal F(t)\|_\Psi
\|f^{(j)}(0)\|_{L^2(\Omega)}\\
&\quad+\int_0^t\|\mathcal A\mathcal F(t-s)\|_\Psi
\|\partial_s^m f(s)\|_{L^2(\Omega)}\,\mathrm{d}s\\
&\le C\Big(t^{\beta-\alpha-\gamma-m}\|p_0\|_{L^2(\Omega)}
+\sum\nolimits_{j=0}^{m-1}t^{\beta+j-m}\|f^{(j)}(0)\|_{L^2(\Omega)}\\
&\qquad+\int_0^t(t-s)^{\beta-1}\|\partial_s^m f(s)\|_{L^2(\Omega)}\,\mathrm{d}s\Big).
\end{aligned}
\end{displaymath}
These estimates complete the proof of part $(i)$.

For part $(ii)$, suppose additionally that $p_0\in D(\mathcal A)=\dot H^2(\Omega)$. Since $\mathcal A$ is positive and self-adjoint, it commutes with the resolvent $(\eta(z)I+\mathcal A)^{-1}$ and hence with $\mathcal H(z)$. By \eqref{eq:denoteH}, we have $\mathcal H(z)p_0=\frac{\eta(z)}{z}\mathcal A^{-1}(\eta(z)I+\mathcal A)^{-1}\mathcal A p_0=\frac{p_0}{z}-\frac1z(\eta(z)I+\mathcal A)^{-1}\mathcal A p_0$, where we used the identity
$\eta(z)\mathcal A^{-1}(\eta(z)I+\mathcal A)^{-1}=\mathcal A^{-1}-(\eta(z)I+\mathcal A)^{-1}$. Since $m\ge1$, the inverse Laplace transform of $z^{m-1}p_0$ is supported at $t=0$ and therefore does not contribute for $t>0$. Thus,
\begin{displaymath}
\partial_t^m\mathcal E(t)p_0
=-\frac1{2\pi i}\int_{\Gamma_{\theta,\delta}}
e^{zt}z^{m-1}(\eta(z)I+\mathcal A)^{-1}\mathcal A p_0\,\mathrm{d}z.
\end{displaymath}
Applying Lemma~\ref{Lem:Resolvent} together with Lemma~\ref{Lem:function}, we obtain
\begin{displaymath}
\|\partial_t^m\mathcal E(t)p_0\|_{L^2(\Omega)}
\le Ct^{\alpha+\gamma-\beta-m}\|p_0\|_{\dot H^2(\Omega)}.
\end{displaymath}
The source contribution is the same as in part $(i)$. Consequently,
\begin{displaymath}
\begin{aligned}
\|\partial_t^m p(t)\|_{L^2(\Omega)}
\le C\Big(&t^{\alpha+\gamma-\beta-m}\|p_0\|_{\dot H^2(\Omega)}
+\sum\nolimits_{j=0}^{m-1}t^{\alpha+\gamma+j-m}\|f^{(j)}(0)\|_{L^2(\Omega)}\\
&+\int_0^t(t-s)^{\alpha+\gamma-1}\|\partial_s^m f(s)\|_{L^2(\Omega)}\,\mathrm{d}s\Big).
\end{aligned}
\end{displaymath}
Similarly, applying $\mathcal A$ to the preceding resolvent representation and using the same operator estimates, we obtain
\begin{displaymath}
\|\mathcal A\partial_t^m\mathcal E(t)p_0\|_{L^2(\Omega)}
\le Ct^{-m}\|p_0\|_{\dot H^2(\Omega)}.
\end{displaymath}
Combining this bound with the source estimate established in part $(i)$ yields
\begin{displaymath}
\begin{aligned}
\|\mathcal A\partial_t^m p(t)\|_{L^2(\Omega)}
\le C\Big(&t^{-m}\|p_0\|_{\dot H^2(\Omega)}
+\sum\nolimits_{j=0}^{m-1}t^{\beta+j-m}\|f^{(j)}(0)\|_{L^2(\Omega)}\\
&+\int_0^t(t-s)^{\beta-1}\|\partial_s^m f(s)\|_{L^2(\Omega)}\,\mathrm{d}s\Big).
\end{aligned}
\end{displaymath}
The above estimates hold for almost every $t\in(0,T)$ and complete the proof of part $(ii)$. \qed
\end{proof}

\subsection{\textbf{Proof of Lemma \ref{thm:prop2}}}\label{proofa}
\begin{proof}
Let $z(\phi)\in\mathcal{N}_e\subseteq\Sigma_\theta$. For $|z|\geq1$ and $0<\alpha+\gamma\leq1$, it follows from Lemma \ref{Lem:function} and the definition of $\widehat{p}(r,z)$ in \eqref{eq:integrand} that $\|\mathcal{H}(z)\|_\Psi\leq C|z|^{-1}$ and $\|g(z)\mathcal{H}(z)\|_\Psi\leq C|z|^{-(\alpha+\gamma)}$. For $0<|z|\leq1$, the positivity of $\mathcal{A}$ and the sectorial property of $\eta(z)$ imply $\|\mathcal{H}(z)\|_\Psi\leq C|z|^{-1}$. Moreover, since $g(z)\mathcal{H}(z)=\mathcal{D}(z)^{-1}(\eta(z)I+\mathcal{A})^{-1}$, and $\mathcal{D}(z)$ is bounded away from zero on the truncated sector $|z|\leq1$, we have $\|g(z)\mathcal{H}(z)\|_\Psi\leq C\leq C|z|^{-(\alpha+\gamma)}$. Consequently, for every $z\in\mathcal{N}_e$, we have
\begin{align}\label{ieq:uhatinhomo}
\big\|\widehat{p}(r,z)\big\|_{L^{2}(\Omega)}
\leq C\left(|z|^{-1}\|p_0\|_{L^{2}(\Omega)}
      +|z|^{-(\alpha+\gamma)}\|\widehat{f}(z)\|_{(\mathcal{N}_e;L^2(\Omega))}\right).
\end{align}

Recall \eqref{eq:integralsolve}. From \eqref{eq:hyperboliccontour}, we have $z'(\phi)=i\mu\cos(i\phi-\widetilde{\alpha})$. Let $\phi=x+iy\in \mathcal{S}$ with $0<y<d$. Then,
\begin{equation}\label{eq:vxr}
u(t,r,x+iy)=\frac{\mu}{2\pi}e^{\mu\big(1-\sin(\widetilde{\alpha}+y-ix)\big)t}
\cos(\widetilde{\alpha}+y-ix)\widehat{p}\big(r,\mu(1-\sin(\widetilde{\alpha}+y-ix))\big).
\end{equation}
Let $l':=\widetilde{\alpha}+y$, so that $l'\in(\widetilde{\alpha}, \widetilde{\alpha}+d)\subset(0,\pi/2-\delta')$. Taking the $L^{2}$-norm of both sides yields
\begin{equation}\label{eq:estvxr}
\|u(t,r,x+iy)\|_{L^{2}(\Omega)}
\leq \frac{\mu}{2\pi}e^{\mu t(1-\sin l'\cosh x)}\big\|\cos(l'-ix)\widehat{p}\big(r,\mu(1-\sin(l'-ix))\big)\big\|_{L^{2}(\Omega)}.
\end{equation}
Using the earlier estimate \eqref{ieq:uhatinhomo} and the identity \eqref{eq:estvxr}, we obtain
\begin{displaymath}\begin{aligned}
&\big\|\cos(l'-ix)\widehat{p}(r,\mu(1-\sin (l'-ix)))\big\|_{L^{2}(\Omega)}\\
&\leq C\frac{|\cos(l'-ix)|}{|\mu(1-\sin(l'-ix))|}\|p_0\|_{L^2(\Omega)}
+C\frac{|\cos(l'-ix)|}{|\mu(1-\sin(l'-ix))|^{\alpha+\gamma}}\|\widehat{f}(z)\|_{(\mathcal{N}_e;L^2(\Omega))}.
\end{aligned}\end{displaymath}

Note that $\frac{|\cos(l'-ix)|}{|1-\sin(l'-ix)|}=\sqrt{\frac{\cosh x+\sin l'}{\cosh x-\sin l'}}\leq \sqrt{\frac{1+\sin l'}{1-\sin l'}}:=\mathcal{P}(l')$, and $|1-\sin(l'-ix)|=\cosh x-\sin l'$. Hence, 
\begin{align*}
\frac{|\cos(l'-ix)|}{|\mu(1-\sin(l'-ix))|^{\alpha+\gamma}}
=&\frac{1}{|\mu(1-\sin(l'-ix))|^{\alpha+\gamma-1}}\frac{|\cos(l'-ix)|}{|\mu(1-\sin(l'-ix))|}\\
\leq&\frac{\mathcal{P}(l')}{\mu^{\alpha+\gamma}(\cosh x-\sin l')^{\alpha+\gamma-1}}.
\end{align*}
Since $\sin\widetilde{\alpha}<\sin l'<\sin(\widetilde{\alpha}+d)$, and $0\leq1-\alpha-\gamma<1$, we conclude that
\begin{equation}
\begin{aligned}\label{eq:vP}
\|u(t,\cdot,x+iy)\|_{L^2(\Omega)}
&\leq\frac{C\mathcal P(\widetilde{\alpha}+d)}
{2\pi}e^{\mu t_0\Lambda}e^{-\mu t_0\sin l'\cosh x}\\
&\quad\times\Big(
\|p_0\|_{L^2(\Omega)}+\mu^{1-\alpha-\gamma}(\cosh x-\sin l')^{1-\alpha-\gamma}
\|\widehat{f}\|_{(\mathcal{N}_e;L^2(\Omega))}\Big)\\
&\leq
\frac{C\mathcal{P}(\widetilde{\alpha}+d)}{2\pi}e^{\mu t_0\Lambda}e^{-\mu t_0\sin(\widetilde{\alpha})\cosh x}\\
&\quad\times
\Big(\|p_0\|_{L^2(\Omega)}+\mu^{1-\alpha-\gamma}(\cosh x)^{1-\alpha-\gamma}\|\widehat f\|_{(\mathcal{N}_e;L^2(\Omega))}\Big).
\end{aligned}
\end{equation}

A similar argument applies to the lower part of $\mathcal{S}$. For $\phi=x-iy\in \mathcal{S}$ with $-d<-y\leq0$, and $\sin(\widetilde{\alpha}-d)<\sin\widetilde{\alpha}<\sin(\widetilde{\alpha}+d)$, we obtain
\begin{equation}\label{eq:vN}
\begin{aligned}
\|u(t,\cdot,x-iy)\|_{L^2(\Omega)}
&\leq\frac{C\mathcal{P}(\widetilde{\alpha}+d)}{2\pi}e^{\mu t_0\Lambda}e^{-\mu t_0\sin(\widetilde{\alpha}-d)\cosh x}\\
&\quad\times
\Big(\|p_0\|_{L^2(\Omega)}+\mu^{1-\alpha-\gamma}(\cosh x)^{1-\alpha-\gamma}\|\widehat f\|_{(\mathcal{N}_e;L^2(\Omega))}\Big).
\end{aligned}
\end{equation}

Therefore, for any $\phi=x\pm iy\in \mathcal{S}$, the aforementioned estimates imply
%, in conjunction with \eqref{eq:vN}, leads to the following result:
\begin{displaymath}
\begin{aligned}
\|u(t,\cdot,\phi)\|_{L^2(\Omega)}
&\leq
\frac{C\mathcal P(\widetilde{\alpha}+d)}{2\pi}e^{\mu t_0\Lambda}e^{-\mu t_0\sin(\widetilde{\alpha}-d)\cosh x}\\
&\quad\times
\Big(\|p_0\|_{L^2(\Omega)}+\mu^{1-\alpha-\gamma}(\cosh x)^{1-\alpha-\gamma}\|\widehat f\|_{(\mathcal N_e;L^2(\Omega))}\Big).
\end{aligned}
\end{displaymath}

Let $\rho:=1-\alpha-\gamma\in[0,1)$ and $\vartheta:=\mu t_0\sin(\widetilde{\alpha}-d)>0$. For $\rho>0$ and any $\nu\in(0,1)$, following \cite{Fernandez2004}, using $s^{\rho}e^{-as}\le \Big(\frac{\rho}{ae}\Big)^{\rho}$ $(s\ge0,\ a>0)$, for $s=\cosh x$ and $a=\nu\vartheta$, we obtain
\begin{align*}
&\|u(t,\cdot,\phi)\|_{L^2(\Omega)}\\
&\leq \frac{C\mathcal P(\widetilde{\alpha}+d)}
{2\pi}e^{\mu t_0\Lambda}
\Big(\|p_0\|_{L^2(\Omega)}+\mu^\rho(\cosh x)^\rho\|\widehat{f}\|_{(\mathcal{N}_e;L^2(\Omega))}
\Big)e^{-\vartheta\cosh x}\\
&\leq
\frac{C\mathcal{P}(\widetilde{\alpha}+d)}{2\pi}e^{\mu t_0\Lambda}
\Big(\|p_0\|_{L^2(\Omega)}+\mu^\rho\left(\frac{\rho}{\nu\vartheta e}\right)^\rho
\|\widehat{f}\|_{(\mathcal{N}_e;L^2(\Omega))}\Big)e^{-(1-\nu)\vartheta\cosh x}\\
&=
\frac{C\mathcal{P}(\widetilde{\alpha}+d)}
{2\pi}e^{\mu t_0\Lambda}
\Big(\|p_0\|_{L^2(\Omega)}+t_0^{-\rho}\left(\frac{\rho}{\nu\sin(\widetilde{\alpha}-d)e}
\right)^\rho\|\widehat{f}\|_{(\mathcal{N}_e;L^2(\Omega))}\Big)e^{-(1-\nu)\vartheta\cosh x}.
\end{align*}
Define $Q:=\frac{1}{2\pi}\left(\frac{\rho}{\nu\sin(\widetilde{\alpha}-d)e}\right)^\rho\mathcal P(\widetilde{\alpha}+d)$.
Since $0<\nu\sin(\widetilde{\alpha}-d)<1$ and $0<\rho<1$, we have $\left(\frac{\rho}{\nu\sin(\widetilde{\alpha}-d)e}\right)^\rho\geq\left(\frac{\rho}{e}\right)^\rho\geq e^{-1}$. Thus, the coefficient of the initial-data term can be bounded by $CQ$. Consequently,
\begin{align*}
\|u(t,\cdot,\phi)\|_{L^2(\Omega)}
&\leq CQe^{\mu t_0\Lambda}
\Big(\|p_0\|_{L^2(\Omega)}+t_0^{\alpha+\gamma-1}\|\widehat f\|_{(\mathcal N_e;L^2(\Omega))}\Big)\\
&\quad\times e^{-(1-\nu)\mu t_0 \sin(\widetilde{\alpha}-d)\cosh x}.
\end{align*}
The case $\rho=0$ follows directly by replacing the power factor by $1$. This completes the proof. \qed
\end{proof}


%% \section{Example of another appendix section} \label{secA2}%

%%%%%%%%%%%%%%%%%%%%%%%%%% BACK MATTERS %%%%%%%%%%%%%%%%%%%%%%%%%%%%

\begin{acknowledgements}
 % Acknowledgments are not compulsory. Where included they should be brief.
 % Grant or contribution numbers may be acknowledged, or help by colleagues. Example:
The author thanks the referees for their valuable suggestions. This work was funded by the Peking University Boya Postdoctoral Fellowship. The author is grateful to Prof. Weihua Deng for his academic guidance, to Prof. Heping Ma for the MATLAB code of the 2D Chebyshev-Legendre transform, to Prof. Can Li for his guidance on spectral methods, and to Prof. Tiejun Li for his generous support and the development opportunities he provided.
\end{acknowledgements}

 \section*{\small 
 Conflict of interest} %%%%%%%%%%%%%%%%%%%%

 {\small
 The author declares no conflict of interest.}

%%%%%%%%%%%%%%%%%% REFERENCES: %%%%%%%%%%%%%%%%%%%%%%%%%%%%%%%%%%%%%%%%%%%%
%% BibTeX users: please use \bibliographystyle{spmpsci} %% for math. and phys. sci.
%% Non-BibTeX users: please use the model as below !!! %%

%%%% for FCAA - pls. include directly the Refs items here ! %%%
%%%% following STRICTLY the models below %%%%%
%%%% and ARRANGE the items in ALPHABETIC ORDER for authors' family names !!!

%\bibliographystyle{spmpsci}
%\bibliography{references}   

\bigskip  %%%%%%%%%%%%%%%%%%%%%%%%%%%%%%%%%

\small %%%
\noindent
{\bf Publisher's Note}
Springer Nature remains neutral with regard to jurisdictional claims in published maps and institutional affiliations.

\end{document}